%% file: IMRN.tex
\documentclass{article}
\usepackage[top=1in, bottom=1in, left=1in, right=1in]{geometry}

\usepackage[utf8]{inputenc}
\usepackage[T1]{fontenc}
\usepackage[ngerman,english]{babel}
\usepackage{lmodern}
\usepackage{titling}

\usepackage{amsmath}
\usepackage{amsthm}
\usepackage{amsfonts}
\usepackage{amssymb}
\usepackage{amscd}
\usepackage{amsbsy}

\usepackage{pdfpages}
\usepackage{fancyhdr}
\usepackage{geometry}
\usepackage{setspace}
\usepackage{xcolor}
\usepackage{pifont}

\usepackage{graphicx}
\usepackage{tabularx}
\usepackage{epsfig}
\usepackage[all]{xy}
\usepackage{tikz}
\usetikzlibrary{matrix}

\usepackage{fixmath}
\usepackage{appendix}

\usepackage{enumerate}
\usepackage[sort, numbers]{natbib}
\usepackage{bibentry}

\newtheorem{theorem}{Theorem}[section]

\newtheorem{lemma}[theorem]{Lemma}
\newtheorem{proposition}[theorem]{Proposition}

\newtheorem{corollary}[theorem]{Corollary}

\theoremstyle{definition}
\newtheorem{definition}[theorem]{Definition}

\newtheorem*{remark}{Remark}
\newtheorem*{remarks}{Remarks}

\numberwithin{equation}{section}

\newcommand{\N}{{\mathbb N}}
\newcommand{\Z}{{\mathbb Z}}

\newcommand{\R}{{\mathbb R}}
\newcommand{\C}{{\mathbb C}}
\newcommand{\D}{{\mathbb D}}
\newcommand{\T}{{\mathbb T}}
\newcommand{\A}{{\mathbb A}}

\newcommand{\E}{{\mathsf E}}

\newcommand{\mj}{j}
\newcommand{\mjH}{j}

\newcommand{\supp}{{\operatorname{supp}}}
\newcommand{\Res}{\operatorname{Res}}

\newcommand{\tr}{\operatorname{tr}}

\newcommand{\lin}{\mathrm{span}}

\newcommand{\diag}{\operatorname{diag}}

\newcommand{\cB}{{M}}

\newcommand{\cD}{{\mathcal{D}}}

\newcommand{\cF}{{\mathcal{F}}}

\newcommand{\cH}{{\mathcal{H}}}

\newcommand{\cK}{{\mathcal{K}}}
\newcommand{\cL}{{\mathcal{L}}}

\newcommand{\cM}{{T}}

\newcommand{\cS}{{\mathcal{S}}}
\newcommand{\cT}{{\mathcal{T}}}

\newcommand{\cU}{{U}}

\newcommand{\fA}{{\mathfrak{A}}}
\newcommand{\fB}{{\mathfrak{B}}}

\newcommand{\fb}{{\mathfrak{b}}}

\newcommand{\fw}{{\mathfrak{w}}}
\newcommand{\fz}{{\mathfrak{z}}}

\newcommand{\G}{{\Gamma}}
\newcommand{\e}{{\varepsilon}}

\newcommand{\z}{\zeta}
\newcommand{\g}{\gamma}

\renewcommand{\l}{\lambda}
\renewcommand{\a}{\alpha}

\renewcommand{\chi}{y}
\renewcommand{\b}{\beta}
\renewcommand{\d}{\delta}

\renewcommand{\L}{{\Lambda}}
\newcommand{\dd}{\mathrm{d}}
\newcommand{\spa}{\operatorname{span}}
\renewcommand{\Re}{\operatorname{Re}}
\renewcommand{\Im}{\operatorname{Im}}

\newcommand{\rz}{z^*}

\renewcommand{\i}{\infty}
\newcommand{\vz}{\vec{\vspace{1cm}z}}
\newcommand{\minus}{\scalebox{0.5}[1.0]{\( -\)}\hspace{-0.04cm}}

\newcommand{\apa}{\mathbf a}
\newcommand{\apr}{\mathbold\rho}
\newcommand{\apC}{\mathcal C}
\newcommand{\mv}{\text{\rm{CMV}}}
\newcommand{\MV}{\text{\rm{MCMV}}}
\newcommand{\per}{\text{\rm{per}}}
\newcommand{\ess}{\text{\rm{ess}}}
\newcommand{\DD}{D_{0}}

\usepackage{tikz}
\usetikzlibrary{matrix}

\title{Finite-gap CMV matrices: \\
Periodic Coordinates and a Magic Formula}

\author{Jacob~S.~Christiansen\thanks{Centre for Mathematical Sciences, Lund University, Box 118, 22100 Lund, Sweden. \newline
E-mail: stordal@maths.lth.se}, Benjamin~Eichinger\thanks{Departments of Mathematics, Rice University MS-136, Box 1892,
Houston, TX 77251-1892, USA. \newline E-mail: benjamin.eichinger@rice.edu}, and Tom~VandenBoom\thanks{Department of Mathematics,
Yale University, 10 Hillhouse Ave, New Haven, CT 06511, USA. \newline E-mail: thomas.vandenboom@yale.edu}
}

\thanksmarkseries{arabic}

\begin{document}

\maketitle

\begin{abstract}
	We prove a bijective unitary correspondence between 1) the isospectral torus of almost-periodic, absolutely continuous CMV matrices having fixed finite-gap spectrum $\E$ and 2) special periodic block-CMV matrices satisfying a Magic Formula.  This latter class arises as $\E$-dependent operator M\"obius transforms of certain generating CMV matrices which are periodic up to a rotational phase; for this reason we call them ``MCMV''.  Such matrices are related to a choice of orthogonal rational functions on the unit circle, and their correspondence to the isospectral torus follows from a functional model in analog to that of GMP matrices.  As a corollary of our construction we resolve a conjecture of Simon; namely, that Caratheodory functions associated to such CMV matrices arise as quadratic irrationalities.
\end{abstract}

\input{intro-28-09}

\input{orf-28-08}
\input{funcModel-28-09}
\input{DST-28-09}
\input{proofsMain}

\appendix
\input{appendix}

\subsection*{Funding}
J.\;S.\;C. is supported by the Swedish Research Council (VR) Grant No.\;2018-03500 and in part by the Project Grant DFF-4181-00502
from the Danish Council for Independent Research.
The research of B.\;E. is supported by the Austrian Science Fund FWF, Project No.\;J 4138-N32.
T.\;V. was partially supported by an AMS Simons travel grant 2018--2020.

\subsection*{Acknowledgements}
It is a pleasure to thank Peter Yuditskii for very helpful discussions. 

\bibliographystyle{amsplain}
\bibliography{lit}

\end{document}

%% file: intro-28-09.tex
\section{Introduction and main results}

This paper studies two equivalent notions of interest, one in the spectral theory of certain unitary operators and the other in the theory of analytic functions mapping the unit disk into its closure.  This connection between CMV matrices and Schur functions is no mystery to experts in either field.  
We introduce here our results from both perspectives because we find them striking and attractive, as well as not immediately equivalent.

We begin in the context of the title:

\subsection{Finite-gap CMV and MCMV matrices}

Fix a sequence $\{a_k\}_{k \in \Z} \in \D^\Z$ of numbers in the unit disk $\D$, and define a family of $2 \times 2$ unitary matrices by
\begin{align*}
\Theta_k := \begin{bmatrix}
\overline{a_k} & \rho_k \\
\rho_k & -a_k
\end{bmatrix}, \quad \rho_k = \sqrt{1-|a_k|^2}.
\end{align*}
Letting $\Theta_k$ act on the span of $\{\delta_k, \delta_{k+1}\}$ in $\ell^2(\Z)$, define
\begin{align*}
L = \bigoplus_{l \in \Z} \Theta_{2l}, \quad M = \bigoplus_{l \in \Z} \Theta_{2l+1}.
\end{align*}
The (whole-line) CMV matrix associated to the sequence $\{a_k\}_{k \in \Z}$ is the unitary operator
\begin{align*}
C &:= C(\{a_k\}) = LM.
\end{align*}
Represented in the basis $\{\delta_k\}_{k \in \Z}$, $C$ is a five-diagonal matrix with repeating $2 \times 4$ block structure:
\begin{align*}
C=\begin{bmatrix}
		\ddots &\ddots &\ddots & & & \\
		\rho_{2l-1}\overline{a_{2l}} & -\overline{a_{2l}}a_{2l-1} & \overline{a_{2l+1}}\rho_{2l}&\rho_{2l}\rho_{2l+1}& & \\
		\rho_{2l}\rho_{2l-1} & -\rho_{2l}a_{2l-1} & -\overline{a_{2l+1}}a_{2l}&-\rho_{2l+1}a_{2l}& & \\
		& &\rho_{2l+1}\overline{a_{2l+2}} & -\overline{a_{2l+2}}a_{2l+1} & \overline{a_{2l+3}}\rho_{2l+2}&\rho_{2l+2}\rho_{2l+3} \\
		&&\rho_{2l+2}\rho_{2l+1} & -\rho_{2l+2}a_{2l+1} & -\overline{a_{2l+3}}a_{2l+2}&-\rho_{2l+3}a_{2l+2}& & \\
		& & &\ddots &\ddots&\ddots
	\end{bmatrix}.
\end{align*}
We will study those CMV matrices with almost-periodic coefficients $\{a_k\}_{k\in\Z}$ having full a.c. spectrum consisting of a fixed finite union of non-degenerate closed circular arcs  $\E \subset \partial\D$.  We denote the set of all such CMV matrices by $\cT_{\mv}(\E)$.  Topologically, $\cT_{\mv}(\E)$ is a torus of dimension equal to the number of disjoint arcs in $\E$.

CMV matrices are natural objects of interest in the context of orthogonal polynomials on the unit circle (see \cite{SimonOPUC1,SimonOPUC2}).  This in part relies on the interesting fact that half-line CMV matrices $C_+$, formed by setting $a_{-1} = -1$ above and restricting to $\ell^2(\N)$, are universal for cyclic unitary operators in the sense that, for any probability measure $\nu$ with infinite support on the unit circle $\partial\D$, multiplication by the independent variable in $L^2(\dd\nu)$ is unitarily equivalent to some half-line CMV matrix $C_+$.  This discovery was made surprisingly recently by Cantero, Moral, and Vel\'azquez \cite{CanMoVel03} by considering the basis of $L^2(\dd\nu)$ generated by orthonormalizing $\{1, z^{-1}, z, z^{-2}, z^2, \ldots \}$.  In comparison, it has long been known (see, e.g., \cite{MHS32}) that tridiagonal Jacobi matrices are universal models for selfadjoint operators with a cyclic vector.  CMV matrices are also important, e.g., in the theory of random matrices and integrable systems \cite{KiNe04,Ne06}, and for quantum walks \cite{CGMV10}; see also \cite{KiNe07,Gol08}.

The CMV basis is far from the only generating set for $L^2(\dd\nu)$.  Letting $b_w$ be the elementary Blaschke factor for $\D$ vanishing at $w$, i.e.
\begin{align}
b_w(z) := \frac{z - w}{1-\overline{w}z}
\end{align}
and denoting by $w^* = 1/\overline{w}$ reflection with respect to $\partial\D$, one has that $b_w(z^*) = b_{w}(z)^*=c_w\overline{b_{w^*}(z)}$, for some explicit unimodular constant $c_w$.   We can then suggestively rewrite the CMV basis above as instead being generated by orthonormalizing the sequence $\{1, b_\infty, b_0, b_\infty^2, b_0^2, \ldots\}$.  For a fixed sequence of points $\{z_k\}_{k \in \N} \in \D^\N$ with modulus bounded uniformly away from $1$, if we denote by $\{B_k\}$ and $\{B_k^*\}$ the families of Blaschke products
\begin{align}
B_0(z) = 1, \quad B_k(z) = \prod_{j=1}^k b_{z_j}(z), \quad B_k^*(z) = \overline{B_k(z^*)} = \prod_{j=1}^k  b_{z_j}(z)^{-1},
\end{align}
one could just as well have spanned $L^2(\dd\nu)$ by the sequence $\{B_0, B_1^*, B_1, B_2^*, B_2, \ldots\}$. In \cite{Ve08}, Vel\'azquez showed that the structure of multiplication by the independent variable $z$ in $L^2(\dd\nu)$ in the orthonormalization of this new generating set is related to CMV matrices via an operator M\"obius transform; specifically, denoting by $D_+ := \diag_{\N}\{0, z_1, z_1, z_2, z_2, \ldots\}$, he showed that multiplication by $z$ in $L^2(\dd\nu)$ is unitarily equivalent to the operator
\begin{align}\label{eq:VelazquezMCMV}
b_{\minus D_+}(C_+) := \eta_{D_+}(1+C_+D_+^*)^{-1}(D_+ + C_+)\eta_{D_+}^{-1}, \quad \eta_{D_+} = \sqrt{1-D_+D_+^*}
\end{align}
for some half-line CMV matrix $C_+$.  This theorem suggests we should study this new class of operator M\"obius transforms of CMV matrices more closely:
\begin{definition}[MCMV matrices]\label{def:MCMV}
Fix $n \geq 1$ and let $\vz=\{z_k\}_{k=0}^{n-1}\in\D^{n}$ with $z_0 = 0$, $\{a_k\}_{k \in \Z} \in \D^\Z$, and $\vartheta \in \R/2\pi\Z$.  Denote by $\DD$ the $2n$-periodic diagonal matrix
\begin{align}
\label{eq:perdiag}
\DD:=\DD(\vz)=\diag_{\Z}\{ \ldots, z_{n-2}, z_{n-1}, z_{n-1}, z_0 \,|\, z_0, z_1, z_1, z_2, \ldots \},
\end{align}
let $\Lambda_k(\vartheta)$ be the $2n \times 2n$ diagonal matrix $\Lambda_k(\vartheta):=\operatorname{diag}_{2n}\{e^{ik\vartheta},e^{-ik\vartheta},\dots, e^{ik\vartheta},e^{-ik\vartheta}\}$, and define
	\begin{align}
		\Lambda(\vartheta):=\bigoplus_{k \in \Z}\Lambda_k(\vartheta).
	\end{align}
With this notation, the (whole-line) MCMV matrix for $\vz$, $\{a_k\}$, and $\vartheta$ and  is defined by
	\begin{align}
		A:=A(\{a_k\},\vartheta; \vz)=\Lambda(\vartheta)^* b_{\minus \DD}(C)\Lambda(\vartheta),
	\end{align}
where $C = C(\{a_k\})$ is the CMV matrix associated to $\{a_k\}_{k\in\Z}$.
\end{definition}
We shall explain the role of $\Lambda(\vartheta)$ momentarily and be even more specific in Section \ref{sec:3}.  In short, this diagonal matrix enables us to change from periodicity up to a rotational phase to pure periodicity.

Like CMV matrices, an MCMV matrix $A$ is again band-structured. If we split $A$ into $2n\times 2n$ blocks $\mathbf{A_{i j}}$, then $\mathbf{A_{i j}}=\bf{0}$ if $|i-j|>1$. Moreover, the off-diagonal blocks are of the form $\mathbf{A_{i,i-1}}=\mathbf{v^i}\delta_{2n-1}^\intercal$ and $\mathbf{A_{i,i+1}}=\mathbf{u^i}\delta_0^\intercal$ for some explicit vectors $\mathbf{u^i},\mathbf{v^i}\in\C^{2n}$; cf. Lemma \ref{thm:structureA} and the figure below:
\begin{align}
\label{eq:Astructure}
A&=
\begin{tikzpicture}[baseline=-0.5ex]
  \matrix (m) [matrix of math nodes, nodes in empty cells, column sep=0mm, row sep=0mm, nodes={rectangle, 
                minimum size=1.2em, text depth=0.25ex,
                inner sep=0pt, outer sep=0pt,
                fill opacity=0.5, text opacity=1,
                anchor=center},
	left delimiter={[},right delimiter={]},ampersand replacement=\&] {
    \ddots \&   \& 			    \& 			\ddots    \&  \&          	      \&  \ddots\&   	\&			   \&\&\\
    		 \&   \& 			  	\& 		\&          \&       		      \&  \&		\&			   \&\&\\
 \mathbf{0} \; \; \&   \& \mathbf{0} \; \; \& \mathbf{v^i}   	\&          \&\mathbf{A_{ii}}\&  \&\mathbf{u^i}	    \& \; \; \mathbf{0} \&\& \; \; \mathbf{0}\\
   			 \&   \& 			    \& 	    \&          \&       		      \&  \&		\&			   \&\&\\
   			 \&   \& 			    \&               \& \ddots \&       		      \&  \&		\ddots\&			   \&\&\ddots\\
  } ;
 \draw (m-4-2.south west) rectangle (m-2-4.north east);
\draw (m-4-5.south west) rectangle (m-2-7.north east);
\draw (m-4-8.south west) rectangle (m-2-10.north east);
\end{tikzpicture}.
\end{align}
Furthermore, since operator M\"obius transforms preserve unitarity, MCMV matrices are likewise unitary operators.  Thus MCMV matrices can be viewed as being ``block-CMV''.  This special structure does not hold for arbitrary operator M\"obius transforms of CMV matrices; it follows in our case from $D_0$ having periodically repeated zero entries.


We denote the class of all MCMV matrices associated to $\vz \in \D^n$ by
	\begin{align}
	\A(\vz) := \bigl\{A(\{a_k\},\vartheta; \vz) : \{a_k\} \in \D^\Z, \; \vartheta \in \R/2\pi\Z \bigr\}
	\end{align}
and give special consideration to the subset $\A_{\per}(\vz) \subset \A(\vz)$ of periodic operators, i.e.
\begin{align}
\A_{\per}(\vz) := \bigl\{ A \in \A(\vz) : S^{2n}A = AS^{2n} \bigr\},
\end{align}
where, as usual, $S$ is the right shift operator.

Notice that the usual CMV matrices are simply the special case where $z_k = 0$ for all $k$.  This realization gives rise to a natural (if somewhat ill-posed) question: is there a ``best'' generating set of Blaschke products for a given measure $\nu$ on $\partial\D$? In the context of whole-line CMV matrices $\apC \in \cT_{\mv}(\E)$
, we offer an affirmative answer:

\begin{theorem}[Periodic Coordinates for finite-gap CMV matrices]
\label{t:periodiccoords}
Let $\E \subset \partial\D$ be a disjoint union of $g+1$ non-degenerate closed circular arcs.  There exists a sequence $\vz_{\E} := \{z_k\}_{k=0}^g \in \D^{g+1}$ of points depending only on $\E$ such that, denoting
\begin{align}
\cT_{\MV}(\E) := \{A \in \A_{\per}(\vz_{\E}) : \sigma(A) = \E\},
\end{align}
there is a unitary bijection between $\cT_{\mv}(\E)$ and $\cT_{\MV}(\E)$; i.e.
\begin{align}  \label{eq:bij}
\cT_{\mv}(\E) \simeq \cT_{\MV}(\E).
\end{align}
In particular, for an almost-periodic CMV matrix $\apC$ with purely absolutely continuous spectrum $\E$, there exists an associated CMV matrix $C = C(\{a_k\})$ with phase-periodic coefficients
\begin{align}\label{eq:phasePeriodic}
a_{k+2(g+1)} = e^{-2i\vartheta}a_k, \quad k \in \Z
\end{align}
such that $\apC$ is unitarily equivalent to the periodic MCMV matrix $A(\{a_k\}, \vartheta; \vz_{\E}) \in \A_{\per}(\vz_\E)$ and the spectral measures of the one-sided restrictions $\apC_+$ and $A_+$ coincide; cf. \eqref{def:A_+}.
\end{theorem}
\begin{remarks}
\begin{itemize}
\item[(i)] 
    The above theorem shows that a periodic MCMV matrix is naturally related to two different CMV matrices: the almost-periodic CMV matrix in $\cT_{\mv}(\E)$ and the ``generating'' phase-periodic CMV matrix. Throughout, we will denote the former by $\apC$ with parameters $\{\apa_k\}$ (resp.\;$\{\apr_k\}$) and the latter by $C$ with parameters $\{a_k\}$ (resp.\;$\{\rho_k\}$).
\item[(ii)] As a consequence of \eqref{eq:phasePeriodic}, the operator $b_{\minus \DD}(C)$ is periodic up to a phase. By conjugating it with $\Lambda(\vartheta)$ -- and this is the main purpose of introducing such a diagonal matrix -- we get that $A(\{a_k\}, \vartheta; \vz_{\E})$ becomes periodic in the standard sense. This is particularly important in view of Theorem \ref{t:magicformula} below, since by Naiman's lemma \cite{Nai62} an operator satisfying the right-hand side of \eqref{eq:magicformula} is necessarily periodic.
\end{itemize}
\end{remarks}
To summarize, we prove that for CMV matrices in $\cT_{\mv}(\E)$, there is a ``best'' basis of the $L^2$-space related to its half-line truncation $\apC_+$ in the sense that the associated MCMV matrices are periodic.  In fact, we can completely characterize this basis via the vector $\vz_\E$, as well as the MCMV matrices $A$ in $\cT_{\MV}(\E)$, in terms of a ``Magic Formula'' analogous to that of Damanik, Killip, and Simon \cite{DaKiSi10}.  Thus MCMV matrices are a unitary analog of GMP matrices \cite{Eich16, Yud18}.  To formulate our result, we first need to introduce a certain rational function called the discriminant.
%


Fix a finite-gap set $\E \subset \partial\D$, that is, a finite disjoint union of non-degenerate closed circular arcs.  Let us refer to the arc-components as bands and to the connected components of $\partial\D\setminus \E$ as gaps. We assume that the number of gaps (and respectively bands) is $g+1$.  For any point $z_0$ in the domain $\overline{\C} \setminus \E$, there exists an Ahlfors function $w_{z_0}$ which maximizes the modulus of the derivative at $z_0$ (or, in the case $z_0 = \infty$, maximizes $\lim_{z \to \infty}|zw_\infty(z)|$) among all analytic functions on $\overline{\C} \setminus \E$ with modulus bounded by $1$; cf. \cite{Ah47,Gar49}. This extremal property defines $w_{z_0}$ uniquely up to a unimodular multiplier.  In the right normalization, these Ahlfors functions for $\E \subset \partial\D$ have the symmetry property
\begin{align}\label{eq:symmetryAhlfors}
w_{z_0^*}(z^*) = \overline{w_{z_0}(z)};
\end{align}
in particular, the zeros of $w_{\infty}$ can be obtained by reflecting the zeros of $w_0$ with respect to $\partial\D$.  In terms of these functions, we can define a special function, which we call the generalized discriminant, related to the set $\E$:
\begin{definition}[Generalized discriminant]
For a finite union of non-degenerate closed circular arcs $\E \subset \partial\D$, the generalized discriminant is defined by
\begin{align}
\Delta_\E := \frac{1}{w_0w_\infty} + w_0w_\infty.
\end{align}
\end{definition}
By \eqref{eq:symmetryAhlfors}, we see that $\Delta_\E$ is real-valued on $\partial\D$; since $|w_{z_0}(z)| = 1$ for $z \in \E$ in the sense of nontangential limits and $|w_{z_0}(z)| < 1$ for $z\in\overline{\C}\setminus \E$, it follows that
\begin{align}
 \E = \Delta_\E^{-1}([-2,2]).
\end{align}
The function $\Delta_\E$ has $2(g+1)$ poles, half of which lie inside the unit disk. Moreover, there is exactly one critical point (i.e., a zero of $\Delta_\E'$) in each band of $\E$ and each gap of $\E$. While $\Delta_E$ maps all critical points in bands to $-2$, the critical points in gaps have $\Delta_\E$-value strictly greater than $2$.  For more details on the Ahlfors function and the discriminant, we refer to Appendix \ref{sec:AppAhlfors} (where in particular these properties are proven).  As will be crucial for our analysis, we define $\vz_\E$ to be some fixed ordering
of the poles of $\Delta_\E$ inside $\D$, i.e.
\begin{align}\label{eq:zEdef}
\vz_\E := \{z_0 = 0, z_1, \ldots, z_g\} \in \D^{g+1}, \quad z_k \in \D \text{ a pole of } \Delta_\E.
\end{align}

At the same time, there is a natural way of associating a rational function to an MCMV matrix $A = A(\{a_k\}, \vartheta; \vz) \in \A_{\per}(\vz)$. Given a value $a \in \D$, let
\begin{align}\label{eq:cUmat}
\cU(a) := \frac{1}{\rho}\begin{bmatrix}
1 & a \\
\overline{a} & 1
\end{bmatrix}, \quad \rho = \sqrt{1-|a|^2}
\end{align}
and define the monodromy matrix $\cM_A$ by
\begin{multline}\label{def:monodromy}
\cM_A(z) := \cU(a_0)
\begin{bmatrix}
b_{z_1}(z) & 0 \\
0 & 1
\end{bmatrix}
\cU(a_1)
\begin{bmatrix}
b_{z_1}(z) & 0 \\
0 & 1
\end{bmatrix}
\cU(a_2)
\begin{bmatrix}
b_{z_2}(z) & 0 \\
0 & 1
\end{bmatrix}
\cdots \\
\cdots
\cU(a_{2n-3})
\begin{bmatrix}
b_{z_{n-1}}(z) & 0 \\
0 & 1
\end{bmatrix}
\cU(a_{2n-2})
\begin{bmatrix}
b_{z_0}(z) & 0 \\
0 & 1
\end{bmatrix}
\cU(a_{2n-1})
\begin{bmatrix}
b_{z_0}(z) & 0 \\
0 & 1
\end{bmatrix}
\begin{bmatrix}
	e^{-i\vartheta}& 0\\
	0& e^{i\vartheta}
\end{bmatrix}.
\end{multline}
Denoting by $\mj$ the signature matrix $\mj=\left[\begin{smallmatrix} 1& 0\\ 0& -1 \end{smallmatrix}\right]$, we note that
$\cM_A(z)^*\mj\cM_A(z)\leq\mj$ for $z\in\D$ while $\cM_A(z)^*\mj\cM_A(z)=\mj$ when $z\in\partial\D$.
Functions of that type are called $\mj$-inner matrix functions and equation \eqref{def:monodromy} represents a factorization of $\cM_A$ into so-called elementary Blaschke--Potapov factors of the first kind. The study of general $\mj$-contractive matrix functions and their multiplicative structure goes back to Potapov \cite{Pot60}.

Now, let $B(z)=z\prod_{j=1}^{n-1}b_{z_j}(z)=\sqrt{\det \cM_A(z)}$ and consider the rational function
\begin{align}\label{def:discriminant}
\Delta_A(z) := \frac{1}{B(z)}\tr\bigl(\cM_A(z)\bigr).
\end{align}
We will show the following key result:
\begin{theorem}[Magic Formula for MCMV matrices]
\label{t:magicformula}
Fix a finite disjoint union of $g+1$ non-degenerate closed circular arcs $\E \subset \partial\D$, and let $\vz_\E$ be as in \eqref{eq:zEdef}.  Then, for any $A \in \A(\vz_\E)$,
\begin{align}\label{eq:magicformula}
A \in \cT_{\MV}(\E) \iff \Delta_\E(A) = S^{2(g+1)} + S^{-2(g+1)},
\end{align}
and in this case $\Delta_A = \Delta_\E$.

On the other hand, for fixed $\vz \in \D^n$ and $A_0 \in \A_{\per}(\vz)$, one has that $\sigma(A_0) = \Delta_{A_0}^{-1}([-2,2])$, and consequently
\begin{align}
\cT_{\MV}(\sigma(A_0)) = \{ A \in \A_{\per}(\vz) : \Delta_{A} = \Delta_{A_0}\}.
\end{align}
\end{theorem}

The Magic Formula reveals further structure of MCMV matrices relating to the discriminant $\Delta_\E$.  For simplicity, let us assume that the poles of $\Delta_E$ are simple (as this is typically the case). In that case, since $\Delta_\E$ is real-valued on $\partial\D$, 
it can be expressed in the form
\begin{align}\label{eq:DeltaERational}
\Delta_\E(z)=c+\sum_{k=0}^g \Bigl(c_kb_{z_k}(z)+\overline{c_k}b_{z_k}(z)^{-1}\Bigr).
\end{align}
We will abuse the notation for residues and define
\begin{align}\label{eq:coeffDeltaE}
	\Res_{z_k^*}\Delta_\E:=(b_{z_k}^{-1}\Delta_\E)(z_k^*) = c_k.
\end{align}

With \eqref{eq:DeltaERational} in mind, to understand the Magic Formula we need to understand each of the operators $b_{z_k}(A)$.  In Lemma \ref{lem:BlaschkeA} we show that these can be represented by an operator M\"obius transform of the same underlying CMV matrix $C$, but related to the ``shifted'' diagonal operator $D_k=(1-z_k\DD^*)^{-1}(\DD-z_k)$. Since $(D_k)_{2k-1,2k-1}=(D_k)_{2k,2k}=0$, it follows that $S^{-2k}b_{z_k}(A)S^{2k}$ has the same structure as an MCMV matrix. Moreover, since $b_{z_k}(A)$ is unitary, we have that $b_{z_k}(A)^{-1}=b_{z_k}(A)^*$. Hence, at the outermost diagonal of $\Delta_{E}(A)$ only one of the summands is non-vanishing. We have illustrated this for the off-diagonal block of  $c_{1}b_{z_{1}}(A)+\overline{c_1}b_{z_{1}}(A)^{-1}$ in the case $n = 4$ below:
\begin{align}\label{eq:MagicFormBlock}
\bigl(c_{1}b_{z_{1}}(A)+\overline{c_1}b_{z_{1}}(A)^{-1}\bigr)_{i,i+1}=\begin{bmatrix}
0&&&&&&&\\
0&0&&&&&&\\
\ast&\ast&\ast&&&\text{\Large{\textbf{0}}}&&\\
\ast&\ast&\ast&0&&&&&\\
\ast&\ast&\ast&0&0&&\\
\ast,\star&\ast,\star&\ast,\star&\star&\star&\star&&\\
\ast,\star&\ast,\star&\ast,\star&\star&\star&\star&0&\\
\ast,\star&\ast,\star&\ast,\star&\star&\star&\star&0&0\\
\end{bmatrix}.
\end{align}
Here $\ast$ and $\star$ indicate nonvanishing entries of $b_{z_{1}}(A)$ and $b_{z_{1}}(A)^{-1}$, respectively. In general, the outermost nonvanishing entry of $b_{z_k}(A)$ is the $(2k,2(g+1+k))$-entry. Since all the other summands in \eqref{eq:DeltaERational} are  vanishing at this position, the magic formula fixes the corresponding value of $b_{z_k}(A)$, i.e.
\begin{align}\label{eq:blasResidue}
\Delta_{\E}(A)=S^{2(g+1)}+S^{-2(g+1)}\implies b_{z_k}(A)_{2k,2(g+1+k)}=\frac{1}{\Res_{z_k^*}\Delta_{\E}}.
\end{align}
That this is a consequence of the structure of MCMV matrices will be proved in Theorem \ref{t:MCMVstructure}.

The above structure reveals an important property of MCMV matrices compared to their self-adjoint analog, GMP matrices.  The relation \eqref{eq:blasResidue} already indicates the importance of the values $\Res_{z_k^*}\Delta_\E$. In order for $b_{z_k}(A)$ to be well-defined, they should be nonzero. In fact, for a general MCMV matrix $A$ the values $|\Res_{z_k}\Delta_{S^{-2nj}AS^{2nj}}|$ are necessarily bounded away from zero.  We will later see (cf. Lemma \ref{lem:Resolvents}) that
\begin{align*}
\inf\{|\Res_{z_k}\Delta_{A}|:\ A\in\A(\vz)\}>0.
\end{align*}
This should be compared with \cite[Theorem 3.3]{Yud18}, where such a property was part of the definition of GMP matrices and guaranteed the existence of certain resolvents analogous to $b_{z_j}(A)$. It is natural that we do not need this condition since we are in the setting of unitary operators.

We also point out that the generalized discriminant for MCMV matrices involves the Ahlfors functions associated to two different points; in contrast, the analogous object for GMP matrices involves only the Ahlfors function at infinity.  This discrepancy has the consequence that the associated MCMV matrices are even-periodic with half of the gaps closed (cf. Appendix \ref{sec:AppAhlfors}).  While this could be avoided using a different discriminant, doing so would introduce further complications elsewhere.  In particular, the benefits of defining the discriminant as we have are 1) we can treat the even- and odd-periodic CMV cases uniformly, and 2) our discriminant is a rational function.

\subsection{Consequences for Schur and Caratheodory functions} \label{sec:Schur}

Of course, one cannot discuss CMV matrices without discussing Schur functions.  A Schur function is an analytic function $f: \D \to \overline{\D}$ mapping the open unit disk to its closure.  We denote by $\cS$ the class of all Schur functions.  Provided $f \in \cS$ is not a finite Blaschke product, the Schur algorithm
\begin{align*}
	\begin{split}
	f_0(z)&=f(z),\\
	zf_{k+1}(z)&= \frac{f_k(z)-\apa_k}{1-\overline{\apa_k}f_k(z)},\quad \apa_k=f_k(0)
	\end{split}
\end{align*}
determines an infinite sequence of parameters $\{\apa_k\} \in \D^\N$, also known as Schur parameters; conversely, any sequence $\{\apa_k\} \in \D^{\N}$ determines a function $f \in \cS$ by an associated continued fraction expansion (see, e.g., \cite{SimonOPUC1}).  For our purposes, it is more convenient to denote the Schur algorithm in terms of equivalences of projective lines, i.e.
\begin{align}\label{eq:CP1Schur}
\begin{bmatrix}
f_k(z) \\
1
\end{bmatrix}
\sim U(\apa_k)
\begin{bmatrix}
z & 0 \\
0 & 1
\end{bmatrix}
\begin{bmatrix}
f_{k+1}(z) \\
1
\end{bmatrix},
\end{align}
where $v_1 \sim v_2$ if and only if there exists some nonzero $\lambda \in \C$ such that $v_1 = \lambda v_2$. 

In correspondence to Schur functions are Caratheodory functions, analytic functions $F$ from $\D$ to the right half-plane normalized such that $F(0) = 1$.  A Caratheodory function $F$ can be determined from a function $f \in \cS$ by
\begin{align*}
F(z) = \frac{1 + zf(z)}{1-zf(z)},
\end{align*}
or in terms of projective lines by
\begin{align}\label{eq:CarathSchur}
\begin{bmatrix}
F(z) \\
1
\end{bmatrix} \sim
\begin{bmatrix}
1 & 1 \\
-1 & 1
\end{bmatrix}
\begin{bmatrix}
z & 0 \\
0 & 1
\end{bmatrix}
\begin{bmatrix}
f(z) \\
1
\end{bmatrix}.
\end{align}
In this latter language it is clear this process is invertible, so indeed this correspondence is one-to-one.  Caratheodory functions have a Herglotz integral representation as
\begin{align*}
F(z) = \int \frac{e^{it}+z}{e^{it}-z}\dd\nu_F(e^{it})
\end{align*}
for a unique probability measure $\nu_F$ on $\partial\D$, and are thus in correspondence with probability measures on the unit circle.  Consequently, Schur functions can be put into correspondence with half-line CMV matrices, in the sense that for a given half-line CMV matrix $\apC_+$, there exists $f_+ \in \cS$ such that
\begin{align}\label{eq:CaraOnesided}
\left\langle (\apC_+-z)^{-1}(\apC_++z)\delta_0,\delta_0 \right\rangle = \frac{1+zf_+(z)}{1-zf_+(z)},
\end{align}
and in fact one can check this $f_+$ is the Schur function with parameters $\{\apa_k\}_{k \in \N}$ agreeing with the coefficients of $\apC_+ = \apC_+(\{\apa_k\})$.  Similarly, whole-line CMV matrices have two associated Schur functions, one corresponding to each half-line.  Specifically, if $\{\apa_k\}_{k \in \Z} \in \D^\Z$ and $\apC = \apC(\{\apa_k\})$ is the associated CMV matrix, then one has that
\begin{align}
\left\langle (\apC-z)^{-1}(\apC+z)\delta_0, \delta_0 \right\rangle &= \frac{1+zf_+(z)f_-(z)}{1-zf_+(z)f_-(z)},
\end{align}
where $f_+$ is the Schur function with parameters $\{\apa_k\}_{k \in \N}$ and $f_-$ is the Schur function with parameters $\{-\overline{\apa_{- 1}},-\overline{\apa_{-2}}, \ldots\}$; cf. \cite{GeZi06,YudPeher06}.
%
%
Any Schur function $f$ has a natural factorization
\begin{align*}
f(z) = \biggl( \prod_k b_{w_k}(z)\frac{|w_k|}{w_k} \biggr) \exp\left(-\int \frac{e^{it} + z}{e^{it}-z}\dd\nu_f(e^{it}) + i\tau\right),
\end{align*}
where $\nu_f$ is a non-negative measure on $\partial\D$, $\tau\in\R/2\pi\Z$ (in fact, $\tau$ is the argument of $f(0)$), and the sequence $\{w_k\}$ of zeros satisfies the Blaschke condition $\sum (1 - |w_k|) < \infty$.  We define
\begin{align*}
\sigma_{\ess}(f) := \supp(\dd\nu_f) \cup \{w_k\}',
\end{align*}
where $\{w_k\}' \subset \partial\D$ denotes the set of limit points of the sequence $\{w_k\}$ of zeros of $f$; cf \cite[Lecture III]{Nikol86}.  Note that $\sigma_{\ess}(f)$ is a closed subset of $\partial\D$.

For a finite-gap set $\E \subset \partial\D$, a CMV matrix $\apC$ lies in $\cT_{\mv}(\E)$ precisely when its Schur functions $f_+$ and $f_-$ solve the following Riemann--Hilbert problem:
\begin{gather}\label{eq:RHprob1}
f_-(e^{it}) = \overline{e^{it}f_+(e^{it})} \; \text{ for a.e. } e^{it} \in \E, \\
\label{eq:RHprob2}
1-e^{it}f_+(e^{it})f_-(e^{it}) \neq 0 \; \text{ for } e^{it} \in \partial\D \setminus \E, \\
\label{eq:RHprob3}
\sigma_{\ess}(zf_+f_-) \subset \E.
\end{gather}
We denote the class of all such admissible functions $f_+$ by
\begin{align} \label{eq:Schurplus}
\cS_+(\E) := \bigl\{f_+ \in \cS : \exists f_- \in \cS \; \mbox{ s.t. } \eqref{eq:RHprob1}, \eqref{eq:RHprob2}, \eqref{eq:RHprob3} \mbox{ hold}\bigr\}.
\end{align}
The description of $\cS_+(\E)$ as being equivalent to $\cT_{\mv}(\E)$ in the finite-gap setting (and for even more general sets) was known already to Peherstorfer and Yuditskii \cite{YudPeher06}; specifically, they showed that, for $f_+ \in \cS_+(\E)$, the corresponding sequence $\{\apa_k\} \in \D^\Z$ is almost-periodic.

This perspective gives us an alternative way of stating our main results: membership in $\cS_+(\E)$ is equivalent to the existence of an $\E$-dependent Nevanlinna--Pick type interpolation, analogous to \eqref{eq:CP1Schur}, whose coefficients are periodic up to a rotational phase.
\begin{theorem}\label{t:nevPick}
Fix a finite disjoint union of $g+1$ non-degenerate closed circular arcs $\E \subset \partial\D$, and let $\vz_\E$ be as in \eqref{eq:zEdef}.  Then $f_+ \in \cS_+(\E)$ if and only if
\begin{align*}
\begin{bmatrix}
f_+ \\
1
\end{bmatrix}
\sim
\underbrace{
\cU(a_0)
\begin{bmatrix}
b_{z_1} & 0 \\
0 & 1
\end{bmatrix}
\cU(a_1)
\begin{bmatrix}
b_{z_1} & 0 \\
0 & 1
\end{bmatrix}
\cU(a_2)
\begin{bmatrix}
b_{z_2} & 0 \\
0 & 1
\end{bmatrix}
\cdots
\cU(a_{2g+1})
\begin{bmatrix}
b_{z_{0}} & 0 \\
0 & 1
\end{bmatrix}
\begin{bmatrix}
	e^{-i\vartheta}& 0\\
	0&e^{i\vartheta}
\end{bmatrix}}_{\textnormal{\normalsize$=\cM_{A(\{a_k\},\vartheta;\vz_\E)}(z)$}}
\begin{bmatrix}
f_+ \\
1
\end{bmatrix}
\end{align*}
for some $\{a_k\} \in \D^{2(g+1)}$ and some $\vartheta \in \R/2\pi\Z$ such that
\begin{align*}
\frac{1}{B(z)}\tr\Bigl(\cM_{A(\{a_k\},\vartheta;\vz_\E)}(z)\Bigr) = \Delta_\E(z).
\end{align*}
\end{theorem}

As an immediate corollary of Theorem \ref{t:nevPick} (cf. \eqref{eq:CarathSchur}), we resolve a conjecture of Simon \cite[Conjecture 11.9.6]{SimonOPUC2}:
\begin{corollary}\label{c:simonConj}
Fix a finite-gap set $\E \subset \partial\mathbb{D}$.  For any $\apC \in \mathcal{T}_{\mv}(\E)$, the Caratheodory function $F_+$ associated to the half-line restriction $\apC_+$ is a quadratic irrationality; i.e., there exist polynomials $a(z), b(z),$ and $c(z)$ such that $F_+$ solves
\begin{align*}
a(z)F_+(z)^2 + b(z)F_+(z) + c(z) = 0
\end{align*}
for all $z \in \mathbb{C} \setminus \E$.
\end{corollary}

Real numbers which are quadratic irrationalities (with $a, b$, and $c$ above as integers) are precisely those having eventually periodic continued fraction expansions.  If one understands the interpolation of Theorem \ref{t:nevPick} as a special continued fraction expansion for the Schur function $f_+$, Corollary \ref{c:simonConj} should come as no surprise; indeed, our result shows that almost-periodicity of the Schur parameters associated to absolutely continuous finite-gap CMV matrices is actually a consequence of an underlying periodicity which the Schur algorithm was too na\"ive to see.

\subsection{Methods and structure of the paper}

The relationship of CMV matrices to orthogonal polynomials was discovered by Cantero, Moral, and Vel\'azquez in 2003 \cite{CanMoVel03}.   Soon thereafter, the relationship of operator M\"obius transforms of CMV matrices to the study of orthogonal rational functions on the unit circle was studied in work of Vel\'azquez \cite{Ve08}.  We recall these relationships in Section 2 to motivate the following construction, as well as to prove a coefficient stripping formula for Caratheodory functions associated to bases of orthogonal rational functions.

Our approach to MCMV in the context of reflectionless operators is based on the functional model for the same, developed initially for Jacobi matrices by Sodin and Yuditskii \cite{SoYud97} and later adapted for Schur functions and CMV matrices by Peherstorfer and Yuditskii \cite{YudPeher06}.  Using the ideas developed by Eichinger and Yuditskii for GMP matrices (the Jacobi analog of MCMV matrices, cf. \cite{Eich16, Yud18}) and comparing this construction to that of Vel\'azquez proves one direction of the equivalences in Theorems \ref{t:periodiccoords}, \ref{t:magicformula}, and \ref{t:nevPick}.  We review the functional model for CMV matrices and reveal the corresponding MCMV structure in Section 3.

Having motivated our class of MCMV matrices and shown that finite-gap CMV matrices correspond to periodic MCMV matrices, we perform a direct spectral analysis for periodic MCMV matrices after reviewing the corresponding classical analysis of CMV matrices (cf., e.g., \cite{SimonOPUC1, SimonOPUC2, SimonSzego}) in Section 4. We also analyze the structure of a general MCMV matrix in Section \ref{sec:structure}.

Finally, we use the tools developed in Sections 2 through 4 to completely resolve the proofs of Theorems \ref{t:periodiccoords}, \ref{t:magicformula}, and \ref{t:nevPick} in Section \ref{sec:proofs}.

%% file: orf-28-08.tex
\section{Orthogonal rational functions}
The aim of this section is to establish a coefficient stripping formula for Caratheodory functions associated to bases of orthogonal rational functions (ORF). Our main result, Theorem \ref{thm:coeffStripping}, is an analog of the Stieltjes expansion for $m$-functions of Jacobi matrices \cite[Theorem 3.2.4]{SimonSzego} and of Peherstorfer's formula for orthogonal polynomials on the unit circle (OPUC) \cite[Theorem 3.4.2]{SimonOPUC1}. It will play an important role in Section \ref{sec:DST} where we seek to solve the direct spectral problem for periodic MCMV matrices.
\subsection{The Szeg\H o recursion}
Given a nontrivial (i.e., of infinite support) probability measure $\nu$ on $\partial\D$, one obtains the monic orthogonal polynomials $\Phi_k:=\Phi_k(z,\nu)$ by orthogonalizing the family $\{1,z,z^2,\dots\}$ in $L^2(\dd\nu)$.  The $\Phi_k$'s are known to satisfy a recurrence relation of the form
\begin{align} \label{eq:Szego}
	\Phi_{k+1}(z)=z\Phi_{k}(z)-\overline{\apa_k}\Phi_{k}^*(z)
\end{align}
for some sequence $\{\apa_k\}_{k=0}^\infty$ of numbers in $\D$.  Though it may look strange, we purposely write $-\overline{\apa_k}$ in \eqref{eq:Szego} so that the $\apa_k$'s coincide with the Schur parameters (introduced in Section \ref{sec:Schur}).  Following \cite{SimonOPUC1}, we shall also refer to the $\apa_k$'s as Verblunsky coefficients and recall there is a one-to-one correspondence between such sequences (in $\D^\N$) and nontrivial probability measures on $\partial\D$.

$\Phi_k^*$ is the reversed polynomial, that is,
\[
 \Phi_k^*(z)=z^k\overline{\Phi_k(\rz)}.
\]
While the notation of ${}^*$ is convenient, it is ambiguous.  It depends on the class
\[
\cL_k:=\lin\{1, z, \ldots, z^k\}
\]
and has a different meaning for $\cL_k$ and $\cL_j$ when $k\neq j$.  Note that the operation $\varphi\mapsto z^k\overline{\varphi(\rz)}$ acts as an involution on the subspace $\cL_k$.  We shall also use this abuse of notation for bases of orthogonal rational functions (where naturally $z^k$ is substituted by the Blaschke product corresponding to the poles of the first $k$ basis elements). Applying ${}^*$ for the class $\cL_{k+1}$ to \eqref{eq:Szego} yields the so-called Szeg\H{o} recursion
\begin{align}\label{eq:SzegoRecursion}
	\begin{bmatrix}
		\Phi_{k+1}(z)\\
		\Phi_{k+1}^*(z)
	\end{bmatrix}
	=
	\begin{bmatrix}
		1& -\overline{\apa_k}\\
		-\apa_k& 1
	\end{bmatrix}
	\begin{bmatrix}
		z& 0\\
		0& 1
	\end{bmatrix}
	\begin{bmatrix}
		\Phi_{k}(z)\\
		\Phi_{k}^*(z)
	\end{bmatrix}.
\end{align}
For the orthonormal polynomials $\varphi_k:=\Phi_k/\|\Phi_k\|$, this recursion takes the form
\begin{align}\label{eq:SzegoRecursionNorm}
	\begin{bmatrix}
		\varphi_{k+1}(z)\\
		\varphi_{k+1}^*(z)
	\end{bmatrix}
	=
	\cU(-\overline{\apa_k})
	\begin{bmatrix}
		z& 0\\
		0& 1
	\end{bmatrix}
	\begin{bmatrix}
		\varphi_{k}(z)\\
		\varphi_{k}^*(z)
	\end{bmatrix}
\end{align}
with $\cU(\cdot)$ as in \eqref{eq:cUmat}. 

\subsection{Coefficient stripping for orthogonal rational functions}\label{sec:MorePoles}
Let $z_0=0$ and fix a sequence of points $\{z_k\}_{k=1}^\i$ in $\D$ violating the Blaschke condition, i.e.
\begin{align}\label{eq:blaschkeViolate}
	\sum_k(1-|z_k|)=\i.
\end{align}
Note that \eqref{eq:blaschkeViolate} is trivially satisfied if $\sup_k |z_k|<1$ (which will always be the case in our setting). Recall that by $\{B_k\}$ we denote the family of finite Blaschke products
\begin{align*}
	B_0(z)=1,\quad B_k(z)=\prod_{j=1}^{k}b_{z_j}(z).
\end{align*}

Given a nontrivial probability measure $\nu$ on $\partial\D$, let $\{\varphi_k\}_{k=0}^\infty$ be the corresponding sequence of orthonormal rational functions obtained by orthogonalizing the family $\{B_k\}_{k=0}^\i$ in $L^2(\dd\nu)$.
With $\cL_k:=\lin\{\varphi_j: 0\leq j\leq k\}$, the associated ${}^*$-operator is now defined by $\varphi^*(z)=B_k(z)\overline{\varphi(\rz)}$ for $\varphi\in\cL_k$. Defining
$$
\eta_k^2:=1-|z_k|^2
$$
and choosing the right unimodular constants in the normalization (in particular, $\varphi_0\equiv 1$), the $\varphi_k$'s satisfy the recurrence relation (see \cite[Theorem 4.1.3]{BuGoHe99} and \cite{Ve08})
\begin{align}\label{eq:recurrenceRel1}
	\begin{bmatrix}
		\varphi_{k+1}(z)\\
		\varphi_{k+1}^*(z)
	\end{bmatrix}
	&=
	\frac{1-\overline{z_{k}}z}{1-\overline{z_{k+1}}z}\frac{\eta_{k+1}}{\eta_{k}}\cU(-\overline{a_k})
	\begin{bmatrix}
		b_{z_k}(z)& 0\\
		0& 1
	\end{bmatrix}
\begin{bmatrix}
	\varphi_{k}(z)\\
	\varphi_{k}^*(z)
\end{bmatrix}.
\end{align}
The rational functions $\psi_k$ of the second kind are defined by
\begin{align*}
\psi_0\equiv 1, \quad
\psi_k(z)=\int(\varphi_k(e^{it})-\varphi_k(z))\frac{e^{it}+z}{e^{it}-z}\dd\nu(e^{it}),
\quad k\geq 1
\end{align*}
and satisfy almost the same recurrence relation as $\varphi_k$, namely
\begin{align}\label{eq:recurrenceRel2}
\begin{bmatrix}
\psi_{k+1}(z)\\
-\psi_{k+1}^*(z)
\end{bmatrix}
&=
\frac{1-\overline{z_{k}}z}{1-\overline{z_{k+1}}z}\frac{\eta_{k+1}}{\eta_{k}}\cU(-\overline{a_k})
\begin{bmatrix}
b_{z_k}(z)& 0\\
0& 1
\end{bmatrix}
\begin{bmatrix}
\psi_{k}(z)\\
-\psi_{k}^*(z)
\end{bmatrix}.
\end{align}
Note that the coefficients $a_k$ in \eqref{eq:recurrenceRel1}--\eqref{eq:recurrenceRel2} belong to $\D$ and are explicitly given by
\begin{align*}
a_k={\Bigl\langle \frac{1-\overline{z_k}z}{z-z_{k+1}}, \varphi_{k}\Bigr\rangle_{\nu}}\,\Big/
{\Bigl\langle\frac{z_k-z}{z_{k+1}-z},\varphi_{k}^*\Bigr\rangle_{\nu}};
\end{align*}
cf. \cite[Theorem 4.1.2]{BuGoHe99}. Conversely, starting from arbitrary coefficients $\{a_k\}_{k=0}^\infty\in\D^\N$ one can generate a sequence of rational functions by \eqref{eq:recurrenceRel1} and show that they are orthogonal with respect to some probability measure $\nu$ on $\partial\D$. This is the content of the following known results:


\begin{theorem}{\cite[Theorems 8.1.4 and 9.2.1]{BuGoHe99}}\label{thm:ORFFavard}
  Suppose $\{z_k\}\in \D^\N$ violates the Blaschke condition, i.e. \eqref{eq:blaschkeViolate} holds.  Given a sequence $\{a_k\}\in \D^\N$, define the rational functions $\{\varphi_k\}_{k=0}^\infty$ (with $\varphi_0\equiv 1$) by \eqref{eq:recurrenceRel1} and let $\nu_k$ denote the Bernstein--Szeg\H{o} approximant
   \begin{align}
     \dd\nu_k=\frac{1-|z_k|^2}{|e^{it}-z_k|^2}\frac{1}{|\varphi_k^*(e^{it})|^2}\frac{\dd t}{2\pi}.
   \end{align}
  Then the associated Caratheodory function $F_{\nu_k}$ can be written as
   \begin{align} \label{eq:CaratheodoryRatio}
	 F_{\nu_k}(z)={\psi_k^*(z)}/{\varphi_k^*(z)}
	\end{align}
  and $\nu_k$ converges weakly to some probability measure $\nu$ which, in turn, is the unique measure of orthogonality for $\{\varphi_k\}$. In particular,
	\begin{align}\label{eq:CaratheodoryApprox}
		\lim\limits_{k\to\i}F_{\nu_k}(z)=F_\nu(z)
	\end{align}
	uniformly on compact subsets of $\D$.
\end{theorem}

Note that if $z_k\equiv 0$, then \eqref{eq:recurrenceRel1} reduces to the standard Szeg\H o recursion \eqref{eq:SzegoRecursion}. Just as for orthogonal polynomials on the unit circle, there is a one-to-one correspondence between coefficient sequences $\{a_k\}\in\D^\N$ and nontrivial probability measures on $\partial\D$. In fact, the theory for ORF generalizes the one for OPUC. We mention in passing that the assumption \eqref{eq:blaschkeViolate} ensures the measure $\nu$ of orthogonality be unique.


We are now ready to derive the promised coefficient stripping formula. 
Let $z_0=0,z_1,\dots, z_{p-1}$ be a finite number of points in $\D$ and consider the sequence $\{z_k\}_{k=0}^\i$ obtained by periodic extension of the initial $p$ values (i.e., $z_{k+p}=z_k$ for all $k$). Note that in this situation the Blaschke condition is trivially violated. Let
\begin{align*}
	Y_k(z)=\begin{bmatrix}
	\psi_k(z) & \varphi_k(z) \\
	-\psi_k^*(z) & \varphi_k^*(z)
	\end{bmatrix}; \quad\text{in particular}, \; Y_0=\begin{bmatrix}
	1 & 1\\
	-1& 1
	\end{bmatrix}
\end{align*}
and define
\begin{align}\label{def:TransferMatrix}
	\cB(z):=\begin{bmatrix}
	\cB_{11}(z)& \cB_{12}(z)\\
	\cB_{22}(z)& \cB_{21}(z)
	\end{bmatrix}
	=Y_0\begin{bmatrix}b_{z_{0}}(z)&0\\
	0& 1
	\end{bmatrix}\cU(a_{0})
	\cdots
	\begin{bmatrix}b_{z_{p-1}}(z)&0\\
	0& 1
	\end{bmatrix}
	\cU(a_{p-1})
	Y_0^{-1}.
\end{align}
Our result then reads as follows:
\begin{theorem}\label{thm:coeffStripping}
	Let $F_\nu$ be the Caratheodory function associated to the sequence $\{a_k\}_{k=0}^\i$ and suppose $F_\nu^{(1)}$ corresponds to the shifted
    sequence $\{a_{k+p}\}_{k=0}^\i$. Then
	\begin{align}\label{eq:coeffStripping}
	F_{\nu}(z)=\frac{\cB_{11}(z)F_{\nu}^{(1)}(z)+\cB_{12}(z)}{\cB_{21}(z)F_{\nu}^{(1)}(z)+\cB_{22}(z)}
	\end{align}
	and $\cB$ can be expressed in terms of the ORF and the rational functions of the second kind by
	\begin{align}\label{eq:Bmatrix}
		\cB(z)=\frac{1}{2}\begin{bmatrix}
		\psi_p(z)+\psi_p^*(z) & \psi_p^*(z)-\psi_p(z) \\
		\varphi_p^*(z)-\varphi_p(z) & \varphi_p(z)+\varphi_p^*(z)
		\end{bmatrix}.
	\end{align}
\end{theorem}

\begin{proof}
Due to \eqref{eq:recurrenceRel1} and \eqref{eq:recurrenceRel2}, we have
	\begin{align*}
		Y_{k+1}(z)=\frac{1-\overline{z_{k}}z}{1-\overline{z_{k+1}}z}\frac{\eta_{k+1}}{\eta_{k}}\cU(-\overline{a_k})
		\begin{bmatrix}
		b_{z_k}(z)& 0\\
		0& 1
		\end{bmatrix}
		Y_k(z).
	\end{align*}
Iterating $p$ times, starting from $k=p-1$, the factors in front of the $U$'s cancel (due to telescoping) and it follows that
	\begin{align}\label{eq:ORFTransfer}
		Y_{p}(z)=W(z) Y_0,
	\end{align}
where $W$ is the transfer matrix
	\begin{align}\label{def:transferMatrix2}
		W(z):=W(z,\{z_k\},\{a_k\})=\cU(-\overline{a_{p-1}})
		\begin{bmatrix}
		b_{z_{p-1}}(z)& 0\\
		0& 1
		\end{bmatrix}
		\cdots
		\cU(-\overline{a_{0}})
		\begin{bmatrix}
		b_{z_{0}}(z)& 0\\
		0& 1
		\end{bmatrix}.
	\end{align}
Using a superscript ${(j)}$ for the objects related to to the shifted sequence $\{a_{k+jp}\}_{k=0}^{\infty}$, 
we obtain in a similar way that
	 \begin{align*}
	 	Y_{np}^{(1)}(z)=W^{(n)}(z)\cdots W^{(1)}(z)Y_0
\end{align*}
and	
\begin{align*}
	 		Y_{(n+1)p}(z)=W^{(n)}(z)\cdots W^{(1)}(z)W(z)Y_0.
	 \end{align*}
Hence,
	 \begin{align*}
	 	Y_{(n+1)p}(z)=Y_{np}^{(1)}(z)Y_0^{-1}W(z)Y_0
	 \end{align*}
and	considering the second row of this identity (or rather its transpose) yields
	 \begin{align} \label{eq:psiphi}
	 \begin{bmatrix}
	 \psi_{(n+1)p}^*(z)\\
	 \varphi_{(n+1)p}^*(z)
	 \end{bmatrix}
	 =
	 (\mj Y_0^{-1} W(z) Y_0 \mj)^\intercal
	 \begin{bmatrix}
	 \bigl(\psi_{np}^{(1)}\bigr)^*(z)\\
	 \bigl(\varphi_{np}^{(1)}\bigr)^*(z)
	 \end{bmatrix},
     \quad \mj=\begin{bmatrix}
	 1& 0\\
	 0& -1
	 \end{bmatrix}.
     \end{align}
To see that
\begin{equation} \label{eq:M(z)}
(\mj Y_0^{-1} W(z) Y_0 \mj)^\intercal = M(z)
\end{equation}
with $\cB$ as defined in \eqref{def:TransferMatrix}, we use that $Y_0\mj=2\mj Y_0^{-1}$, $Y_0^\intercal=2Y_0^{-1}$, and $\mj\cU(-\overline{a})\mj=\cU(a)^\intercal$. Due to \eqref{eq:CaratheodoryRatio} and \eqref{eq:CaratheodoryApprox}, we now obtain \eqref{eq:coeffStripping} by passing to the limit as $n\to\infty$ in \eqref{eq:psiphi}. Finally, the identity \eqref{eq:Bmatrix} follows directly from \eqref{eq:ORFTransfer} and \eqref{eq:M(z)}.
\end{proof}

If the sequence $\{a_k\}$ is periodic (or periodic up to a rotational phase) and the period matches the period of the sequence $\{z_k\}$, then our result simplifies and \eqref{eq:coeffStripping} turns into a quadratic equation for $F_\nu$. The result becomes particularly simple if we pass from the relation for Caratheodory functions to the one for Schur functions:

\begin{corollary}\label{cor:periodicInterSchur}
Suppose $\{a_k\}_{k=0}^\i$ is periodic up to a rotational phase with period $p$ and phase $-2\vartheta$, and let $f_\nu$ denote the associated Schur function; see
    \eqref{eq:CarathSchur}.  If
	\begin{align}\label{eq:transferNatrix3}
		\cM(z)=\cU(a_0)\begin{bmatrix}
		b_{z_1}(z) & 0\\
		0& 1
		\end{bmatrix}
        \cU(a_1)
		\cdots
		\begin{bmatrix}
		b_{z_{p-1}}(z) & 0\\
		0& 1
		\end{bmatrix}
		\cU(a_{p-1})
		\begin{bmatrix}
		b_{z_0}(z) & 0\\
		0& 1
		\end{bmatrix}
		\begin{bmatrix}
		e^{-i\vartheta}& 0\\
		0& e^{i\vartheta}
		\end{bmatrix},
	\end{align}
    then
	\begin{align}\label{eq:nevPickSchur}
	\begin{bmatrix}
	f_{\nu}\\
	1
	\end{bmatrix}
	\sim
	\cM(z)
	\begin{bmatrix}
	f_\nu\\
	1
	\end{bmatrix}.
	\end{align}
\end{corollary}
\begin{proof}
	Since $a_{k+p}=e^{-2i\vartheta}a_k$, it follows that $f_\nu=e^{2i\vartheta}f_\nu^{(1)}$. Thus \eqref{eq:nevPickSchur} follows from \eqref{eq:coeffStripping} and \eqref{eq:CarathSchur}.
\end{proof}

%% file: funcModel-28-09.tex
\section{The functional model}  \label{sec:3}

In the last two decades a significant amount of progress has been made in understanding reflectionless one-dimensional operators as being related to multiplication operators on certain subspaces of Hardy spaces associated to multiply-connected Riemann surfaces.  We broadly refer to this construction as a ``functional model'' for the associated operators.  In this section, we first recall the requisite definitions to develop such models, followed by the specific model of Peherstorfer--Yuditskii for almost-periodic CMV matrices, and finally associate to this our model for MCMV matrices.

\subsection{Hardy spaces of character automorphic functions}
Fix a finite-gap set $\E \subset \partial\D$.  By means of the Koebe--Poincar\'e uniformization theorem, the spectral complement in the Riemann sphere $\overline{\C} \setminus \E$ is uniformized by the disk $\D$; that is, there exists a Fuchsian group $\G$ and a meromorphic function $\fz : \D \to \overline{\C} \setminus \E$ with the following properties:
\begin{align*}
1. \quad & \forall z \in \overline{\C} \setminus \E \;\; \exists \, \z \in \D : \fz(\z) = z, \\
2. \quad & \fz(\z_1) = \fz(\z_2) \iff \exists \, \g \in \G : \z_1 = \g(\z_2).
\end{align*}
We fix this map by requiring $\fz$ map the interval $(-1,1)$ onto a fixed connected component of $\partial\D \setminus \E$.  Due to this choice, there exists a fundamental domain $\cF$ for the action of $\G$ which is symmetric with respect to complex conjugation, i.e.
\begin{align} \label{eq:domain}
\cF = \{\overline{\z} : \z\in\cF\},\quad  \g^{-1}(\z) = \overline{\g}(\z) := \overline{\g(\overline{\z})},\quad \g \in \G.
\end{align}
It follows that
\begin{equation}
\overline{\fz(\overline{\z})} = \fz(\z)^{-1};
\end{equation}
in particular, if $\z_0 \in \cF$ is such that $\fz(\z_0) = 0$, then $\fz(\overline{\z_0}) = \infty$.

We denote by $\G^*$ the group of unitary characters of $\G$; that is, group homomorphisms from $\G$ into $\T := \R/2\pi\Z$.  By the covering space formalism, $\G$ is group isomorphic to the fundamental group $\pi_1(\overline{\C} \setminus \E)$, and so $\G^* \cong \T^g$ (where $g+1$ is the number of gaps of $\E$).

Let $H^2 = H^2(\D)$ denote the usual Hardy space of the unit disk.  For $\a \in \G^*$, we consider the Hardy space of character automorphic functions
\begin{align*}
H^2(\a) := \bigl\{ f \in H^2 : f \circ \g = e^{ i\a(\g)}f \;\; \forall \g \in \G\bigr\}
\end{align*}
equipped with the standard $H^2$ inner product
\begin{align*}
\langle f, g \rangle &= \int_0^{2\pi} f(e^{it})\overline{g(e^{it})} \frac{\dd t}{2\pi}.
\end{align*}
More generally, we define the larger space $L^2(\a)$ as the space of those functions $f : \partial\D \to \C$ which are square integrable and $\a$-automorphic:
\begin{align*}
L^2(\a) := \bigl\{ f : \partial\D \to \C : \|f\|^2 < \infty,  f \circ \g = e^{ i\a(\g)}f \;\; \forall \g \in \G\bigr\}.
\end{align*}
Naturally, $H^2(\a) \subset L^2(\a)$ via the identification of a function $f \in H^2$ with its radial limit function on the boundary.

For finite-gap sets $\E$, a fundamental result of Widom \cite{Wid71} implies that $H^2(\a)$ is nontrivial for all $\a \in \G^*$.  This in fact applies to all subsets $\E\subset\C$ of so-called Parreau--Widom type (see, e.g., \cite{Has83} for details).  By continuity of the point evaluation functional, $H^2(\a)$ admits a family of reproducing kernels $\{k^\a(\z,\z_1)\}_{\z_1 \in \D}$ such that
\begin{align}
\langle f, k^\a(\,\cdot\,, \z_1) \rangle = f(\z_1) \quad \forall f \in H^2(\a).
\end{align}
By the reproducing property and nontriviality of $H^2(\a)$, it follows that $k^\a(\z_1,\z_1) > 0$.  We may thus define the corresponding normalized vectors by
\begin{align*}
K^\a(\z,\z_1) &= \frac{k^\a(\z, \z_1)}{\sqrt{k^\a(\z_1,\z_1)}}
\end{align*}
and note that
\begin{equation} \label{eq:inner}
 \langle f, K^\a(\,\cdot\,, \z_1) \rangle = \frac{f(\z_1)}{K^\a(\z_1,\z_1)} \quad \forall f \in H^2(\a).
\end{equation}
We will sometimes abbreviate $k^\a_{\z_1}(\z) := k^\a(\z, \z_1)$ and $K^\a_{\z_1}(\z) := K^\a(\z, \z_1)$.  Since $f \in H^2(\a)$ implies $\overline{f(\overline{\z})} \in H^2(\a)$, we have
\begin{align}
\label{eq:reprod}
\overline{K^\a(\overline{\z},\z_1)}=K^\a(\z, \overline{\z_1});
\end{align}
in particular, $K^\a(\z_1,\z_1) = K^\a(\overline{\z_1},\overline{\z_1})$.
We shall make frequent use of \eqref{eq:reprod} (as well as \eqref{eq:bsym} below) in our computations.

For finite-gap sets $\E$, the map $\a \mapsto K^\a(\z_1,\z_1)$ is continuous for every $\z_1 \in \D$.   This property is in fact responsible for the almost-periodic structure of the CMV matrices in $\cT_\mv(\E)$, compare with Theorem \ref{t:PY} below. We mention in passing that this type of continuity in the character is known to hold for all Parreau--Widom sets $\E\subset\C$ satisfying the so-called Direct Cauchy Theorem (see, e.g., \cite{Has83}).

For a fixed $\z_1 \in \D$ we denote by
\begin{align}\label{def:Blaschke}
\fb(\z, \z_1) := e^{i\phi}\prod_{\g \in \G}\frac{\z - \g(\z_1)}{1-\overline{\g(\z_1)}\z}
\end{align}
the Blaschke product with zeros at the orbit of $\z_1$ under $\G$, and with $\phi = \phi(\z_1)$ normalized such that $\fb(\z_1, \overline{\z_1}) > 0$. As follows directly from \eqref{def:Blaschke} and \eqref{eq:domain}, we have
\begin{equation} \label{eq:bsym}
 \overline{\fb(\overline{\z}, \z_1)} = \fb(\z, \overline{\z_1}).
\end{equation}
If we want to suppress the dependence on $\z$, we may also write $\fb_{\z_1}(\z)=\fb(\z,\z_1)$.  Note that $\fb_{\z_1}$ is related to the potential-theoretic Green's function $G_{\,\overline{\C}\setminus \E}(z,z_1)$ of the domain $\overline{\C} \setminus \E$ with pole at $z_1 = \fz(\z_1)$ by
\begin{align}\label{eq:btoGreens}
-\log\bigl|\fb_{\z_1}(\z)\bigr| &= G_{\,\overline{\C}\setminus \E}\bigl(\fz(\z),z_1\bigr).
\end{align}
Moreover, $\fb_{\z_1}$ is character automorphic with some character $\mu_{\z_1}$, i.e.
\begin{align}
\fb_{\z_1}(\g(\z)) = e^{i\mu_{\z_1}(\g)}\fb_{\z_1}(\z) \quad \forall \g \in \G.
\end{align}

Interestingly, since $\E\subset\partial\D$,
\begin{align*}
\log|z| &= G_{\,\overline{\C}\setminus \E}(z,\infty) - G_{\,\overline{\C}\setminus \E}(z,0).
\end{align*}
Thus we may represent the uniformization $\fz$ as a ratio of distinguished Blaschke products:
\begin{align}\label{eq:phase}
\fz(\z) &= e^{i\phi_0} \frac{\fb(\z, \z_0)}{\fb(\z, \overline{\z_0})},
\end{align}
where $\fz(\z_0) = 0$ as before and $\phi_0\in\T$ is some phase.  Since $\fz$ is automorphic, it follows that $\mu_{\z_0} = \mu_{\overline{\z_0}}$; we will abbreviate this common character by $\mu_0$.

In the coming subsections, we will study multiplication by this uniformization map $\fz$ as a linear operator on $L^2(\a)$ with respect to different bases.  To this end, we require a technical lemma on reproducing kernels which allows us to effectively compute residues.  We first recall the following orthogonal decomposition of $H^2(\a)$:
\begin{lemma}[\cite{SoYud97}]\label{lem:orthogonalDecomp}
	For $\z_1\in \D$, we have
	\begin{align}
		\cK_{\fb_{\z_1}}(\a):=H^2(\a)\ominus \fb_{\z_1}H^2(\a-\mu_{\z_1})=\lin\{k^\a_{\z_1}\}.
	\end{align}
\end{lemma}
\begin{proof}
	Using the reproducing kernel property, it is clear that $k^\a_{\z_1}\in\cK_{\fb_{\z_1}}(\a)$.  Conversely, let $f\in H^2(\a)$ and suppose $f\perp k^\a_{\z_1}$. Then
	\begin{align*}
		0=\langle f,k^\a_{\z_1}\rangle=f(0)
	\end{align*}
	and since $f$ is character automorphic, $f(\g(0))=0$ for all $\g\in\G$. The standard factorization theorem for $H^2$ functions and a comparison of the characters now imply that $f\in\fb_{\z_1}H^2(\a-\mu_{\z_1})$.
\end{proof}

\begin{lemma}\label{lem:EvaluatePole}
	For $\z_1 \in \D$, let $\fb_{\z_1}$ and $\mu_{\z_1}$ be as above.  If $f\in L^2(\a)$ is such that $\fb_{\z_1}f\in H^2(\a+\mu_{\z_1})$, then for $\z_2 \neq \z_1$ we have
	\begin{align}
		\langle f,K^\a_{\z_2}\rangle=\frac{f(\z_2)}{K^\a(\z_2,\z_2)}-\frac{(\fb_{\z_1}f)(\z_1)}{K^{\a+\mu_{\z_1}}(\z_1,\z_1)}\frac{K^{\a+\mu_{\z_1}}(\z_2,\z_1)}{\fb_{\z_1}(\z_2)K^\a(\z_2,\z_2)}.
	\end{align}
\end{lemma}

\begin{proof}
By our assumptions and Lemma \ref{lem:orthogonalDecomp},
\begin{align*}
g: = \fb_{\z_1}f - \langle \fb_{\z_1}f, K_{\z_1}^{\a + \mu_{\z_1}}\rangle K_{\z_1}^{\a + \mu_{\z_1}} \in \fb_{\z_1}H^2(\a).
\end{align*}
Since $\langle K_{\z_1}^{\a + \mu_{\z_1}}, \fb_{\z_1}K_{\z_2}^\a \rangle = 0$, we have on the one hand that
\begin{align*}
\langle g, \fb_{\z_1}K_{\z_2}^\a \rangle = \langle \fb_{\z_1}f, \fb_{\z_1}K_{\z_2}^\a \rangle - 0
= \langle f, K_{\z_2}^\a \rangle.
\end{align*}
On the other hand, as $g/\fb_{\z_1}\in H^2(\a)$, 
we also have
\begin{align*}
\langle g, \fb_{\z_1}K_{\z_2}^\a \rangle = \langle g/\fb_{\z_1}, K_{\z_2}^\a \rangle
= \frac{f(\z_2)}{K^\a(\z_2,\z_2)}-\frac{(\fb_{\z_1}f)(\z_1)}{K^{\a+\mu_{\z_1}}(\z_1,\z_1)}\frac{K^{\a+\mu_{\z_1}}(\z_2,\z_1)}{\fb_{\z_1}(\z_2)K^\a(\z_2,\z_2)}. 
\end{align*}
This completes the proof.
\end{proof}

\subsection{The Peherstorfer--Yuditskii model for CMV matrices} \label{sec:PY}

To motivate the MCMV functional model, we first recall the functional model for the usual CMV matrices.  Everything that follows in this section is in some way already presented in the literature. We try to be quite precise anyway, because we feel that the meaning of the additional parameter $\tau\in \R/2\pi\Z$ has not really been discussed yet in terms of the functional model. Moreover, it will give us an understanding of the notion of \textit{periodicity up to a phase} in CMV matrices, which will be important in the later part of our paper.

Let $(\a,\tau)\in\G^*\times \T$ and define
\begin{align*}
	x_0^{\a,\tau}=K^{\a}_{\overline{\z_0}},\quad x_1^{\a,\tau}=e^{i\tau}\fb_{\overline{\z_0}} K^{\a-\mu_0}_{\z_0}, \\
	\chi_0^{\a,\tau} = e^{i\tau}K^\a_{\z_0}, \quad \chi_1^{\a,\tau} = \fb_{\z_0}K^{\a - \mu_0}_{\overline{\z_0}}.
\end{align*}
For $\phi_0$ given by \eqref{eq:phase} we define, for every $l\in\Z$,
\begin{align}\label{def:CMVbases1}
x_{2l}^{\a,\tau}=e^{-il\phi_0}\fb_{\z_0}^l\fb_{\overline{\z_0}}^lx_0^{\a-2l\mu_0,\tau},\quad x_{2l+1}^{\a,\tau}=e^{il\phi_0}\fb_{\z_0}^l\fb_{\overline{\z_0}}^{l+1}x_1^{\a-(2l+1)\mu_0,\tau}, \\
\label{def:CMVbases2}
\chi_{2l}^{\a,\tau} = e^{il\phi_0}\fb_{\z_0}^l\fb_{\overline{\z_0}}^l\chi_0^{\a-2l\mu_0,\tau}, \quad \chi_{2l+1}^{\a,\tau}=e^{-il\phi_0}\fb_{\z_0}^{l+1}\fb_{\overline{\z_0}}^{l}\chi_1^{\a-(2l+1)\mu_0,\tau}.
\end{align}
It is straightforward to see that for any $\tau \in \T$, $\{x^{\a,\tau}_0, x^{\a,\tau}_1\}$ and $\{\chi^{\a,\tau}_0, \chi^{\a,\tau}_1\}$ form two distinct orthonormal bases of the two-dimensional subspace
\begin{align}\label{eq:CMVKspace}
\cK_{\fb_{\z_0}\fb_{\overline{\z_0}}}(\a) := \spa\{ K_{\z_0}^\a, K_{\overline{\z_0}}^\a\}
 = H^2(\a) \ominus b_{\z_0}b_{\overline{\z_0}} H^2(\a - 2\mu_0).
\end{align}
Iterating this decomposition exhausts $H^2(\a)$ (and in fact, the larger space $L^2(\a)$); in particular, we have the following:
\begin{proposition}
The systems $\{x_k^{\a,\tau}\}$ and $\{\chi_k^{\a,\tau}\}$ for $k \in \N$ (resp., $k \in \Z$) form orthonormal bases for $H^2(\a)$ (resp., $L^2(\a)$).
\end{proposition}

Almost-periodic absolutely continuous whole-line CMV matrices with spectrum $\E$ arise exactly as multiplication by $\fz$ in the basis $\{\chi_k^{\a,\tau}\}_{k \in \Z}$:

\begin{theorem}[Peherstorfer--Yuditskii \cite{YudPeher06}]\label{t:PY}
Multiplication by $\fz$ in the basis $\{\chi_k^{\a,\tau}\}_{k \in \Z}$ is a CMV matrix $\apC(\a,\tau)$ with almost-periodic Verblunsky coefficients given by
\begin{align}\label{eq:CMVCoeffs}
	\apa_k(\a,\tau)=e^{-ik\phi_0}\mathbf{A}(\a-k\mu_0,\tau),\quad \apr_k(\a)=\mathbf{R}(\a-k\mu_0),
\end{align}
where
\begin{align}
	\mathbf{A}(\alpha,\tau)=e^{-i\tau}\frac{K^{\a}(\z_0,\overline{\z_0})}{K^{\a}(\z_0,\z_0)}, \quad
	\mathbf{R}(\a)=\fb(\overline{\z_0},\z_0)\frac{K^{\a-\mu_0}(\z_0,\z_0 )}{K^{\a}(\z_0,\z_0)}=\sqrt{1-|\mathbf{A}(\a,\tau)|^2}.
\end{align}
\end{theorem}
\begin{remark}
Peherstorfer and Yuditskii actually studied the family of Schur functions $f^{\a,\tau}$ given by
\begin{align}\label{eq:functSchur}
f^{\a,\tau}\circ\fz := e^{-i\tau}\frac{K^\a_{\overline{\z_0}}}{K^\a_{\z_0}},
\end{align}
but this is equivalent by equality of Schur parameters and Verblunsky coefficients.  This perspective explains the necessity of including the parameter $\tau$; we wish to completely classify such Schur functions, not merely classify them up to a rotation.
\end{remark}

We can see this theorem via the $LM$ structure by alternating between the basis $\{\chi_k^{\a,\tau}\}$ and the dual basis $\{x_k^{\a,\tau}\}$.  Denoting
\begin{align*}
\Theta_k(\a,\tau) := \begin{bmatrix}
\,\overline{\apa_k(\a,\tau)} & \apr_k(\a) \\
\,\apr_k(\a) & -\apa_k(\a,\tau)
\end{bmatrix},
\end{align*}
we have the following
\begin{lemma}
With notation as above,
\begin{align}
\begin{bmatrix}
\chi_0^{\a,\tau} \\
\chi_1^{\a,\tau}
\end{bmatrix}
=
\Theta_0(\a,\tau) \begin{bmatrix}
x_0^{\a,\tau} \\
x_1^{\a,\tau}
\end{bmatrix}, \quad
\fz\begin{bmatrix}
x_1^{\a,\tau} \\
x_2^{\a,\tau}
\end{bmatrix}
=
\Theta_1(\a,\tau)
\begin{bmatrix}
\chi_1^{\a,\tau}\\
\chi_2^{\a,\tau}
\end{bmatrix}.
\end{align}
\end{lemma}
\begin{proof}
Since $K^\a_{\z_0}, K^{\a}_{\overline{\z_0}}\in\cK_{\fb_{\z_0}\fb_{\overline{\z_0}}}(\a)$, it follows from the reproducing kernel property that 
\begin{align*}
\chi_0^{\a,\tau}=\overline{\mathbf{A}(\a,\tau)}x_0^{\a,\tau}+\mathbf{R}(\a)x_1^{\a,\tau},\quad x_0^{\a,\tau}=\mathbf{A}(\a,\tau)\chi_0^{\a,\tau}+\mathbf{R}(\a)\chi_1^{\a,\tau}.
\end{align*}
Using \eqref{eq:phase}, the lemma follows by algebraic manipulations.
\end{proof}

\begin{proof}[Proof of Theorem \ref{t:PY}]
We can shift the relations in the previous lemma to see that, taking
\begin{align*}
L := L(\a,\tau) = \bigoplus_{l \in \Z}\Theta_{2l}(\a,\tau), \quad
M := M(\a,\tau) = \bigoplus_{l \in \Z}\Theta_{2l+1}(\a,\tau),
\end{align*}
then $M$ sends the basis $\{\chi_k^{\a,\tau}\}_{k \in \Z}$ to $\{\fz(x_k^{\a,\tau})\}_{k \in \Z}$ and $L$ sends $\{x_k^{\a,\tau}\}_{k\in\Z}$ to $\{\chi_k^{\a,\tau}\}_{k\in\Z}$.  Thus, we have that multiplication by $\fz$ in the basis $\{\chi_k^{\a,\tau}\}$ is given by $\apC = LM$, which is a CMV matrix with precisely the Verblunsky coefficients $\apa_k(\a,\tau)$ as above.
\end{proof}

We conclude by pointing out that the CMV matrix $\apC(\a,\tau)$ is periodic if and only if $\phi_0\in2\pi \mathbb Q$ and there exists $N\geq 1$ such that $\mu_0N=\mathbf{0}_{\G^*}$.  If only the latter holds (i.e., $\phi_0\notin2\pi \mathbb Q$), then $\apC(\a,\tau)$ is periodic up to a phase with phase $e^{-iN\phi_0}$.
%

\subsection{A modified basis suited for periodicity}\label{sec:modifiedBasis}
We have seen in the previous subsection that whether the isospectral torus of CMV matrices consists of periodic or almost periodic operators is related to whether there exists $N\geq 1$ such that $\bigl(\fb_{\z_0}\fb_{\overline{\z_0}}\bigr)^N$ can be lifted to a single valued function on $\overline\C\setminus \E$. In this section  we will study a basis associated to Blaschke products which have this property, and by definition the corresponding multiplication operator in this basis will be periodic.
To fix the notation, let
\[
\vz := \{z_0=0,z_1,\dots, z_{n-1}\} \in\D^n
\]
and take a point $\z_l\in\fz^{-1}(z_l)$ for $l=0,1,\ldots,n-1$. Define
\begin{align} \label{eq:BB}
\fB := \fB_{\vz} =\prod_{l=0}^{n-1}\fb_{\z_l}
\end{align}
and let $\beta := \beta_{\vz}$ denote its character.  Our condition on the vector $\vz$ is that $\b$ is a half-period (i.e., $2\b=\mathbf{0}_{\Gamma^*}$).

\begin{remark}
The Ahlfors function shows by example that this condition always can be met and a function as in \eqref{eq:BB} indeed exists. Recall that $w_\i$ denotes the Ahlfors function of $\overline\C\setminus\E$ and the point $\i$. If $\E$ has $g+1$ gaps, $w_\i$ has exactly $g$ zeros in $\D$, say $z_1,\dots, z_g$, and one zero at $\i$. Moreover, $|w_\i|=1$ on $\E$ and $|w_\i|<1$ in $\overline\C\setminus\E$. From this it follows that the pullback of $zw_\i$, that is, $\mathfrak w_\i:=\fz (w_\i\circ\fz)$ is a function with the properties mentioned above, with $n = g+1$; see Appendix \ref{sec:AppAhlfors} for a more detailed discussion.
\end{remark}

Denoting by $\fz_l$ the pullback of $b_{z_l}$ to the uniformization, i.e.
\begin{align}
\fz_l := b_{z_l} \circ \fz,
\end{align}
we see that there exists a certain phase $\phi_l$ such that
\begin{align}\label{eq:BlaschkeUniformization}
	\fz_l(\z) =e^{i\phi_l}\frac{\fb(\z,\z_l)}{\fb(\z,\overline{\z_l})},
\end{align}
cf. \eqref{eq:phase}. Hence the characters of $\fb_{\z_l}$ and $\fb_{\overline{\z_l}}$ coincide. Let us abbreviate them by $\mu_l$. If we denote $\fB^*(\z):=\overline{\fB(\overline{\z})}=\prod_{j=0}^{n-1}\fb_{\overline{\z_j}}(\z)$, then this implies that the character of $\fB\fB^*$ is $2\b$. By our assumption,
\begin{align*}
	2\b=2(\mu_0+\mu_1+\dots+\mu_{n-1})=\mathbf{0}_{\G^*}.
\end{align*}
This allows us to decompose $H^2(\a)$ by iterations of the finite-dimensional subspace
\begin{align}\label{eq:exhaust}
\cK_{\fB\fB^*}(\a) := \spa\{K^\a_{\z_0}, K^\a_{\overline{\z_0}},K^\a_{\z_1}, K^\a_{\overline{\z_1}}, \cdots, K^\a_{\z_{n-1}}, K^\a_{\overline{\z_{n-1}}}\} = H^2(\a) \ominus \fB\fB^* H^2(\a),
\end{align}
\textit{without shifting the character}.  This lack of shift is ultimately what will lead to periodicity up to a phase.

Our strategy is as follows: suppose we have a vector $\vz \in \D^n$ with associated Blaschke product $\fB$ as above having character $\b$ a half-period.  Similar to CMV matrices, we will have one step comparing symmetric pairs $\z_l, \overline{\z_l}$ corresponding to shifting from a pole $z_l$ inside the disk to its symmetric point $z_l^*$ outside the disk; this corresponds to the representation
\begin{align} \label{eq:spanKK}
\spa\{K^\a_{\z_l}, K^\a_{\overline{\z_l}}\} = H^2(\a) \ominus \fb_{\z_l}\fb_{\overline{\z_l}}H^2(\a - 2\mu_l),
\end{align}
which we can iterate to exhaust $\cK_{\mathfrak{B}\mathfrak{B}^*}(\a)$ as follows:
\begin{align*}
\cK_{\mathfrak{B}\mathfrak{B}^*}(\a) = H^2(\a) \ominus \fb_{\z_0}\fb_{\overline{\z_0}}\left(H^2(\a - 2\mu_0) \ominus \fb_{\z_1}\fb_{\overline{\z_1}}\bigl(H^2(\a - 2(\mu_0 + \mu_1)) \ominus \cdots\bigr)\right).
\end{align*}
As in the CMV case, we will be able to act on even steps by a $2\times 2$ block-diagonal operator $M$ to alternate between dual bases respecting the symmetric poles on each two-dimensional subspace $H^2(\a) \ominus \fb_{\z_l}\fb_{\overline{\z_l}}H^2(\a - 2\mu_l)$.  However -- and this is the difference relative to CMV matrices -- in the odd steps we wish to pass from the pole $\overline{\z_l}$ to the new pole $\z_{l+1}$.  Of course, since $K^\a_{\z_k} \notin H^2(\a) \ominus \fb_{\z_l}\fb_{\overline{\z_l}}H^2(\a - 2\mu_l)$ when $\z_k \neq \z_l$, something new is required to perform this shift.  In this sense, the fundamental lemma allowing for our analysis is the following simple realization:
\begin{lemma}\label{lem:poleshift1}
For any $\a \in \G^*$ and with $z_l, z_k\in\D$ and $\z_l, \z_k$ as above, we have
\begin{align}
\frac{\fz - z_k}{\fz- z_l}\fb_{\z_l}K^{\a - \mu_l}_{\overline{\z_k}} \in H^2(\a) \ominus \fb_{\z_l}\fb_{\overline{\z_l}}H^2(\a - 2\mu_l).
\end{align}
\end{lemma}
\begin{proof}
For $f \in H^2(\a - 2\mu_l)$, we have
\begin{align*}
\frac{1 - \overline{z_k}\fz}{1 - \overline{z_l}\fz}\fb_{\overline{\z_l}}f \in H^2(\a - \mu_l).
\end{align*}
Since $\fb_{\z_l}$ is unimodular on the boundary, the adjoint in $H^2$ of multiplication by $\fb_{\z_l}$ (and consequently $\fz$) is multiplication by $\fb_{\z_l}^{-1}$ (respectively $\fz^{-1}$).  Thus, by computing adjoints and applying the reproducing property, one has
\begin{align*}
\left\langle \fb_{\z_l}\fb_{\overline{\z_l}}f, \frac{\fz - z_k}{\fz- z_l}\fb_{\z_l}K^{\a - \mu_l}_{\overline{\z_k}} \right\rangle = \left\langle \frac{1 - \overline{z_k}\fz}{1 - \overline{z_l}\fz}\fb_{\overline{\z_l}}f, K^{\a - \mu_l}_{\overline{\z_k}} \right\rangle = 0,
\end{align*}
as claimed.
\end{proof}

Now the way ahead is clear: we apply Lemma \ref{lem:poleshift1} to expand the shifted reproducing kernel in terms of the reproducing kernels for the previous pole.  Let $\a \in \G^*$ and $\z_l, \z_k \in \D$ be as above and define
\begin{align}
\label{eq:c1}
c^\a_1(\z_k,\z_l) &= e^{-i\phi_l}\frac{z_k - z_l}{1-|z_l|^2}\frac{K^\a(\z_k,\z_l)}{\fb_{\z_l}(\z_k)K^{\a - \mu_l}(\z_k,\z_k)}, \\
\label{eq:c2}
c^\a_2(\z_k,\z_l) &= -c^\a_1(\z_k,\z_l)\frac{K^\a(\z_k,\overline{\z_l})}{K^\a(\z_k,\z_l)},
\end{align}
where at the removable singularity $\z_k = \z_l$ we take
\begin{align}
c_1^\a(\z_l,\z_l) &= \frac{K^\a(\z_l,\z_l)}{\fb_{\overline{\z_l}}(\z_l)K^{\a - \mu_l}(\z_l,\z_l)},
\end{align}
cf. \eqref{eq:BlaschkeUniformization}.  Then
\begin{lemma}\label{lem:poleshift2}
\begin{align}
\label{eq:poleshift1}
\frac{\fz - z_k}{\fz- z_l}\fb_{\z_l}K^{\a - \mu_l}_{\overline{\z_k}} &= c_1^\a(\z_k,\z_l)K^\a_{\overline{\z_l}} + c_2^\a(\z_k,\z_l)K^\a_{\z_l}, \\
\label{eq:poleshift2}
\frac{1-\overline{z_k}\fz}{1-\overline{z_l}\fz}\fb_{\overline{\z_l}}K^{\a - \mu_l}_{\z_k} &= \overline{c_2^\a(\z_k,\z_l)}K^\a_{\overline{\z_l}} + \overline{c_1^\a(\z_k,\z_l)}K^\a_{\z_l}.
\end{align}
\end{lemma}
\begin{proof}
That such a decomposition exists is precisely the content of Lemma \ref{lem:poleshift1} and \eqref{eq:spanKK}.  Since $K^\a_{\z_l}$ is orthogonal to $\fb_{\z_l}K^{\a - \mu_l}_{\z_k}$, we find that the coefficient in front of $K^\a_{\overline{\z_l}}$ in \eqref{eq:poleshift1} is given by
\begin{align*}
\frac{\left\langle \frac{\fz - z_k}{\fz- z_l}\fb_{\z_l}K^{\a - \mu_l}_{\overline{\z_k}}, \fb_{\z_l}K^{\a - \mu_l}_{\z_k} \right\rangle}{\langle K^\a_{\overline{\z_l}}, \fb_{\z_l}K^{\a - \mu_l}_{\z_k} \rangle}.
\end{align*}
Using Lemma \ref{lem:EvaluatePole} and \eqref{eq:BlaschkeUniformization}, the numerator can be written as
\begin{align*}
	\left\langle \frac{\fz - z_k}{\fz- z_l}K^{\a - \mu_l}_{\overline{\z_k}}, K^{\a - \mu_l}_{\z_k} \right\rangle &= -\frac{(z_l-z_k)K^{\a - \mu_l}(\z_l,\overline{\z_k})}{K^{\a}(\z_l,\z_l)} \bigg(\frac{\fb_{\z_l}}{\fz- z_l}\bigg)(\z_l)\frac{K^\a(\z_k,\z_l)}{\fb_{\z_l}(\z_k)K^{\a-\mu_l}(\z_k,\z_k)}\\
	&=e^{-i\phi_l}\frac{z_k-z_l}{1-|z_l|^2}\frac{\fb_{\overline{\z_l}}(\z_l) K^{\a - \mu_l}(\z_l,\overline{\z_k})}{K^{\a}(\z_l,\z_l)}\frac{K^\a(\z_k,\z_l)}{\fb_{\z_l}(\z_k)K^{\a-\mu_l}(\z_k,\z_k)}
\end{align*}
and we thus arrive at the expression for $c_1^\a(\z_k,\z_l)$ in \eqref{eq:c1}. Plugging in $\z_k$ to \eqref{eq:poleshift1} makes the left-hand side vanish and we deduce that the coefficient in front of $K^\a_{\z_l}$ is given by $c_2^\a(\z_k,\z_l)$ as in \eqref{eq:c2}.  Equation \eqref{eq:poleshift2} follows by applying the operation $f(\z)\mapsto \overline{f(\overline{\z})}$ to \eqref{eq:poleshift1}.
\end{proof}
Since the decomposition in the previous lemma is not orthogonal, we do not immediately get a nice Pythagorean identity; however, if we define
\begin{align}
\eta_l^2 := 1 - |z_l|^2,
\end{align}
then we have that
\begin{lemma}\label{lem:pythagoreanid}
\begin{align}\label{eq:pythid}
|K^{\a}(\z_{k},\z_l)|^2+|\fb_{\z_l}(\z_{k})K^{\a-\mu_l}(\z_{k},\z_k)|^2=|K^{\a}(\z_{k},\overline{\z_l})|^2+|\fb_{\overline{\z_l}}(\z_{k})K^{\a-\mu_l}(\z_{k},\z_k)|^2.
\end{align}
In particular, for $c_1^\a$, $c_2^\a$ defined in \eqref{eq:c1}--\eqref{eq:c2},
\begin{align}\label{eq:detRel}
|c_1^\a(\z_k,\z_l)|^2 - |c_2^\a(\z_k,\z_l)|^2 &= \eta_k^2/\eta_l^2.
\end{align}
\end{lemma}

\begin{proof}
Note that $k^\a_{\z_k}$ simultaneously lives in both $H^2(\a) \ominus \fb_{\z_k}\fb_{\z_l}H^2(\a - \mu_l - \mu_k)$ and $H^2(\a) \ominus \fb_{\z_k}\fb_{\overline{\z_l}}H^2(\a - \mu_l - \mu_k)$.  Equation \eqref{eq:pythid} follows immediately from the Pythagorean identity after expanding $k^\a_{\z_k}$ in the two orthonormal bases $\{K^\a_{\z_l},\fb_{\z_l}K^{\a-\mu_l}_{\z_{k}}\}$ and $\{K^\a_{\overline{\z_l}},\fb_{\overline{\z_l}}K^{\a-\mu_l}_{\z_{k}}\}$.

It remains to show \eqref{eq:detRel}. With \eqref{eq:BlaschkeUniformization} in mind, we see that \eqref{eq:pythid} is equivalent to
\begin{align*}
 \frac{|K^\a(\z_{k},\z_l)|^2-|K^\a(\z_{k},\overline{\z_l})|^2}{|\fb_{\overline{\z_l}}(\z_{k})K^{\a-\mu_l}(\z_{k},\z_{k})|^2}=1-|b_{z_l}(z_{k})|^2.
\end{align*}
A simple calculation shows that
\begin{align} \label{eq:etaRelation}
 1-|b_{z_l}(z_k)|^2 = \frac{\eta_l^2 \eta_k^2}{|1-\overline{z_l}z_k|^2} 
\end{align}
and \eqref{eq:BlaschkeUniformization} implies
\begin{align*}
\frac{|z_k - z_l|^2}{|\fb_{\z_l}(\z_k)|^2} = \frac{|1-\overline{z_l}z_k|^2}{|\fb_{\overline{\z_l}}(\z_k)|^2}.
\end{align*}
Thus we have
\begin{align*}
|c_1^\a(\z_k,\z_l)|^2 - |c_2^\a(\z_k,\z_l)|^2 = \frac{|1-\overline{z_l}z_k|^2}{\eta_l^4}\frac{|K^\a(\z_{k},\z_l)|^2-|K^\a(\z_{k},\overline{\z_l})|^2}{|\fb_{\overline{\z_l}}(\z_{k})K^{\a-\mu_l}(\z_{k},\z_{k})|^2} = \frac{\eta_k^2}{\eta_l^2},
\end{align*}
	as claimed.
\end{proof}

Combining all of the above results, we arrive at
\begin{proposition}
\begin{multline}
\label{eq:ThetamatOdd}
\qquad \fz_l^{-1}\begin{bmatrix}
\fz - z_l & 0 \\
0 & \fz - z_k
\end{bmatrix}
\begin{bmatrix}
K_{\z_l}^\a \\
\fb_{\z_l}K^{\a - \mu_l}_{\overline{\z_k}}
\end{bmatrix} \\
=
\frac{1}{\,\overline{c_1^\a(\z_k,\z_l)}\,}\begin{bmatrix}
-\overline{c_2^\a(\z_k,\z_l)} & 1 \\
\eta_k^2/\eta_l^2 & c_2^\a(\z_k,\z_l)
\end{bmatrix}
\begin{bmatrix}
1- \overline{z_l}\fz & 0\\
0 & 1-\overline{z_k}\fz
\end{bmatrix}
\begin{bmatrix}
K_{\overline{\z_l}}^\a \\
\fb_{\overline{\z_l}}K^{\a - \mu_l}_{\z_k}
\end{bmatrix}. \qquad
\end{multline}
When $\z_l = \z_k$, this simplifies to
\begin{align}\label{eq:ThetamatEven}
\begin{bmatrix}
K_{\z_l}^\a \\
\fb_{\z_l}K^{\a - \mu_l}_{\overline{\z_l}}
\end{bmatrix}
=
\frac{1}{\,\overline{c_1^\a(\z_l,\z_l)}\,}\begin{bmatrix}
-\overline{c_2^\a(\z_l,\z_l)} & 1 \\
1 & c_2^\a(\z_l,\z_l)
\end{bmatrix}
\begin{bmatrix}
K_{\overline{\z_l}}^\a \\
\fb_{\overline{\z_l}}K^{\a - \mu_l}_{\z_l}
\end{bmatrix}.
\end{align}
\end{proposition}
\begin{proof}
Multiplying the identity \eqref{eq:ThetamatOdd} through by $(1 - \overline{z_l}\fz)^{-1}$, the first line of the identity is simply \eqref{eq:poleshift2}.  The second line follows from \eqref{eq:poleshift1}, the first line, and an application of \eqref{eq:detRel}.
\end{proof}

We now have all the tools to show that multiplication by $\fz$ in $L^2(\a)$ has the appropriate structure.  Fix $(\a,\tau) \in \G^* \times \T$ and define the following quantities
\begin{align}
A(\a,\tau;\z_k,\z_l) &:= -e^{-i\tau}\frac{c_2^\a(\z_k,\z_l)}{c_1^\a(\z_k,\z_l)} = e^{-i\tau}\frac{K^\a(\z_k, \overline{\z_l})}{K^\a(\z_k,\z_l)},\\
\label{eq:RfunctionMCMV}
R(\a;\z_k,\z_l) &:= \frac{1}{|c_1^\a(\z_k,\z_l)|}\frac{\eta_k}{\eta_l} = \sqrt{1-|A(\a,\tau;\z_k,\z_l)|^2},
\end{align}
and
\begin{align}
\Theta(\a,\tau;\z_k,\z_l) := \begin{bmatrix}
\,\overline{A(\a,\tau;\z_k,\z_l)} & R(\a;\z_k,\z_l) \\
\,R(\a;\z_k,\z_l) & -A(\a,\tau;\z_k,\z_l)
\end{bmatrix}.
\end{align}
Define also
\begin{align}\label{def:thetal}
\omega^\a_{k,l} := \arg\bigl(c_1^\a(\z_k,\z_l)\bigr)
\end{align}
and note that $\omega^\a_{l,l} = 0$ due to our normalization $\fb_{\overline{\z_l}}(\z_l) > 0$.  Then the content of the previous proposition is that, considering the cases $\z_k = \z_l$ and $\z_k = \z_{l+1}$, respectively,
\begin{align}\label{eq:ThetaMatEven2}
\begin{bmatrix}
e^{i\tau}K_{\z_l}^\a \\
\fb_{\z_l}K^{\a - \mu_l}_{\overline{\z_l}}
\end{bmatrix}
=
\Theta(\a,\tau;\z_l,\z_l)
\begin{bmatrix}
K_{\overline{\z_l}}^\a \\
e^{i\tau}\fb_{\overline{\z_l}}K^{\a - \mu_l}_{\z_l}
\end{bmatrix}
\end{align}
and
\begin{multline}\label{eq:ThetaMatOdd2}
\qquad e^{-i\phi_l}
\begin{bmatrix}
\frac{\fz - z_l}{\eta_l} & 0 \\
0 & \frac{\fz - z_{l+1}}{\eta_{l+1}}
\end{bmatrix}
\begin{bmatrix}
e^{i\tau}\fb_{\overline{\z_l}}K_{\z_l}^\a \\
e^{-i\omega^\a_{l+1,l}}\fb_{\z_l}\fb_{\overline{\z_l}}K^{\a - \mu_l}_{\overline{\z_{l+1}}}
\end{bmatrix} \\
=
\Theta(\a,\tau;\z_{l+1},\z_l)
\begin{bmatrix}
\frac{1- \overline{z_l}\fz}{\eta_l} & 0\\
0 & \frac{1-\overline{z_{l+1}}\fz}{\eta_{l+1}}
\end{bmatrix}
\begin{bmatrix}
\fb_{\z_l}K_{\overline{\z_l}}^\a \\
e^{i\tau}e^{i\omega^\a_{l+1,l}}\fb_{\z_l}\fb_{\overline{\z_l}}K^{\a - \mu_l}_{\z_{l+1}}
\end{bmatrix}. \qquad
\end{multline}

We are finally ready to establish our basis.  Let
\begin{align*}
\a_l := \a_{l-1} - 2\mu_l, \quad \a_0 := \a - \mu_0, \quad
\vartheta_l^\a := \sum_{j=0}^{l} \omega_{j+1,j}^{\a_{j}+\mu_{j}} + \phi_j,
\end{align*}
and
\begin{align*}
\mathfrak{B}_l := \prod_{j=1}^l \fb_{\z_j}, \quad
\mathfrak{B}_l^\ast := \prod_{j=1}^l \fb_{\overline{\z_j}}, \quad
\mathfrak{B}_0 = \mathfrak{B}_0^\ast = 1.
\end{align*}
Taking as convention $\vartheta^\a_{-1} = 0$, $\mathfrak{B}_{-1} = \fb_{\z_0}^{-1}$, $\mathfrak{B}_{-1}^\ast = \fb_{\overline{\z_0}}^{-1}$, and $\z_{n} = \z_0$, we define for $0\leq l \leq n-1$ the functions
\begin{align}\label{def:Basis}
x_{2l}^{\a,\tau} := e^{-i\vartheta_{l-1}^\a}\fb_{\overline{\z_0}}\mathfrak{B}_l\mathfrak{B}_{l-1}^\ast K_{\overline{\z_l}}^{\a_{l-1} - \mu_l},  \quad x_{2l+1}^{\a,\tau} := e^{i\tau}e^{i\vartheta_{l}^\a}\fb_{\overline{\z_0}}\mathfrak{B}_{l}\mathfrak{B}_{l}^\ast K_{\z_{l+1}}^{\a_{l}}, \\
\label{def:Basis2}
\chi_{2l}^{\a,\tau} := e^{i\tau}e^{i\vartheta_{l-1}^\a}\fb_{\z_0}\mathfrak{B}^*_l\mathfrak{B}_{l-1}K_{\z_l}^{\a_{l-1}-\mu_l}, \quad \chi_{2l+1}^{\a,\tau} := e^{-i\vartheta_l^\a}\fb_{\z_0}\mathfrak{B}_l^*\mathfrak{B}_lK^{\a_l}_{\overline{\z_{l+1}}}.
\end{align}
In analog to the CMV case \eqref{eq:CMVKspace}, they form two different bases of the $2n$ dimensional subspace $\cK_{\fB\fB^*}$; cf. \eqref{eq:exhaust}.
Letting $p = 2n$, we extend this family of functions (for $j\in\Z$) by
\begin{align}\label{def:extensionBasis}
x^{\a,\tau}_{2l+jp} = (\fB\fB^*)^je^{-ij\vartheta_{n-1}^\a}x_{2l}^{\a,\tau}, \quad x^{\a,\tau}_{2l+1+jp} = (\fB\fB^*)^je^{ij\vartheta_{n-1}^\a}x_{2l+1}^{\a,\tau}, \\\label{def:extensionBasis2}
\chi^{\a,\tau}_{2l+jp} = (\fB\fB^*)^je^{ij\vartheta_{n-1}^\a}\chi^{\a,\tau}_{2l}, \quad \chi^{\a,\tau}_{2l+1+jp} = (\fB\fB^*)^je^{-ij\vartheta_{n-1}^\a}\chi^{\a,\tau}_{2l+1}.
\end{align}
By iterating the exhaustion \eqref{eq:exhaust}, it isn't difficult to see that the systems of functions $\{x_k^{\a,\tau}\}_{k \in I}$ and $\{\chi_k^{\a,\tau}\}_{k\in I}$ each form an orthonormal basis of $H^2(\a)$ for $I=\N$. By \cite[Lemma 3.5]{Eich16}, it follows that they  also form a basis of $L^2(\a)$ when $I=\Z$.

Let us comment on the meaning of the unimodular constants appearing in the definitions above. First of all, we can choose the unimodular constant freely in the normalization of the basis functions $\{x_0,x_1\}$. This explains the meaning of the additional parameter $\tau$. Once this normalization is fixed, the normalization of the following basis functions is already determined: comparing \eqref{eq:ThetamatOdd} and \eqref{eq:ThetaMatOdd2}, we see that -- apart from the additional parameter $\tau$ -- the main difference between the constant matrix on the right-hand side of \eqref{eq:ThetamatOdd} and the matrix $\Theta(\a,\tau;\z_l,\z_{l+1})$ in \eqref{eq:ThetaMatOdd2} is that the latter has positive off-diagonal entries. This has been achieved by adding the phase $e^{i\omega^\a_{l+1,l}}$ to the reproducing kernels.  These phases accumulate with each step as the phases $e^{i\vartheta_{l}^\a}$.  

Define now the periodic up to a phase Verblunsky coefficients $\{a_k(\a,\tau;\vz)\}$ by
\begin{align}
\left.\begin{aligned}
a_{2l-1}(\a,\tau;\vz) &:= e^{-2i\vartheta^\a_{l-1}}A(\a_{l-1},\tau;\z_{l},\z_l), \\
a_{2l}(\a,\tau;\vz) &:= e^{-i(\phi_l + 2\vartheta_{l-1}^\a)}A(\a_{l-1} - \mu_l,\tau;\z_{l+1},\z_{l}),
\end{aligned}
\right. \quad\quad 0 \leq l \leq n-1,
\end{align}
and
\begin{align}
a_{m+jp}(\a,\tau;\vz) &= e^{-2ij\vartheta^\a_{n-1}}a_m(\a,\tau;\vz), \quad -1\leq m\leq 2n-2, \; \; j \in \Z, \\
\rho_k(\a;\vz) &:= \sqrt{1-|a_k(\a,\tau;\vz)|^2}, \quad k\in\Z.
\end{align}
Then multiplication by $\fz$ in our modified basis is represented by an MCMV matrix with the above parameters:
\begin{theorem}\label{thm:multiplicationMCMV}
Let $C = C(\a,\tau;\vz)$ be the periodic up to a phase CMV matrix with Verblunsky coefficients $a_k(\a,\tau;\vz)$ and let $\DD$ be the $2n$-periodic diagonal matrix given by \eqref{eq:perdiag}.  Then with respect to the basis $\{{\chi}^{\a,\tau}_k\}_{k\in\Z}$ of $L^2(\a)$, multiplication by $\fz$ is represented by the MCMV matrix $b_{\minus \DD}(C)$. 
\end{theorem}

\begin{proof}
Denote by
\begin{align}
\Theta_k(\a,\tau;\vz) := \begin{bmatrix}
\, \overline{a_k(\a,\tau;\vz)} & \rho_k(\a;\vz) \\
\, \rho_k(\a;\vz) & -a_k(\a,\tau;\vz)
\end{bmatrix}.
\end{align}
We use liberally the following two simple observations: that diagonal matrices commute, and that, for $a \in \D$, $\rho = \sqrt{1-|a|^2}$, and $\theta_0, \theta_1 \in \R/2\pi\Z$,
\begin{align*}
\begin{bmatrix}
e^{-i\theta_0} & 0 \\
0 & e^{i\theta_1}
\end{bmatrix}
\begin{bmatrix}
\overline{a} & \rho \\
\rho & -a
\end{bmatrix}
\begin{bmatrix}
e^{-i\theta_1} & 0 \\
0 & e^{i\theta_0}
\end{bmatrix}
&=
\begin{bmatrix}
e^{-i(\theta_0 + \theta_1)}\overline{a} & \rho \\
\rho & -e^{i(\theta_0 + \theta_1)}a
\end{bmatrix}.
\end{align*}
Let $0 \leq l \leq n-1$.  With the above facts in hand, it is then clear that \eqref{eq:ThetaMatEven2} is equivalent to
\begin{align*}
\begin{bmatrix}
e^{i\tau}e^{i\vartheta_{l-1}^\a}K_{\z_l}^{\a_{l-1}} \\
e^{-i\vartheta_{l-1}^\a}\fb_{\z_l}K^{\a_{l-1}-\mu_l}_{\overline{\z_l}}
\end{bmatrix}
=
\Theta_{2l-1}(\a,\tau;\vz)
\begin{bmatrix}
e^{-i\vartheta_{l-1}^\a}K_{\overline{\z_l}}^{\a_{l-1}} \\
e^{i\tau}e^{i\vartheta_{l-1}^\a}\fb_{\overline{\z_l}}K^{\a_{l-1}-\mu_l}_{\z_l}
\end{bmatrix}.
\end{align*}
Multiplying both sides by $\fb_{\z_0}\mathfrak{B}_{l-1}\mathfrak{B}_{l-1}^\ast = e^{-i\phi_0}\fz\fb_{\overline{\z_0}}\mathfrak{B}_{l-1}\mathfrak{B}_{l-1}^\ast$, we get that
\begin{align}\label{eq:ThetaMatEven3}
e^{-i\phi_0}\fz \begin{bmatrix}
x_{2l-1}^{\a,\tau} \\
x_{2l}^{\a,\tau}
\end{bmatrix}
&=
\Theta_{2l-1}(\a,\tau;\vz)
\begin{bmatrix}
\chi_{2l-1}^{\a,\tau} \\
\chi_{2l}^{\a,\tau}
\end{bmatrix}.
\end{align}
Since in the context of \eqref{eq:ThetaMatOdd2} we have $e^{i\phi_l} = e^{i\phi_l}I$ (where $I$ is the $2 \times 2$ identity matrix), that equation can also be written as
\begin{align*}
\begin{bmatrix}
\frac{\fz - z_l}{\eta_l} & 0 \\
0 & \frac{\fz - z_{l+1}}{\eta_{l+1}}
\end{bmatrix}
\begin{bmatrix}
e^{i\tau}e^{i\vartheta_{l-1}^\a}\fb_{\overline{\z_l}}K_{\z_l}^{\a_{l-1} - \mu_l} \\
e^{-i\vartheta_l^\a}\fb_{\z_l}\fb_{\overline{\z_l}}K^{\a_l}_{\overline{\z_{l+1}}}
\end{bmatrix}
=
\Theta_{2l}(\a,\tau;\vz)
\begin{bmatrix}
\frac{1- \overline{z_l}\fz}{\eta_l} & 0\\
0 & \frac{1-\overline{z_{l+1}}\fz}{\eta_{l+1}}
\end{bmatrix}
\begin{bmatrix}
e^{-i\vartheta^\a_{l-1}}\fb_{\z_l}K_{\overline{\z_l}}^{\a_{l-1} - \mu_l} \\
e^{i\tau}e^{i\vartheta^\a_l}\fb_{\z_l}\fb_{\overline{\z_l}}K^{\a_l}_{\z_{l+1}}
\end{bmatrix}.
\end{align*}
Multiplying both sides by $\fb_{\overline{\z_0}}\mathfrak{B}_{l-1}\mathfrak{B}_{l-1}^\ast = e^{i\phi_0}\fz^{-1}\fb_{\z_0}\mathfrak{B}_{l-1}\mathfrak{B}_{l-1}^\ast$ and rearranging yields
\begin{align}\label{eq:ThetaMatOdd3}
e^{i\phi_0}\fz^{-1}\begin{bmatrix}
\frac{\fz - z_l}{\eta_l} & 0 \\
0 & \frac{\fz - z_{l+1}}{\eta_{l+1}}
\end{bmatrix}
\begin{bmatrix}
\chi_{2l}^{\a,\tau} \\
\chi_{2l+1}^{\a,\tau}
\end{bmatrix}
=
\Theta_{2l}(\a,\tau;\vz)
\begin{bmatrix}
\frac{1- \overline{z_l}\fz}{\eta_l} & 0\\
0 & \frac{1-\overline{z_{l+1}}\fz}{\eta_{l+1}}
\end{bmatrix}
\begin{bmatrix}
x_{2l}^{\a,\tau} \\
x_{2l+1}^{\a,\tau}
\end{bmatrix}.
\end{align}
Extending to all $l$ follows similarly from the definitions.

Denote now by $\DD := \DD(\vz)$ the $2n$-periodic diagonal matrix in \eqref{eq:perdiag}, let $\eta_{\DD} = \sqrt{1-\DD\DD^*}$, and fix
\begin{align*}
L := \bigoplus_{l \in \Z} \Theta_{2l}(\a,\tau;\vz), \quad
M := \bigoplus_{l \in \Z} \Theta_{2l+1}(\a,\tau;\vz),
\end{align*}
where $\Theta_k$ acts on the two-dimensional subspace $\{\delta_k, \delta_{k+1}\}$.  Combining the statements \eqref{eq:ThetaMatEven3} and \eqref{eq:ThetaMatOdd3} above, we have shown the following:
\begin{align*}
L(1-\fz \DD^*)\eta_{\DD}^{-1} M\vec{\, \chi}^{\a,\tau} = \eta_{\DD}^{-1}(\fz - {\DD})\vec{\, \chi}^{\a,\tau},
\end{align*}
where $\vec{\, \chi}^{\a,\tau}$ is shorthand notation for the vector $(\chi^{\a,\tau}_k)_{k\in\Z}$.
Since the operators $\eta_{\DD}^{-1}$, $\fz - \DD$, and $1-\fz {\DD}^*$ commute with $M$ (for they are orthogonal sums of multiples of $I$ along the odd terms), taking $C = LM$ we have
\begin{align}\label{eq:MCMV1}
(1 - \eta_{\DD}(\fz - {\DD})^{-1}C(1 - \fz {\DD}^*)\eta_{\DD}^{-1})\vec{\, \chi}^{\a,\tau} = 0
\end{align}
which can be rearranged as
\begin{align*}
(\fz - \eta_{\DD}(1+C{\DD}^*)^{-1}(C+{\DD})\eta_{\DD}^{-1})\vec{\, \chi}^{\a,\tau} = 0.
\end{align*}
Thus, in the basis $\{\chi^{\a,\tau}_k\}$, multiplication by $\fz$ is given by $b_{\minus {\DD}}(C)$.
\end{proof}

Of course, \eqref{eq:ThetamatOdd} and \eqref{eq:ThetamatEven} in combination with the exhaustion \eqref{eq:exhaust} without shifted character imply a transfer matrix relation in terms of the reproducing kernels.  To see explicitly this relation, first note that, denoting as shorthand $a_k = a_k(\a,\tau;\vz)$, $\rho_k = \rho_k(\a;\vz)$, and $\cU(a_k)$ as in \eqref{eq:cUmat}, we can rewrite \eqref{eq:ThetaMatEven3}--\eqref{eq:ThetaMatOdd3} in the following way:
\begin{align}
\label{eq:transfMatOdd}
\begin{bmatrix}
e^{-i\phi_0}\fz & 0 \\
0 & 1
\end{bmatrix}
\cU(a_{2l+1})&
\begin{bmatrix}
1 & 0 \\
0 & e^{i\phi_0}\fz^{-1}
\end{bmatrix}
\begin{bmatrix}
x^{\a,\tau}_{2l+2} \\
\chi^{\a,\tau}_{2l+2}
\end{bmatrix}
=
\begin{bmatrix}
\chi_{2l+1}^{\a,\tau} \\
x_{2l+1}^{\a,\tau}
\end{bmatrix}, \\
\label{eq:transfMatEven}
\frac{1-\overline{z_{l+1}}\fz}{1-\overline{z_l}\fz}\frac{\eta_l}{\eta_{l+1}}
\begin{bmatrix}
1 & 0\\
0 & \fz_{l}^{-1}
\end{bmatrix}
\begin{bmatrix}
e^{i\phi_0}\fz^{-1} & 0\\
0 & 1
\end{bmatrix}
\cU&(a_{2l})
\begin{bmatrix}
1 & 0\\
0 & e^{-i\phi_0}\fz
\end{bmatrix}
\begin{bmatrix}
\fz_{l+1} & 0\\
0 & 1
\end{bmatrix}
\begin{bmatrix}
\chi^{\a,\tau}_{2l+1} \\
x^{\a,\tau}_{2l+1}
\end{bmatrix}
=
\begin{bmatrix}
x^{\a,\tau}_{2l} \\
\chi^{\a,\tau}_{2l}
\end{bmatrix}.
\end{align}
Again using the notation that $z_{n} = z_0 = 0$, $\eta_{n} = \eta_0$, etc., and denoting
\begin{multline}
\cM(z;\vz,\{a_k\}) :=
\cU(a_0)
\begin{bmatrix}
b_{z_1}(z) & 0 \\
0 & 1
\end{bmatrix}
\cU(a_1)
\begin{bmatrix}
b_{z_1}(z) & 0 \\
0 & 1
\end{bmatrix}
\cU(a_2)
\begin{bmatrix}
b_{z_2}(z) & 0 \\
0 & 1
\end{bmatrix}
\cdots  \\
\cdots
\cU(a_{2n-1})
\begin{bmatrix}
b_{n-1}(z) & 0 \\
0 & 1
\end{bmatrix}
\cU(a_{2n})
\begin{bmatrix}
z & 0 \\
0 & 1
\end{bmatrix}
\cU(a_{2n+1})
\begin{bmatrix}
z & 0 \\
0 & 1
\end{bmatrix}
\begin{bmatrix}
e^{-i\vartheta_{n-1}^{\a}} & 0 \\
0 & e^{i\vartheta_{n-1}^{\a}}
\end{bmatrix}
\end{multline}
and
\begin{align}
B(z)=z\prod_{j=1}^{n-1}b_{z_j}(z)=\sqrt{\det\cM(z;\vz,\{a_k\})},
\end{align}
we arrive at the following monodromy relation:
\begin{theorem}\label{thm:monodromy}
The reproducing kernels satisfy the phased monodromy relation
\begin{align}\label{eq:monodromy}
\frac{1}{B(\fz)}\cM(\fz;\vz,\{a_k\})
\begin{bmatrix}
x_0^{\a,\tau} \\
e^{i\phi_0}\chi_0^{\a,\tau}
\end{bmatrix}
&= (\fB\fB^*)^{-1}\begin{bmatrix}
x_0^{\a,\tau} \\
e^{i\phi_0}\chi_0^{\a,\tau}
\end{bmatrix}.
\end{align}
\end{theorem}
\begin{proof}
This follows from iterating \eqref{eq:transfMatOdd}--\eqref{eq:transfMatEven} over a full period of size $p = 2n$, since the multiplier terms telescope and
\[
x_{2n}^{\a,\tau} = \fB\fB^* e^{-i\vartheta_{n-1}^{\a}}x_0^{\a,\tau}, \quad
\chi_{2n}^{\a,\tau} = \fB\fB^* e^{i\vartheta_{n-1}^{\a}}\chi_0^{\a,\tau}.
\qedhere
\]
\end{proof}
\begin{remark}
In terms of projective lines, \eqref{eq:monodromy} is, up to a phase, precisely the relation \eqref{eq:nevPickSchur} for the Schur functions \eqref{eq:functSchur}.
\end{remark}

We are now ready to give a detailed explanation for introducing the matrix $\L(\vartheta)$ in Definition \ref{def:MCMV}. In our extension \eqref{def:extensionBasis2} of the vectors $\{y_l^{\a,\tau}\}_{l=0}^{2n-1}$ to a basis of $L^2(\a)$, we added a phase $e^{i\vartheta_{n-1}^\alpha}$ in order to represent the multiplication operator by $\fz$ as an operator M\"obius transform of a CMV matrix $C(\a,\tau;\vz)$; the phase was needed to have the off-diagonal entries of $\Theta_{l}(\alpha,\tau;\vz)$ in its $LM$ factorization positive. The price we paid is that the corresponding matrix is merely periodic up to a phase. If we had chosen the extension of multiplying by $(\fB\fB^*)^j$ without the phase, then the corresponding operator would have been periodic. This alternative basis, say $\{y^{\alpha,\tau}_{\per, l}\}$, is related to $\{y_l^{\a,\tau}\}$ in the following way:
\begin{align*}
	\vec{\, \chi}_{\per}^{\a,\tau}=\Lambda(\vartheta_{n-1}^{\a})^*\vec{\, \chi}^{\a,\tau}.
\end{align*}
\indent To sum up, we have obtained a map from $\Gamma^*\times\T$ to $\cT_\MV(\E)$:
\begin{corollary}\label{cor:MagicFormulafunc}
	Let $C, \DD$ be as in Theorem \ref{thm:multiplicationMCMV} and set
	\begin{align}\label{def:Aalpha}
		A:=A({\a,\tau})=\Lambda(\vartheta_{n-1}^{\a})^*b_{\minus {\DD}}\bigl(C(\a,\tau)\bigr)\Lambda(\vartheta_{n-1}^{\a}).
	\end{align}
	Then $A \in \A_{\per}(\vz)$ and $\sigma(A) = \E$.  Moreover, with $\Delta_A$ as in \eqref{def:discriminant},
	\begin{align}\label{eq:MagicFormula1}
		\Delta_{A}\circ\fz=\fB\fB^*+\frac{1}{\fB\fB^*}
	\end{align}
	and
	\begin{align}\label{eq:MagicFormula2}
		\Delta_{A}(A)=S^{2n}+S^{-2n}.
	\end{align}
	In particular, in the special case $\vz = \vz_\E$ we have $A \in \cT_{\MV}(\E)$ and $\Delta_A = \Delta_\E$.
\end{corollary}
\begin{proof}
	The first statement follows from the discussion above. The fact that $\sigma(A)=\E$ is clear since $A$ is the matrix of multiplication by $\fz$. If we set $\tilde{\cM}=B^{-1}\cM$, then Theorem \ref{thm:monodromy} states that $(\fB\fB^*)^{-1}$ is an eigenvalue of $\tilde \cM$. As $\det\tilde\cM=1$, we obtain \eqref{eq:MagicFormula1}. Finally, \eqref{eq:MagicFormula2} is a direct consequence of the fact that multiplication by $\fB\fB^*$ corresponds to the action of $S^{2n}$ in the basis $\{{\chi}_{\per, l}^{\a,\tau}\}_{l\in\Z}$.
\end{proof}

%% file: DST-28-09.tex
\section{Direct spectral theory}\label{sec:DST}

In the previous section, we saw that the functional model developed by Peherstorfer and Yuditskii to represent finite-gap almost-periodic CMV matrices has corresponding representations as periodic MCMV matrices.  In this section, we develop the necessary tools to address the converse: that any periodic MCMV matrix arises from such a functional model.

For periodic CMV matrices $\apC \in \cT_\mv(\E)$, the bijective nature of this correspondence is by now classical (cf. \eqref{fig:1} below); we recall the elements of its construction in Section \ref{sec:DSTCMV}.  A key component of this correspondence is a set of spectral data associated to the one-sided restriction $\apC_+$, called the \emph{divisor} or \emph{Dirichlet data}, which, together with the discriminant, allows one to uniquely recover the spectral measure and hence the operator $\apC_+$.  We will adapt this construction to periodic MCMV matrices in Section \ref{sec:DirectMCMV}, culminating in the uniqueness statement Proposition \ref{prop:uniqueDiv}.  Finally, Section \ref{sec:structure} explores the block structure \eqref{eq:Astructure} of general MCMV matrices and its invariance under certain M\"obius transformations.

%

\subsection{The isospectral torus of periodic CMV matrices}\label{sec:DSTCMV}
Let $\{a_k\}_{k \in\Z}$ be a periodic sequence with even period $p=2n$ and let $\apC$ be the corresponding whole-line CMV matrix. If we define the discriminant by
\begin{align*}
\Delta_\apC(z)=\tr\bigg(\cU(a_{0})\begin{bmatrix}
z & 0 \\
0 & 1
\end{bmatrix}
\cU(a_{1})
\begin{bmatrix}
1 & 0           \\
0 & \frac{1}{z}
\end{bmatrix}
\cdots
\cU(a_{p-2})
\begin{bmatrix}
z & 0           \\
0 & 1
\end{bmatrix}
\cU(a_{p-1})
\begin{bmatrix}
1 & 0           \\
0 & \frac{1}{z}
\end{bmatrix}
\bigg),
\end{align*}
then the spectrum of $\apC$ is given by
\[
 \E:=\sigma(\apC)=\Delta_{\apC}^{-1}\bigl([-2,2]\bigr).
\]
This spectrum is purely absolutely continuous and of multiplicity two. Moreover, there are $p$ critical points $\{c_i\}_{i=1}^p$ on $\partial \D$ (i.e., zeros of  $\Delta_{\apC}^\prime$)  which all satisfy $|\Delta_\apC(c_i)|\geq 2$. Therefore, the set $\Delta_\apC^{-1}\bigl((-2,2)\bigr)$ can be partitioned into $p$ non-intersecting open arcs. The connected components of the complement of this set on $\partial\D$ are called the gaps. If a gap consists only of a single point (which is the case if $|\Delta_{\apC}(c_i)|=2$), we refer to it as a closed gap. Otherwise the gap is called open. Let $g+1$ denote the number of open gaps and let us fix a labeling of the open gaps. That is, let $\l_j^-,\l_j^+$ denote the gap edges of an open gap such that you can pass from $\l_j^-$ to $\l_j^+$ by traversing the gap counterclockwise. Moreover, let $[\l_j^-,\l_j^+]$ denote the closed arc induced by this order.

As we have seen, the isospectral torus  $\cT_{\mv}(\E)$ is a $g+1$-dimensional torus. In particular, the spectrum does not uniquely determine the operator $\apC$. In order to get the full spectral data to solve the inverse problem, we consider the half-line operator $\apC_+$ with spectral measure $\nu$. One can show there are explicit rational functions $u,v$ such that, for a suitable branch of the square root, the associated Caratheodory function is given by
\begin{align}\label{eq:CarathCMV}
F_\nu(z)=\frac{v(z)+\sqrt{\Delta_{\apC}^2(z)-4}}{u(z)}.
\end{align}
It is known that $u$ has precisely one zero in each gap of $\E$, and if $u(z)=0$, then $\sqrt{\Delta_{\apC}^2(z)-4}$ is either $-v(z)$ or $v(z)$.  A zero of $u$ for which the numerator in \eqref{eq:CarathCMV} does not vanish corresponds to an eigenvalue of $\apC_+$.  Note that for closed gaps, the numerator always vanishes. Let $\{x_j\}_{j = 0}^g$ be the set of all zeros of $u$ which lie in open gaps, and let us write $(x_j,1)$ if $x_j$ is an eigenvalue of $\apC_+$ and $(x_j,-1)$ otherwise. Then the spectrum together with the divisor $D=\{(x_j,\e_j)\}_{j=0}^g$ form the full spectral data and determine $\apC$ completely. In fact, if we define
\begin{align}\label{def:Divisors}
\cD(\E)=\{(x_j,\e_j): \; x_j\in[\l_j^-,\l_j^+],\; \e_j=\pm 1, \; 0\leq j\leq g\}/\sim
\end{align}
with the identifications $(\l^\pm_j,-1)\sim(\l^\pm_j,1)$, then $\cD(\E)$ equipped with the product topology of circles is homeomorphic to $\G^*\times \T$ and hence also to $\cT_{\mv}(\E)$. Inspired by results in the framework of Jacobi matrices \cite{SoYud97}, this has been generalized in \cite{YudPeher06} to the much more general class of Parreau--Widom sets $\E\subset\partial\D$ satisfying the Direct Cauchy Theorem. 
The homeomorphism
\begin{align}\label{def:AbelMao}
\mathfrak A:\cD(\E)\to\G^*\times \T
\end{align}
is called the generalized \textit{Abel map}; it is the map which completes the following diagram:
\begin{equation}
\begin{aligned}\label{fig:1}
\xymatrix{
\cT_{\mv}(\E)\ar[dr]
&&\ar[ll]_{\mbox{Thm\ \ref{t:PY}}}\G^*\times \T\\
&\cD(\E)\ar[ur]_{\mbox{$\mathfrak A$}}&
}
\end{aligned}
\end{equation}
In the finitely connected setting, the Abel map is well understood (see, e.g., \cite{MumTata1,Schlag14}). The connection to spectral theory of Jacobi matrices goes back to Akhiezer \cite{Akh60}; see also \cite{Aptek84,Kri78,MumTata2}.

\subsection{Spectral theory for periodic MCMV matrices}\label{sec:DirectMCMV}
In this section we will perform a spectral analysis for MCMV matrices that are periodic up to a phase. The spectral data will be given by the discriminant and zeros of a certain function which is explicitly defined in terms of the orthogonal rational functions. For CMV matrices it is easy to see that the leading coefficient of the discriminant is positive, and the discriminant is always of maximal degree. For MCMV matrices, however, the situation is more involved; we shall clarify the degree issue in Lemma \ref{lem:Resolvents}.

First we have to clarify what we mean by a half-line MCMV matrix. Recall that given a vector $\vz=\{z_0,\dots,z_{n-1}\}$, an MCMV matrix $A=A(\{a_k\},\vartheta; \vz)\in\A(\vz)$ is defined by
\begin{align*}
		A=\Lambda(\vartheta)^* b_{\minus \DD}(C)\Lambda(\vartheta),
\end{align*}
cf. Definition \ref{def:MCMV}. Given $A\in\A_{\per}(\vz)$ we will study the half-line MCMV matrix
\begin{align}\label{def:A_+}
A_+= b_{\minus D_+}(C_+),
\end{align}
associated to the sequence of Verblunsky coefficients $\{a_k\}_{k=0}^\infty$ and $\vz$.  Specifically, in \eqref{def:A_+}, $C_+=C_+(\{a_k\}_{k=0}^\i)$ is the half-line CMV matrix with Verblunsky coefficients $\{a_k\}_{k=0}^\infty$ and $D_+$ denotes the diagonal operator $D_+ = \diag\{z_0, z_1, z_1, \cdots, z_{n-1}, z_0 | z_0, \cdots\}$.  Since $A\in \A_\per$ and $D_+$ is diagonal, it follows that $C_+$ is periodic up to a phase with phase $e^{-2i\vartheta}$. 
Due to \cite[Theorem 5.4.]{Ve08}, the measure $\nu$ of orthogonality for the family of orthonormal rational functions related to the poles $\{z_0,z_1,z_1,\dots, z_{n-1},z_0|z_0,\dots\}$ and Verblunsky coefficients $\{a_k\}_{k=0}^\i$ is precisely the spectral measure for $A_+$ (and the cyclic vector $\d_0$).  The main result of this section will be an explicit expression for the Caratheodory function $F_\nu$ analogous to \eqref{eq:CarathCMV}.

Let $p=2n$ and fix a sequence $\{a_j\}_{j=0}^{p-1}$ and a phase $\vartheta$. Let $\cM$ be the monodromy matrix defined by \eqref{eq:transferNatrix3}. Moreover, let us define
\begin{align}
\cB_\vartheta(z):=\cB(z)\begin{bmatrix}
\cos\vartheta& i\sin\vartheta\\
i\sin\vartheta&\cos\vartheta
\end{bmatrix},\quad
W_\vartheta(z):=\begin{bmatrix}
e^{-i\vartheta}&0\\
0& e^{i\vartheta}
\end{bmatrix}
W(z),
\end{align}
where $\cB$ and $W=W^{(0)}$ are the matrices defined by \eqref{def:TransferMatrix} and \eqref{def:transferMatrix2}, respectively. The shape of $M_\vartheta$ is such that
\begin{align}
Y_0\begin{bmatrix}
z& 0\\
0& 1
\end{bmatrix}
T(z)=M_\vartheta(z)Y_0
\begin{bmatrix}
z& 0\\
0& 1
\end{bmatrix}.
\end{align}
This follows from the commutant relation
\begin{align}\label{eq:Y0Commutant}
Y_0\begin{bmatrix}
e^{-i\vartheta}& 0\\
0& e^{i\vartheta}
\end{bmatrix}
=
\begin{bmatrix}
\cos\vartheta& i\sin\vartheta\\
i\sin\vartheta&\cos\vartheta
\end{bmatrix}
Y_0.
\end{align}
Hence, if we consider the sequence $\{a_j\}_{j=0}^\i$ which is obtained by extending $\{a_j\}_{j=0}^{p-1}$ in such a way that $a_{j+kp}=e^{-2ik\vartheta}a_j$, then due to Corollary \ref{cor:periodicInterSchur} and \eqref{eq:CarathSchur}, the associated Caratheodory function satisfies
\begin{align}\label{eq:quadraticCara}
\begin{bmatrix}
F_\nu(z)\\
1
\end{bmatrix}
\sim
\cB_\vartheta(z)
\begin{bmatrix}
F_\nu(z)\\
1
\end{bmatrix}.
\end{align}
Using again the simple observation
\begin{align*}
\cU(e^{2i\vartheta}\overline{a})=\begin{bmatrix}
e^{i\vartheta}& 0\\
0& e^{-i\vartheta}
\end{bmatrix}
\cU(\overline{a})
\begin{bmatrix}
e^{-i\vartheta}& 0\\
0& e^{i\vartheta}
\end{bmatrix},
\end{align*}
we find that
\begin{align*}
W^{(k)}(z)=\begin{bmatrix}
e^{ik\vartheta}& 0\\
0& e^{-ik\vartheta}
\end{bmatrix}W^{(0)}(z)\begin{bmatrix}
e^{-ik\vartheta}& 0\\
0& e^{ik\vartheta}
\end{bmatrix}
\end{align*}
and thus
\begin{align}\label{eq:productTransfer}
W^{(k-1)}(z)\cdots W^{(0)}(z)=\begin{bmatrix}
e^{ik\vartheta}& 0\\
0& e^{-ik\vartheta}
\end{bmatrix}
W_\vartheta(z)^k.
\end{align}
Moreover, \eqref{eq:M(z)} and \eqref{eq:Y0Commutant} show that
$\mj Y_0^{-1}W_\vartheta(z)^\intercal Y_0\mj=\cB_\vartheta(z)$.
So we conclude that $\tr W_\vartheta=\tr \cB_\vartheta=\tr\cM$.

Recalling that
\begin{align*}
\cB(z)=\frac{1}{2}\begin{bmatrix}
\psi_p(z)+\psi_p^*(z) & \psi_p^*(z)-\psi_p(z) \\
\varphi_p^*(z)-\varphi_p(z) & \varphi_p(z)+\varphi_p^*(z)
\end{bmatrix},
\end{align*}
a direct computation shows that
\begin{align}
\cB_\vartheta(z)=\begin{bmatrix}
\cB^\vartheta_{11}(z)&\cB^\vartheta_{12}(z)\\
\cB^\vartheta_{21}(z)&\cB^\vartheta_{22}(z)
\end{bmatrix}=\frac{1}{2}\begin{bmatrix}
\psi_{p,\vartheta}(z)+\psi_{p,\vartheta}^*(z) & \psi_{p,\vartheta}^*(z)-\psi_{p,\vartheta}(z) \\
\varphi_{p,\vartheta}^*(z)-\varphi_{p,\vartheta}(z) & \varphi_{p,\vartheta}(z)+\varphi_{p,\vartheta}^*(z)
\end{bmatrix}
\end{align}
for the rotated rational functions $\varphi_{p,\vartheta}=e^{-i\vartheta}\varphi_p$, $\psi_{p,\vartheta}=e^{-i\vartheta}\psi_p$, $\varphi_{p,\vartheta}^*=e^{i\vartheta}\varphi_p^*$, and $\psi_{p,\vartheta}^*=e^{i\vartheta}\psi_p^*$. 
The discriminant $\Delta_A$ defined by \eqref{def:discriminant} can therefore be written as
\begin{align}
\Delta_{A}(z)=\frac{1}{B(z)}\tr\cB_\vartheta(z)=\frac{\psi_{p,\vartheta}(z)+\psi_{p,\vartheta}^*(z)+\varphi_{p,\vartheta}(z)+\varphi_{p,\vartheta}^*(z)}{2B(z)},
\end{align}
where
\begin{align}
B(z)=z\prod_{j=1}^{n-1}b_{z_j}(z)=\sqrt{\det \cB_\vartheta(z)}.
\end{align}
What follows is a detailed study of properties of $\Delta_A$. This will enable us to give a complete description of the spectral measure of $A_+$ by means of a uniquely associated divisor $D$.

With \eqref{eq:productTransfer} in mind, the analog of the Lyapunov exponent (see, e.g., \cite{PaFi92}) in our periodic setting is given by
\begin{align}
L(z)=\lim_{k\to\infty}\frac{1}{kp}\log \|W_\vartheta(z)^k\|.
\end{align}
provided the limit exists.
We shall shortly relate the Lyaponov exponent to the discriminant. The lemma below is critical when showing that $\Delta_{A}^{-1}([-2,2])\subset\partial\D$.
\begin{lemma}
For every $z\in \C$ with $z\notin \{z_j^*: 0\leq j\leq n-1\}$ the limit exists and satisfies $L(z)\geq 0$.
\end{lemma}
\begin{remark}
In fact, we will see that $L(z)=0$ if and only if $z\in\Delta_{A}^{-1}([-2,2])=\sigma(A)\subset\partial\D$.
\end{remark}
\begin{proof}
Let $\l_1(z),\l_2(z)$ be the eigenvalues of $W_\vartheta(z)$. Then by the spectral radius formula we have
\begin{align}\label{eq:limitexists}
\lim\limits_{k\to\infty}\|W_\vartheta(z)^k\|^{1/kp}=\max\{|\l_1(z)|,|\l_2(z)|\}^{1/p}.
\end{align}
Moreover, since the inequality $|\det N|\leq \|N\|^2$ holds for every $2\times 2$-matrix $N$ and $\det W_\vartheta=B^2$, we see that on $\C\setminus \D$ the limit exists and satisfies $L\geq 0$. To show that this also holds inside $\D$, we first note that due to \eqref{eq:ORFTransfer} and \eqref{eq:productTransfer},
\begin{align}\label{eq:WTheta}
W_\vartheta^k=\frac{1}{2}
\begin{bmatrix}
\varphi_{kp,\vartheta}+\psi_{kp,\vartheta}     & \varphi_{kp,\vartheta}-\psi_{kp,\vartheta}     \\
\varphi_{kp,\vartheta}^*-\psi_{kp,\vartheta}^* & \varphi_{kp,\vartheta}^*+\psi_{kp,\vartheta}^*
\end{bmatrix}.
\end{align}
Hence,
\begin{align*}
|\varphi_{kp,\vartheta}^*(z)|\leq \|W_\vartheta(z)^k\left[\begin{smallmatrix}
	1\\
	1
\end{smallmatrix}\right]\|
\leq \|W_\vartheta(z)^k\|\sqrt{2}
\end{align*}
and we now apply the Christoffel--Darboux formula (see \cite[Theorem 3.1.3]{BuGoHe99})
\begin{align*}
\sum_{j=0}^{l-1}|\varphi_{j,\vartheta}(z)|^2=\frac{|\varphi^*_{l,\vartheta}(z)|^2-|\varphi_{l,\vartheta}(z)|^2}{1-| b_{z_l}(z)|^2}
\end{align*}
to deduce that
\begin{align*}
1-|z|^2\leq |\varphi^*_{kp,\vartheta}(z)|^2,
\end{align*}
recalling that $b_{z_{kp}}(z)=z$ and $|\varphi_0|=1$.  So for fixed $z\in\D$ we have a uniform lower bound on $\|W_\vartheta(z)^k\|$, and this implies that $L(z)\geq 0$.
\end{proof}

The following lemma collects several important properties of $\Delta_{A}$. It is the analog of \cite[Theorem 11.1.1]{SimonOPUC2} for CMV matrices and thus we seek to keep the proofs rather short by merely indicating where adaptations are needed.
\begin{lemma}\label{lem:PropDelta}
\begin{itemize}
\item[(i)] $\overline{\Delta_A(1/\overline z)}=\Delta_A(z)$,
\item[(ii)] ${\displaystyle L(z)=\frac{1}{p}\log|B(z)|+\frac{1}{p}\log\left|\frac{\Delta_A(z)+\sqrt{\Delta_A(z)^2-4}}{2}\right|}$,
\item[(iii)] $\Delta_A(z)\in[-2,2]$ implies that $z\in\partial\D$,
\item[(iv)] for all critical points $c\in \partial \D$ (i.e., zeros of $\Delta_A')$, we have that $|\Delta_A(c)|\geq 2$.
\end{itemize}
\end{lemma}

\begin{proof}
(i) The key is that $B(z)^{-1}W_\vartheta(z)\in \mathbb S\mathbb U(1,1)$ if $z\in\partial\D$  and to use that the trace of a matrix in this class is real. The statement then follows by analytic continuation.

(ii) Let again denote $\l_1(z),\l_2(z)$ the eigenvalues of $W_\vartheta(z)$ and let $\tilde \l_1(z),\tilde \l_2(z)$ denote the eigenvalues of $B(z)^{-1}W_\vartheta(z)$. Then we have
\begin{align*}
B(z)\tilde \l_i(z)=\l_i(z), \quad i=1, 2
\end{align*}
and thus by \eqref{eq:limitexists},
\begin{align*}
L(z)=\log\lim\limits_{k\to\infty}\|W_\vartheta(z)^k\|^{1/kp}
=\frac{1}{p}\log|B(z)|+\frac{1}{p}\log\Bigl(\max\{|\tilde \l_1(z)|,|\tilde \l_2(z)|\}\Bigr).
\end{align*}
Since $\tilde \l_1, \tilde \l_2$ are solutions of $\tilde \l^2-\Delta_A(z)\tilde \l+1=0$, we obtain (ii).

(iii) Suppose $\Delta_A(z)\in[-2,2]$. Then $\Bigl|\frac{\Delta_A(z)+\sqrt{\Delta_A(z)^2-4}}{2}\Bigr|=1$ and so, by (ii),
$$
L(z)=\frac{1}{p}\log|B(z)|.
$$
Since $L(z)\geq 0$, this implies that $|z|\geq 1$. But if $|z|>1$ then, by (i), we have
$$
\Delta_A(z)=\overline{\Delta_A(1/\overline z)}\notin[-2,2].
$$

(iv) This is the same as to say that for $x\in(-2,2)$, the roots of $\Delta_A-x$ are simple. The proof is identical to that of \cite[Theorem 11.1.1]{SimonOPUC2}.
\end{proof}

It follows from the lemma that $\E:=\Delta_A^{-1}([-2,2]) \subset \partial\D$ is a finite-gap set with at most $p$ gaps. As before, we denote by $g+1$ the number of open gaps in $\E$ and by $\l_j^\pm$ their gap edges.
In the following it will be important to know that $\Delta_{A}$ is a rational function of degree $p$. But, in fact, an even stronger statement is true:

\begin{lemma}\label{lem:Resolvents}
Fix $j$ and let $q$ be the number of times $z_j$ appears in the vector $\vz$. Then
\begin{align}\label{eq:openGap}
|(b_{z_j}^q\Delta_A)(z_j)|>C>0,
\end{align}
where the constant $C$ depends only on $\vz$. Moreover,
\begin{align}\label{eq:ResConjugated}
(b_{z_j}^{-q}\Delta_A)(z_j^*)=\overline{(b_{z_j}^q\Delta_A)(z_j)}.
\end{align}
\end{lemma}

\begin{proof}
Note that \eqref{eq:ResConjugated} follows from the fact that $\Delta_A$ is real. Let us first assume that $q=1$ and $z_j=z_0=0$. In that case, \eqref{eq:openGap} is equivalent to $|(\tr \cM)(0)|> C$.  It will be more convenient to consider the product
\begin{align*}
\cM_r(z):=\begin{bmatrix}
z& 0\\
0& 1
\end{bmatrix}
\cM(z)
\begin{bmatrix}
z& 0\\
0& 1
\end{bmatrix}^{-1}
\end{align*}
which clearly has the same trace as $\cM$. We have
\begin{align}\label{eq:AatZero}
\cM_r(0)=
\begin{bmatrix}
0&0\\
0& 1
\end{bmatrix}	
\cU(a_0)
\begin{bmatrix}
-z_1& 0\\
0& 1
\end{bmatrix}
\cdots
\begin{bmatrix}
0&0\\
0& 1
\end{bmatrix}
\cU(a_{p-1})
\begin{bmatrix}
e^{-i\vartheta}& 0\\
0& e^{i\vartheta}
\end{bmatrix}.
\end{align}
Set
\begin{align*}
V(z)=\cU(a_0)
\begin{bmatrix}
b_{z_1}(z) & 0 \\
0          & 1
\end{bmatrix}
\cdots
\cU(a_{p-3})
\begin{bmatrix}
b_{z_{n-1}}(z) & 0 \\
0          & 1
\end{bmatrix}
\end{align*}
and notice that $V$ is a transfer matrix associated to the poles $z_1,\dots, z_{n-1}$ and with coefficients $-\overline{a_0}, \dots, -\overline{a_{p-3}}$ (in reverse order). Therefore, by \eqref{eq:ORFTransfer}, a short computation shows that
\begin{align*}
V=\frac12\begin{bmatrix}
\tilde\varphi+\tilde\psi     & \tilde\varphi-\tilde\psi     \\
\tilde\varphi^*-\tilde\psi^* & \tilde\varphi^*+\tilde\psi^*
\end{bmatrix}
\end{align*}
for the corresponding ORF of degree $p-2$. Due to \eqref{eq:AatZero}, it suffices to show that
\begin{align*}
\left|
\begin{bmatrix}
0 & 1
\end{bmatrix}
V(0)\cU(a_{p-2})
\begin{bmatrix}
0 \\
1
\end{bmatrix}\right|=\left|\frac{1}{2\rho_{p-2}}\Bigl(\tilde\varphi^*(0)+\tilde\psi^*(0)+a_{p-2}\bigl(\tilde\varphi^*(0)-\tilde\psi^*(0)\bigr)\Bigr)\right|
\end{align*}
is uniformly bounded from below.

For arbitrary $z_j$, due to cyclic rotation, we would have obtained the same, just with orthogonal rational functions associated to different coefficients, respectively a different measure. Recall that we can always ``push'' the matrix $\left[\begin{smallmatrix}
e^{i\vartheta}& 0\\
0& e^{-i\vartheta}
\end{smallmatrix}\right]$
to the end of the product by means of the commutant relation
\begin{align*}
\cU(e^{2i\vartheta}a)=\begin{bmatrix}
e^{i\vartheta}& 0\\
0& e^{-i\vartheta}
\end{bmatrix}
\cU(a)
\begin{bmatrix}
e^{-i\vartheta}& 0\\
0& e^{i\vartheta}
\end{bmatrix}.
\end{align*}
Hence, it suffices to show that
\begin{align}\label{eq:CaratheodoryInequ}
\Bigl|\tilde\varphi^*(z_j)+\tilde\psi^*(z_j)+a\bigl(\tilde\varphi^*(z_j)-\tilde\psi^*(z_j)\bigr)\Bigr|
\end{align}
is uniformly bounded from below for arbitrary orthogonal rational functions whose poles are supported on the set $\{z_i^*: 0\leq i\leq n-1, i\neq j\}$.

It is well-known that any Caratheodory function $F$ satisfies the uniform bounds
\begin{align*}
\frac{1-|z|}{1+|z|}\leq|F(z)|\leq \frac{1+|z|}{1-|z|},\quad \frac{1-|z|}{1+|z|}\leq\Re F(z)\leq \frac{1+|z|}{1-|z|}.
\end{align*}
Since $\sup_j|z_j|<1$, this gives positive constants $c_1,c_2$ such that for every Caratheodory function $F$ and every $z_j$, we have that $c_1<|F(z_j)|<c_2$. By the same reasoning, using the Christoffel--Darboux relation it is not hard to see that there exists a function $m(r)$ such that for all measures $\nu$, all $n$, and all $z$ obeying $|z|<r$, we have  $|\varphi_n^*(z,\nu)|>m(r)>0$; cf. \cite[Lemma 9.3.1]{BuGoHe99}.  Hence we obtain a constant $c_3$ such that $|\tilde\varphi^*(z_j)|>c_3$ uniformly.

Writing $\tilde{F}$ as shorthand notation for the Caratheodory function $\tilde\psi^*/\tilde\varphi^*$, we see that
\begin{align*}
\Bigl|\tilde\varphi^*(z_j)+\tilde\psi^*(z_j)+a\bigl(\tilde\varphi^*(z_j)-\tilde\psi^*(z_j)\bigr)\Bigr|=\bigl|\tilde\varphi^*(z_j)\bigr|\bigl|(1+a)+\tilde{F}(z_j)(1-a)\bigr|.
\end{align*}
Let us first assume that $|1-a|<(1+c_2)^{-1}\leq 1/2$. Then standard estimates show that
\begin{align*}
\bigl|\tilde\varphi^*(z_j)\bigr|\bigl|(1+a)+\tilde{F}(z_j)(1-a)\bigr|>c_3.
\end{align*}
If $|1-a|\geq (1+c_2)^{-1}$, we obtain the estimate
\begin{align*}
\bigl|\tilde\varphi^*(z_j)(1-a)\bigr|\left|\frac{1+a}{1-a}+\tilde{F}(z_j)\right|&\geq \bigl|\tilde\varphi^*(z_j)(1-a)\bigr|\Re\left(\frac{1+a}{1-a}+\tilde{F}(z_j)\right)\\
&\geq \bigl|\tilde\varphi^*(z_j)(1-a)\bigr|\Re F(z_j)>\frac{c_1c_3}{1+c_2}>0
\end{align*}
since the real part of $\frac{1+a}{1-a}$ is positive.

Finally, if $q>1$ then \eqref{eq:AatZero} splits into shorter products of the form
\begin{align*}
\begin{bmatrix}
0& 0\\
0& 1
\end{bmatrix}
\tilde V(z)
\begin{bmatrix}
0& 0\\
0& 1
\end{bmatrix}
\end{align*}
and the same arguments can be applied to each single factor.
\end{proof}


Combining the two previous lemmas, 
we obtain the following explicit representation for $\Delta_A$:
\begin{lemma}\label{lem:repDeltaT}
Let $\fz$ be the uniformization of $\overline\C\setminus \E$.
Following the notation of \eqref{eq:BB}, we have
\begin{align}
\Delta_A(z)=\Psi(z)+\Psi(z)^{-1},
\end{align}
where $\Psi\circ\fz=\prod_{j=0}^{n-1}\fb_{\z_j}\fb_{\overline{\z_j}}$.
\end{lemma}

\begin{proof}
Consider the function
\begin{align}\label{eq:Hfn}
H(z) &= \log\left|\frac{\Delta_A(z) + \sqrt{\Delta_A^2(z) - 4}}{2}\right| - \sum_{j=0}^{n-1} \Bigl( G_{\overline{\C}\setminus \E}(z,z_j) + G_{\overline{\C}\setminus \E}(z,z_j^*) \Bigr).
\end{align}
It follows from Lemma \ref{lem:Resolvents} that $H(z)$ has no poles and is thus harmonic in $\overline{\C}\setminus \E$.  Furthermore, since $\E := \Delta_A^{-1}([-2,2])$, we have $H(z) = 0$ for $z \in \E$.  By the maximum principle, it follows that $H(z)\equiv 0$.  In particular, by an application of \eqref{eq:btoGreens} we have
\begin{align*}
\log\left|\frac{\Delta_A\bigl(\fz(\z)\bigr) + \sqrt{\Delta_A^2\bigl(\fz(\z)\bigr) - 4}}{2}\right| = -\log\left|\Psi \circ \fz\right|,
\end{align*}
from which the lemma follows.
\end{proof}

We are now ready to characterize the spectrum of $A_+$. Since $\det M_\vartheta=B^2$, \eqref{eq:quadraticCara} shows that $F_\nu$ can be written as
\begin{align}\label{eq:caratheordory}
F_\nu(z)=\frac{v(z)+\sqrt{\Delta_A(z)^2-4}}{u(z)},
\end{align}
where $u, v$ are explicitly given by
\begin{align*}
v(z)=\frac{\psi_{p,\vartheta}(z)+\psi_{p,\vartheta}^*(z)-\varphi_{p,\vartheta}(z)-\varphi^*_{p,\vartheta}(z)}{2B(z)},\quad u(z)=\frac{\varphi_{p,\vartheta}^*(z)-\varphi_{p,\vartheta}(z)}{B(z)},
\end{align*}
and where the branch of the square root is chosen such that $F_\nu(0)=1$ and then extended analytically to $\overline{\C}\setminus \E$.
Since $\varphi^*=B^2\overline{\varphi}$ on $\partial\D$, it follows that 
\begin{align}\label{eq:normalizationUV}
v(e^{it})\in\R \quad \text{and} \quad u(e^{it})\in i\R.
\end{align}
Write $\dd\nu(t)=\nu_{ac}(t)\frac{\dd t}{2\pi}+\dd\nu_s(t)$, with $\nu_s$ singular to $\dd t/2\pi$. 
Using the standard inversion formula, we see that $\dd\nu_s$ is a finite sum of point masses and $\nu_{ac}$ is explicitly given by
\begin{align}\label{eq:acMeasure}
\nu_{ac}(t)=\lim\limits_{r \uparrow 1}\Re F_{\nu}(re^{it})=\frac{\sqrt{\Delta_A(e^{it})^2-4}}{u(e^{it})}\geq 0,\quad e^{it}\in \E.
\end{align}
The following lemma will be important to characterize the point masses of $\nu$. 

\begin{lemma}\label{lem:uzeros}
$u(z)$ has all its zeros in the set of gaps of $\E$, one in each gap.
\end{lemma}

\begin{proof}
First we show that all zeros of $u$ lie on $\partial\D$. Since $|\varphi_{p,\vartheta}|<|\varphi_{p,\vartheta}^*|$ on $\D$ and $|\varphi_{p,\vartheta}|>|\varphi_{p,\vartheta}^*|$ on $\overline\C\setminus\overline{\D}$ (see \cite[Corollary 3.1.4.]{BuGoHe99}), the assertion $\varphi_{p,\vartheta}(z)=\varphi_{p,\vartheta}^*(z)$ (i.e., $u(z)=0$) implies that $z\in\partial\D$.


Next we show that, at a point $z_0\in\partial\D$ with $\varphi_{p,\vartheta}(z_0)=\varphi_{p,\vartheta}^*(z_0)$, we have that $|\tr\Delta_{A}(z_0)|\geq 2$. First observe that this holds for a  matrix $N\in\mathbb S\mathbb L(2,\mathbb R)$ with $N_{21} = 0$;
indeed, an application of the inequality of arithmetic and geometric means shows that $|\tr N|\geq 2\det N=2$.  Set $M=B(z_0)^{-1}M_\vartheta(z_0)$. Due to \eqref{def:TransferMatrix}, $M$ can be written as $M=Y_0UY_0^{-1}$ for some $U\in\mathbb S\mathbb U(1,1)$.  Viewed as fractional linear transformations, $Y_0$ maps the unit disc $\D$ into the right half plane $\mathbb H_+$ and $U$ preserves the unit disc; thus, conjugating $M$ further by
\begin{align*}
R = \begin{bmatrix}
i & 0 \\
0 & 1
\end{bmatrix},
\end{align*}
we can transform $M$ into an element of $\mathbb S\mathbb L(2,\mathbb R)$.  It remains to note that, due to \eqref{eq:Bmatrix}, $M_{21}$ is zero, and that this property (in addition to the determinant and the trace) is preserved under conjugation by $R$.

Since $\Delta_A$ is of degree $p$, there are exactly $p$ gaps of $\E$. As the square root in \eqref{eq:acMeasure} is analytically extended, it changes sign in every gap. In order to retain the positivity, the denominator must also admit a sign change and this implies there is a zero of $u(z)$ in every gap. Since $u(z)$ is of degree at most $p$, we find that there is exactly one zero in each gap.
\end{proof}

Inspired by the above lemma, let us define a divisor
\begin{align}\label{eq:divisor}
D=\bigl\{(x_j,\e_j)\bigr\}_{j=0}^g\in\cD(\E)
\end{align}
with $\{x_j\}$ accounting for the zeros of $u(z)$ in the open gaps and where $\e_j=1$ if $x_j$ is a mass point of $\nu$ and $-1$ otherwise.

\begin{proposition}\label{prop:uniqueDiv}
We can uniquely recover the measure $\nu$ from the divisor $D$ and the discriminant $\Delta_A$. Specifically, the absolutely continuous part of $\nu$ is given by \eqref{eq:acMeasure} and if $\e_j=1$, the point mass at $x_j$ has the weight
\begin{align}\label{eq:pointMass}
\nu(\{x_j\})=\frac{\sqrt{|\Delta_A(x_j)^2-4|}}{|u'(x_j)|}.
\end{align}
\end{proposition}

\begin{proof}
If $x_j$ is a zero of $u$, then -- using again that $\det\cB_\vartheta=B^2$ -- we see that $\Delta_A(x_j)^2-4=v(x_j)^2$. Hence the sign of the square root determines whether or not $x_j$ is a point mass of $\nu$. If the numerator in \eqref{eq:caratheordory} does not vanish, it is given by $2\sqrt{\Delta_A^2-4}$ and the weight of the point mass can be computed as in \eqref{eq:pointMass}.

Note that up to a multiplicative constant, $u$ is defined by its zeros. This follows from \eqref{eq:normalizationUV} and the fact that $\nu$ is a probability measure. Hence, $D$ and $\Delta_A$ determine $\nu$ uniquely.
\end{proof}

For finite-gap sets $\E$, Lemma \ref{lem:uzeros} (and the comment thereafter) defines a map from $\{A\in\A_\per(\vz): \sigma(A)=\E\}$ to $\cD(\E)$: namely, given $A \in \A_\per(\vz)$, we consider the associated half-line operator $A_+$, compute its Caratheodory function $F_\nu$ as \eqref{eq:caratheordory}, and find and label the zeros of $u(z)$ as \eqref{eq:divisor}.   Proposition \ref{prop:uniqueDiv} shows this map is one-to-one.  That this map is also onto is true in general, and will be shown for the special choice $\vz=\vz_\E$ in Section \ref{sec:proofs}.

\subsection{The structure of a general MCMV matrix}\label{sec:structure}

In this section we demonstrate the ``block-CMV'' nature of MCMV matrices as in \eqref{eq:Astructure}, as well as a structural stability under taking M\"obius transformations related to the points in the generating vector $\vz$.  This structure will be critical to understanding our main theorems; indeed, in light of viewing the discriminant as in \eqref{eq:DeltaERational}, the Magic Formula would be a complete mystery without developing some understanding of the structure of $b_{z_j}(A)$ for a general MCMV matrix $A \in \A(\vz)$.

Our analysis further illustrates the similarities between our MCMV matrices and their self-adjoint analog, GMP matrices.  First, GMP matrices are block-Jacobi; below, we show the band structure \eqref{eq:Astructure} of an MCMV matrix in Lemma \ref{thm:structureA}.  Additionally, one of the characteristic properties of GMP matrices is that they are stable under taking resolvents (cf. \cite[Definition 1.12]{Yud18}); the analogous statement for MCMV is Proposition \ref{lem:BlaschkeA}.   The consequences for the Magic Formula in the setting of \eqref{eq:DeltaERational} are the content of Theorem \ref{t:MCMVstructure}.  Finally, we use all of this structure to prove a uniqueness result for certain rational functions of our MCMV matrices in Proposition \ref{prop:DiscrCoincide}; this will be used to prove the Magic Formula in the end.

Let us again fix a vector $\vz\in\D^n$ having $z_0 = 0$, and recall that $D_0$ denotes the diagonal operator defined in \eqref{eq:perdiag} and $C$ denotes a general CMV matrix.  We begin by proving the block structure of an MCMV matrix $A$; up to conjugation by diagonal matrices, we may consider instead the simpler operator
\begin{align}
	\tilde A:=(1+CD_0^*)^{-1}(C+D_0).
\end{align}
Let us split $\tilde A$ into matrix blocks of size $2n\times 2n$ and denote the blocks by $\mathbf {\tilde{A}_{ij}}$.  For simplicity, we assume throughout this section that $\vz$ is such that $z_j\neq z_k$ for $j\neq k$.  Without this assumption, the block structure below will split into smaller blocks.

\begin{lemma}
\label{thm:structureA}
$\tilde A$ is band structured and $\mathbf {\tilde{A}_{ij}}=0$ if $|i-j|>1$. Moreover, there exists vectors $\mathbf{u^i},\mathbf{v^i}\in\C^{2n}$ such that
\begin{align}\label{eq:offDiagEntries}
\mathbf {\tilde{A}_{i-1,i}}=\mathbf{v^i}\delta_{2g+1}^\intercal,\quad \mathbf{\tilde{A}_{i,i+1}}=\mathbf{u^i}\delta_{0}^\intercal.
\end{align}
In particular,
\begin{align}\label{eq:offDiagfzero}
\tilde A_{0,2n}=
\bigg(
\frac{1}{\rho_{2n-1}}
\begin{bmatrix}
1& 0
\end{bmatrix}
\cU(-\overline{a_{2n-2}})
\begin{bmatrix}
1&0\\
0& -\overline{z_{n-1}}
\end{bmatrix}
\cdots
\begin{bmatrix}
1& 0\\
0& -\overline{z_1}
\end{bmatrix}
\cU(-\overline{a_1})
\begin{bmatrix}
-\frac{1}{\overline{z_1}}& 0\\
0& 1
\end{bmatrix}
\cU(-\overline{a_0})
\begin{bmatrix}
1\\
0
\end{bmatrix}
\bigg)^{-1}.
\end{align}
\end{lemma}

\begin{proof}
Let\begin{align*}
u=(1+CD_0^*)^{-1}(C+D_0)\delta_k
\end{align*}
and note that $u$ satisfies the recursion relation
\begin{align}\label{eq:recf1}
(1+C\DD^*)u=(C+\DD)\delta_k.
\end{align}
The left-hand side of this identity can be written as
\begin{align}
\begin{bmatrix}
u_{2j} \\
u_{2j+1}
\end{bmatrix}
& +\overline{z_j}
\begin{bmatrix}
\overline{a_{2j}} \\
\rho_{2j}
\end{bmatrix}
\begin{bmatrix}
\rho_{2j-1} & -a_{2j-1}
\end{bmatrix}
\begin{bmatrix}
u_{2j-1} \\
u_{2j}
\end{bmatrix}
+\overline{z_{j+1}}
\begin{bmatrix}
\rho_{2j} \\
-a_{2j}
\end{bmatrix}
\begin{bmatrix}
\overline{a_{2j+1}} & \rho_{2j+1}
\end{bmatrix}
\begin{bmatrix}
u_{2j+1} \\
u_{2j+2}
\end{bmatrix}.
\label{eq:fRecursion1}
\end{align}
Introduce the matrices
\begin{align*}
M_j=\begin{bmatrix}
1 \\
0
\end{bmatrix}
\begin{bmatrix}
0 & 1
\end{bmatrix}
+\overline{z_j}
\begin{bmatrix}
\overline{a_{2j}} \\
\rho_{2j}
\end{bmatrix}
\begin{bmatrix}
\rho_{2j-1} & -a_{2j-1}
\end{bmatrix}
\end{align*}
and
\begin{align*}
N_j=-
\begin{bmatrix}
0 \\
1
\end{bmatrix}
\begin{bmatrix}
1 & 0
\end{bmatrix}
-\overline{z_{j+1}}
\begin{bmatrix}
\rho_{2j} \\
-a_{2j}
\end{bmatrix}
\begin{bmatrix}
\overline{a_{2j+1}} & \rho_{2j+1}
\end{bmatrix}
\end{align*}
so that \eqref{eq:fRecursion1} becomes
\begin{align*}
M_j
\begin{bmatrix}
u_{2j-1} \\
u_{2j}
\end{bmatrix}
& -
N_j
\begin{bmatrix}
u_{2j+1} \\
u_{2j+2}
\end{bmatrix}.
\end{align*}
Since $z_{kn}=0$, we see that for $j=kn-1$ and $j=kn$ this becomes
\begin{align*}
M_{kn-1}
\begin{bmatrix}
u_{2kn-3}\\
u_{2kn-2}
\end{bmatrix}
+
\begin{bmatrix}
0\\
u_{2kn-1}
\end{bmatrix}
\quad\text{and}\quad
\begin{bmatrix}
u_{2kn}\\
0
\end{bmatrix}
-N_{kn}
\begin{bmatrix}
u_{2kn+1}\\
u_{2kn+2}
\end{bmatrix},
\end{align*}
respectively. Hence we see that the recursion for the blocks $\{u_{2kn},\dots,u_{2(k+1)n-1} \}$ is decoupled. The finite band structure and \eqref{eq:offDiagEntries} is now a direct consequence of the structure of CMV matrices.

It remains to prove \eqref{eq:offDiagfzero}.
Let now $u=(1+C\DD^*)^{-1}(C+\DD)\delta_{2n}$. We are interested in the block $\{u_0,\dots, u_{2n-1}\}$. In this case, \eqref{eq:recf1} leads to
\begin{align*}
\begin{bmatrix}
0\\
0
\end{bmatrix}
=
\begin{bmatrix}
u_0\\
0
\end{bmatrix}
-N_0
\begin{bmatrix}
u_1\\
u_2
\end{bmatrix},
\quad
M_{n-1}
\begin{bmatrix}
u_{2n-3}\\
u_{2n-2}
\end{bmatrix}
+
\begin{bmatrix}
0\\
u_{2n-1}
\end{bmatrix}
=\begin{bmatrix}
\rho_{2n-1}\rho_{2n-2}\\
-\rho_{2n-1}a_{2n-2}
\end{bmatrix}
\end{align*}
and for $1\leq j<n-1$,
\begin{align*}
\begin{bmatrix}
0\\
0
\end{bmatrix}
=
M_j
\begin{bmatrix}
u_{2j-1} \\
u_{2j}
\end{bmatrix}
-
N_j
\begin{bmatrix}
u_{2j+1} \\
u_{2j+2}
\end{bmatrix}.
\end{align*}
Iterating this leads to
\begin{align*}
\begin{bmatrix}
u_0\\
0
\end{bmatrix}
=
N_0M_1^{-1}N_1\cdots N_{n-2}M_{n-1}^{-1}
\bigg(
\begin{bmatrix}
\rho_{2n-1}\rho_{2n-2}\\
-\rho_{2n-1}a_{2n-2}
\end{bmatrix}
-
\begin{bmatrix}
0\\
u_{2n-1}
\end{bmatrix}
\bigg).
\end{align*}
Using the fact that
\begin{align*}
M_0\begin{bmatrix}
u_{2n-1}\\
u_0
\end{bmatrix}
=
\begin{bmatrix}
u_0\\
0
\end{bmatrix} \; \mbox{ and } \;
N_{n-1}\begin{bmatrix}
u_{2n-1}\\
u_0
\end{bmatrix}
=
\begin{bmatrix}
0\\
u_{2n-1}
\end{bmatrix},
\end{align*}
we arrive at
\begin{align*}
\begin{bmatrix}
\rho_{2n-1}\rho_{2n-2}\\
-\rho_{2n-1}a_{2n-2}
\end{bmatrix}
=
\bigg(M_{n-1}N_{n-2}^{-1}M_{n-2}\cdots N_0^{-1}M_0
-	
N_{n-1}
\bigg)
\begin{bmatrix}
u_{2n-1}\\
u_0
\end{bmatrix}.
\end{align*}
It is straightforward to see that
\begin{align*}
N_j^{-1}&=\frac{-1}{\rho_{2j}\rho_{2j+1}}\bigg(\frac{1}{\overline{z_{j+1}}}\begin{bmatrix}
0\\
1
\end{bmatrix}
\begin{bmatrix}
1& 0
\end{bmatrix}
+
\begin{bmatrix}
\rho_{2j+1}\\
-\overline{a_{2j+1}}
\end{bmatrix}
\begin{bmatrix}
a_{2j}& \rho_{2j}
\end{bmatrix}
\bigg).
\end{align*}
Let
\begin{align*}
M_{n-1}N_{n-2}^{-1}M_{n-2}\cdots N_0^{-1}=C=\begin{bmatrix}
c_{11}& c_{12}\\
c_{21}& c_{22}
\end{bmatrix}.
\end{align*}
Due to the simple structure of $M_0$ and $N_{n-1}$, it follows that
\begin{align*}
u_0=\frac{\rho_{2n-2}\rho_{2n-1}}{c_{11}}
\end{align*}
and hence it suffices to study
$
C\left[\begin{smallmatrix}
1\\
0
\end{smallmatrix}\right].
$
A direct computation shows that
\begin{align*}
N_0^{-1}\begin{bmatrix}
1\\
0
\end{bmatrix}
=
\frac{-1}{\rho_0\rho_1}
\begin{bmatrix}
0& \rho_1\\
1& -\overline{a_1}
\end{bmatrix}
\begin{bmatrix}
\frac{1}{\overline{z_1}}& 0\\
0& 1
\end{bmatrix}
\begin{bmatrix}
1\\
a_0
\end{bmatrix}
\end{align*}
and combined with the identities
\begin{align*}
M_j\begin{bmatrix}
0& \rho_{2j-1}\\
1& -\overline{a_{2j-1}}
\end{bmatrix}
=\begin{bmatrix}
1& \overline{a_{2j}}\\
0& \rho_{2j}
\end{bmatrix}
\begin{bmatrix}
1& 0\\
0&\overline{z_j}
\end{bmatrix}
\begin{bmatrix}
1&-\overline{a_{2j-1}}\\
-a_{2j-1}& 1
\end{bmatrix}
\end{align*}
and
\begin{align*}
N_j^{-1}\begin{bmatrix}
1& \overline{a_{2j}}\\
0& \rho_{2j}
\end{bmatrix}
=
\frac{-1}{\rho_{2j}\rho_{2j+1}}
\begin{bmatrix}
0& \rho_{2j+1}\\
1& -\overline{a_{2j+1}}
\end{bmatrix}
\begin{bmatrix}
\frac{1}{\overline{z_{j+1}}}& 0\\
0& 1
\end{bmatrix}
\begin{bmatrix}
1& \overline{a_{2j}}\\
a_{2j}& 1
\end{bmatrix},
\end{align*}
this allows us to iterate the procedure. It only remains to comment on the sign and for this we note that
\begin{align*}
-\begin{bmatrix}
\frac{1}{\overline{z_{j+1}}}& 0\\
0& 1
\end{bmatrix}
\begin{bmatrix}
1& \overline{a_{2j}}\\
a_{2j}& 1
\end{bmatrix}
\begin{bmatrix}
1& 0\\
0&\overline{z_j}
\end{bmatrix}
=
\begin{bmatrix}
\frac{-1}{\overline{z_{j+1}}}& 0\\
0& 1
\end{bmatrix}
\begin{bmatrix}
1& -\overline{a_{2j}}\\
-a_{2j}& 1
\end{bmatrix}
\begin{bmatrix}
1& 0\\
0&-\overline{z_j}
\end{bmatrix}.
\end{align*}
Therefore,
\begin{align*}
C\begin{bmatrix}
1\\
0
\end{bmatrix}
=
\frac{1}{\rho_0}
\begin{bmatrix}
1& \overline{a_{2n-2}}\\
0&\rho_{2n-2}
\end{bmatrix}
\begin{bmatrix}
1&0\\
0& \overline{z_{n-1}}
\end{bmatrix}
\dots
\begin{bmatrix}
1& 0\\
0& -\overline{z_1}
\end{bmatrix}
\cU(-\overline{a_1})
\begin{bmatrix}
-\frac{1}{\overline{z_1}}& 0\\
0& 1
\end{bmatrix}
\begin{bmatrix}
1\\
a_0
\end{bmatrix}
\end{align*}
and hence
\begin{align*}
\begin{bmatrix}
1& 0
\end{bmatrix}
C
\begin{bmatrix}
1\\
0
\end{bmatrix}
=
\rho_{2n-2}
\begin{bmatrix}
1& 0
\end{bmatrix}
\cU(-\overline{a_{2n-2}})
\begin{bmatrix}
1&0\\
0& -\overline{z_{n-1}}
\end{bmatrix}
\dots
\begin{bmatrix}
1& 0\\
0& -\overline{z_1}
\end{bmatrix}
\cU(-\overline{a_1})
\begin{bmatrix}
-\frac{1}{\overline{z_1}}& 0\\
0& 1
\end{bmatrix}
\cU(-\overline{a_0})
\begin{bmatrix}
1\\
0
\end{bmatrix}.
\end{align*}
This concludes the proof.
\end{proof}

A key feature of the MCMV structure is its stability under M\"obius transformations.  Following the notation of Appendix \ref{sec:operatorMoebius}, 
we can write the operator M\"obius transform defined in \eqref{eq:VelazquezMCMV} as
\begin{align*}
b_{S}(C)=\eta_S(1-CS^*)^{-1}(C-S)\eta_S^{-1}=\Phi_U(C),
\end{align*}
where
\begin{align*}
U=\begin{bmatrix}
1 & -S^* \\
-S & 1
\end{bmatrix}
\begin{bmatrix}
\eta_S^{-1} & 0           \\
0           & \eta_S^{-1}
\end{bmatrix}.
\end{align*}
In particular, this holds for $S=z_j1$, where $1$ denotes the identity matrix in $\ell^2$, i.e., $b_{z_j}(A)$ is the standard Blaschke factor evaluated at $A\in\overline{\D_{\ell^2}}$:
\begin{align}
b_{z_j}(A) = (1-\overline{z_j}A)^{-1}(A-z_j).
\end{align}
In addition to the usual diagonal operator $D_0$, we likewise define shifted diagonal operators
\begin{align}
D_j=(1-z_j\DD^*)^{-1}(\DD-z_j) \; \mbox{ and } \; V_j=\sqrt{(1-\overline{z_j}\DD)^{-1}(1-z_j\DD^*)}.
\end{align}
With this notation, we can explicitly describe how the Blaschke factors associated to the generating vector $\vz$ ``shift'' MCMV matrices:
\begin{proposition}\label{lem:BlaschkeA}
Let $D_j$ and $V_j$ be defined as above.  Then for any $C\in\overline{\D_{\ell^2}}$, we have
\begin{align}
b_{z_j}\bigl(b_{\minus \DD}(C)\bigr)=V_jb_{\minus D_j}(C)V_j.
\end{align}
\end{proposition}

\begin{proof}
Let
\begin{align*}
U_1=\begin{bmatrix}
1    & -\overline{z_j} \\
-z_j & 1
\end{bmatrix} \; \mbox{ and } \;
U_2=
\begin{bmatrix}
1 & \DD^* \\
\DD & 1
\end{bmatrix}
\begin{bmatrix}
\eta_{\DD}^{-1} & 0           \\
0           & \eta_{\DD}^{-1}
\end{bmatrix}.
\end{align*}
Then we have
\begin{align*}
b_{z_j}\bigl(b_{\minus \DD}(C)\bigr)=\Phi_{U_1}(\Phi_{U_2}(C))=\Phi_{U_2U_1}(C).
\end{align*}
Recalling that $\eta_j=\sqrt{1-|z_j|^2}$, it is straightforward to see that
\begin{multline*}
\qquad \qquad \eta_j^{-1}U_2U_1=\begin{bmatrix}
1                               & (\DD^*-\overline{z_j})(1-\overline{z_j}\DD)^{-1} \\
(\DD-z_j)(1-z_j\DD^*)^{-1} & 1
\end{bmatrix} \\
\times \begin{bmatrix}
\eta_j^{-1}\eta_{\DD}^{-1}(1-z_j\DD^*) & 0                                          \\
0                                & \eta_j^{-1}\eta_{\DD}^{-1}(1-\overline{z_j}\DD)
\end{bmatrix}. \qquad\qquad
\end{multline*}
Due to \eqref{eq:etaRelation}, we have
\begin{align*}
(1-\overline{z_j}\DD)(1-z_j\DD^*)\eta_{D_j}^2=\eta_j^2\eta_{\DD}^2
\end{align*}
and hence
\begin{align*}
\eta_jU_2U_1=\begin{bmatrix}
1   & D_j^* \\
D_j & 1
\end{bmatrix}
\begin{bmatrix}
\eta_{D_j}^{-1} & 0               \\
0               & \eta_{D_j}^{-1}
\end{bmatrix}
\begin{bmatrix}
V_j & 0        \\
0   & V_j^{-1}
\end{bmatrix}.
\end{align*}
This concludes the proof.
\end{proof}

\begin{remark}
So far we haven't used that $z_j \in \vz$; however, this assumption is important to retain the banded structure of an MCMV matrix.  In particular, since $b_{z_j}(z_j)=0$, applying $b_{z_j}$ to $A$ shifts the zeros in $D_0$ by $2j$.  It follows that $S^{-2j}b_{z_j}(A)S^{2j}$ is again MCMV-structured (with a new generating vector).

\end{remark}

We have analyzed the block structure and Blaschke shifts of MCMV matrices in order to understand the Magic Formula in the context of the representation \eqref{eq:DeltaERational}.  Critical to this understanding is computing off-diagonal blocks of selfadjoint operators of the form $\Re(c_j b_{z_j}(A))$ (cf. \eqref{eq:MagicFormBlock}); this is essentially the content of the final theorem of this section:


\begin{theorem}\label{t:MCMVstructure}
Let $\vz$ be such that $z_j\neq z_k$ for $j\neq k$ and $A \in \A(\vz)$. Then
\begin{multline}\label{eq:diagonal1}
\bigl(b_{z_k}(A)\bigr)_{2(n+k),2k}=e^{i\vartheta}\tr\bigg(\begin{bmatrix}
1& 0\\
0& 0
\end{bmatrix}
U(a_{2k})
\begin{bmatrix}
b_{z_{k+1}}(z_k^*)& 0\\
0& 1
\end{bmatrix}
\cU(a_{2k+1})
\cdots \\
\cdots \begin{bmatrix}
1& 0\\
0& b_{z_{k-1}}(z_k^*)^{-1}
\end{bmatrix}
\cU(a_{2(n+k)-2})
\begin{bmatrix}
1& 0\\
0& 0
\end{bmatrix}
U(a_{2(n+k)-1})
\bigg)^{-1}.
\end{multline}
In particular,
\begin{align}\label{eq:diagonal2}
A_{2n,0}=e^{i\vartheta}\tr\bigg(\begin{bmatrix}
1& 0\\
0& 0
\end{bmatrix}
U(a_0)
\begin{bmatrix}
b_{z_1}(\i)& 0\\
0& 1
\end{bmatrix}
\cU(a_1)
\cdots
\begin{bmatrix}
1& 0\\
0& b_{z_{n-1}}(\i)^{-1}
\end{bmatrix}
\cU(a_{n-2})
\begin{bmatrix}
1& 0\\
0& 0
\end{bmatrix}
U(a_{n-1})
\bigg)^{-1}.
\end{align}
\end{theorem}
\begin{proof}
Equation \eqref{eq:diagonal1} follows from shifting \eqref{eq:diagonal2} by Lemma \ref{lem:BlaschkeA}; thus, we only need to prove \eqref{eq:diagonal2}.  Note that
\begin{align*}
\tilde A_{2n,0}&=\tr\bigg(
\frac{1}{\rho_{2n-1}}
\begin{bmatrix}
1& 0\\
0& 0
\end{bmatrix}
\cU(-\overline{a_{2n-2}})
\begin{bmatrix}
1&0\\
0& -\overline{z_{n-1}}
\end{bmatrix}
\cdots  \\
& \qquad\qquad \cdots
\begin{bmatrix}
1& 0\\
0& -\overline{z_1}
\end{bmatrix}
\cU(-\overline{a_1})
\begin{bmatrix}
-\frac{1}{\overline{z_1}}& 0\\
0& 1
\end{bmatrix}
\cU(-\overline{a_0})
\begin{bmatrix}
1& 0\\
0& 0
\end{bmatrix}
\bigg)^{-1}\\
&=
\tr\bigg(
U(-\overline{a_{2n-1}})
\begin{bmatrix}
1& 0\\
0& 0
\end{bmatrix}
\cU(-\overline{a_{2n-2}})
\begin{bmatrix}
1&0\\
0& -\overline{z_{n-1}}
\end{bmatrix}
\cdots \\
& \qquad\qquad \cdots
\begin{bmatrix}
1& 0\\
0& -\overline{z_1}
\end{bmatrix}
\cU(-\overline{a_1})
\begin{bmatrix}
-\frac{1}{\overline{z_1}}& 0\\
0& 1
\end{bmatrix}
\cU(-\overline{a_0})
\begin{bmatrix}
1& 0\\
0& 0
\end{bmatrix}
\bigg)^{-1}.
\end{align*}
Since $\tr(N)=\tr(\mj N^\intercal \mj)$, we obtain \eqref{eq:diagonal2} by inserting $\mj^2$ between all factors.
\end{proof}

For the final result of this section, the adaptation of the proof for repeated points in the vector $\vz$ is not straightforward, so we will drop our assumption of distinct $z_j$'s at this place.  By reordering the entries of $\vz$, if needed, we can assume that the entries of $\vz$ with higher multiplicity are ordered consecutively; if $z_j=z_{j+1}$ we will denote both by $z_j$.  For a given vector $\vz$, suppose there are $m \leq n$ distinct entries $z_j$, and let $m_j$ denote the multiplicity with which the point $z_j$ appears in $\vz$, such that $m_0+\dots+m_{m-1}=n$. We call a rational function $\Delta$ \emph{suitable} for $\vz$ if it is of the form
\[
\Delta(z)=c+\sum_{j=0}^{m-1}\sum_{i=1}^{m_j}\bigl(c_{ij}b_{z_j}(z)^i+\overline{c_{ij}}b_{z_j}(z)^{-i}\bigr).
\]
The following result will be used in the proof of Theorem \ref{t:magicformula} and relies fundamentally on the structure of MCMV matrices obtained in Theorem \ref{t:MCMVstructure}.
\begin{proposition}\label{prop:DiscrCoincide}
Let $A\in\A(\vz)$ and suppose $\Delta$ is a rational functions which is suitable for $\vz$. Then $\Delta(A)=0$ implies that $\Delta(z)=0$.
\end{proposition}
\begin{proof}
Let us assume for the sake of simplicity that $e^{i\vartheta}=1$ (the general case is analogous). The key of the proof will be to understand the structure of the powers $A^i$ for $1\leq i\leq m_0$. Due to Proposition \ref{lem:BlaschkeA}, the structure of the powers $b_{z_j}(A)^i$ will then follow by shifting.  Recall that $m_0$ denotes the multiplicity with which $z_0 = 0$ is represented in the vector $\vz$.  We analyze the structure of the $2n \times 2n$ block of $A$ formed by the entries $\{A_{ij}\}_{i, j=0}^{2n-1}$.

Since we have seen that the block structure is obtained by repeated $z_0$ entries, our $2n \times 2n$ block splits up into $m_0-1$ diagonal blocks of size $2\times 2$ and a possibly bigger diagonal block of size $2(n-m_0+1)\times 2(n-m_0+1)$.  Let us denote the $2\times 2$ block matrices by $\mathbf{A_0^1}, \dots, \mathbf{A_0^{m_0-1}}$ and the larger block by $\mathbf A$. On each side of the diagonal blocks $\mathbf{A_0^l}$, there is a column vector of size $2$; we write $\mathbf{v_0^l}$ for the vector to the left of $\mathbf{A_0^l}$ and $\mathbf{u_0^l}$ for the vector to the right.  Similarly, on each side of $\mathbf A$ there is a column vector of size $2(n-m_0+1)$, which we will denote analogously by $\mathbf v$ and $\mathbf u$. In particular, for $0 \leq l \leq m_0-2$ we find due to \eqref{eq:diagonal1} and \eqref{eq:diagonal2} that
$A_{2l,2l+2}=\rho_{2l}\rho_{2l+1}>0$. Moreover, the first entry of $\mathbf u$ is the nonzero value $A_{2(m_0-1),2n} = \rho \neq 0$; this entry is the only non-vanishing entry of $A$ on the $2(n-m_0)+2$th diagonal.  If we consider $A^2$, we get exactly two non-vanishing entries on the $2(n-m_0)+4$th diagonal, given by $A_{2(m_0-2),2n}=\rho_{2(m_0-2)}\rho_{2(m_0-2)+1}\rho$ and $A_{2(m_0-1),2(n+1)}=\rho\rho_{2n}\rho_{2n+1}$.  Similarly, $A^i$ will have $i$ non-vanishing entries on the outermost $2(n-m_0)+2i$th diagonal; in particular, $A^{m_0}$ will have $m_0$ non-vanishing entries on the $2n$th diagonal given by
\begin{align*}
A^{m_0}_{0,2n}&=\rho_0\rho_1\cdots\rho_{2(m_0-2)}\rho_{2(m_0-2)+1}\rho,\\
A^{m_0}_{2j,2(n+j)}&=\rho_{2j}\rho_{2j+1}\cdots\rho_{2(m_0-2)}\rho_{2(m_0-2)+1}\rho\rho_{2n}\rho_{2n+1}\cdots\rho_{2(n+j)}\rho_{2(n+j)+1}, \\
A^{m_0}_{2(m_0-1),2(n+m_0-1)}&=\rho\rho_{2n}\rho_{2n+1}\cdots\rho_{2(n+m_0-2)}\rho_{2(n+m_0-2)+1}
\end{align*}
for $1\leq j < m_0-1$.  A similar structure, but shifted, is obtained for all the matrices $b_{z_j}(A)^i$.

With this structure in mind we can finish the proof. We first consider the entries $\Delta(A)_{j,2n+j}$. On this diagonal, only the operators $b_{z_j}(A)^{\pm m_j}$ have non-vanishing entries (and we can guarantee that they are non-vanishing). But all of them are at different positions. Hence, we see that $c_{m_j,j}=0$ for all $0\leq j \leq m$. In the next step, we consider the diagonal $\Delta(A)_{j,2n-2+j}$ and obtain analogously that $c_{m_j-1,j}=0$. Inductively we see that all coefficients vanish, and consequently $\Delta\equiv0$.
\end{proof} 

%% file: proofsMain.tex
\section{Proofs of the Main Theorems}\label{sec:proofs}
We have laid nearly all the groundwork necessary to complete the proofs of our main theorems.  Before we proceed, we recall the general strategy: in Section \ref{sec:3}, we showed that, in the presence of a function $\mathfrak{B} = \mathfrak{B}_{\vz}$ having half-period character, there is a map $\mathbf{F} : \Gamma^* \times \T \to \mathbb{A}_\per(\vz)$ taking unimodular characters and a phase $(\alpha,\tau)$ to a periodic MCMV matrix $\mathbf{F}(\alpha,\tau) := A(\alpha,\tau)$ satisfying a Magic Formula.  In Section \ref{sec:DST}, we defined a map $\mathbf G$ assigning a divisor to a periodic MCMV matrix.  In this section we glue these constructions together via a third and final map, the Abel map $\mathfrak{A} : \cD(\E) \to \Gamma^* \times \T$, and show that together they in fact form a commuting diagram in analog to \eqref{fig:1}:
\begin{equation}\label{eq:commDiag}
\begin{aligned}
\xymatrix{
\A_\per(\vz)\ar[dr]_{\mathbf{G}}
&&\ar[ll]_{\mathbf F}\G^*\times \T\\
&\cD(\E)\ar[ur]_{\fA}&
}
\end{aligned}
\end{equation}
Once this is done, our main theorems will follow as special cases for a particular choice of function $\mathfrak{B}$ and vector $\vz_\E$.

To properly introduce the Abel map, we briefly recall the construction from \cite{YudPeher06} of the bijective correspondences between the isospectral torus $\cT_\mv(\E)$, the set of Schur functions $\cS_+(\E)$ defined in \eqref{eq:Schurplus}, and the set of divisors $\cD(\E)$ given by \eqref{def:Divisors}.


We first set up the correspondence $\cS_+(E)\simeq\cD(E)$.  Consider a function $f_+\in\cS_+(\E)$ and let
\begin{align*}
F_+(z):=\frac{1+zf_+(z)}{1-zf_+(z)}
\end{align*}
denote the associated Caratheodory function. Due to \eqref{eq:RHprob1}, we have
\begin{align*}
\frac{1+\overline{e^{it}f_+(e^{it})}}{1-\overline{e^{it}f_+(e^{it})}}=\frac{1+f_-(e^{it})}{1-f_-(e^{it})}
\; \mbox{ for a.e. $e^{it}\in\E$}.
\end{align*}
Strictly speaking,
\begin{align*}
F_-(z):=\frac{1+f_-(z)}{1-f_-(z)}
\end{align*}
is not a Caratheodory function since it does not admit the normalization $F_-(0)=1$; however, it still maps $\D$ analytically into the right half-plane. Using \eqref{eq:RHprob2}, we see directly that the function
$$
F(z):= \frac{1}{2}\Bigl(F_+(z)+F_-(z)\Bigr)=\frac{1-zf_+(z)f_-(z)}{\bigl(1-zf_+(z)\bigr)\bigl(1-f_-(z)\bigr)}
$$
has no zeros in the gaps of $\E$. Hence, since $t\mapsto \Im F(e^{it})$ is decreasing on each gap, this function can have at most one sign change per gap, caused by a possible pole $x_j$ (where $j$ indexes the $j$-th gap).  The measure $\nu$ in the integral representation of $F$ is purely absolutely continuous on $\E$, and condition \eqref{eq:RHprob1} implies that $\nu_{ac}$ is split equally between $F_+$ and $F_-$.  But due to \eqref{eq:RHprob2}, the point mass at $x_j$ can only correspond to either $F_+$ or $F_-$ and not to both.  We write $(x_j,1)$ if $x_j$ is a pole of $F_+$ and $(x_j,-1)$ if $x_j$ is a pole of $F_-$.  Special consideration is needed for the endpoints of the gaps: by convention, we write $(\l_j^-,1)$ if $\Im F\leq 0$ in the closed gap $[\l_j^-,\l_j^+]$ and $(\l_j^+,1)$ if $\Im F\geq 0$. With these choices, the collection
$$D=\{(x_j,\e_j)\}_{j=0}^g, \quad \e_j\in \{\pm 1\}$$
is the divisor in $\cD(\E)$ associated to $f_+$.  Conversely, one can show that any divisor $D\in\cD(\E)$ leads to a function $f_+\in\cS_+(\E)$ (see \cite[Theorem 1.4]{YudPeher06} for details).

The correspondence $\cT_\mv(\E)\simeq\cS_+(E)$ is implicitly given in Section \ref{sec:PY}.  The half-line restriction $\apC_+$ of an element $\apC=\apC(\alpha, \tau)\in\cT_\mv(\E)$ is linked to the Schur function $f_+^{\a,\tau}$ given by
\begin{align} \label{eq:CaraReprod}
f_+^{\a,\tau}\circ\fz=e^{-i\tau}\frac{K^{\a}_{\overline{\z}_0}}{K^{\a}_{\z_0}}
\end{align}
through the relation
\begin{align*}
\frac{1+zf^{\a,\tau}_+(z)}{1-zf^{\a,\tau}_+(z)}=\Bigl\langle \bigl(\apC_+(\a,\tau)-z\bigr)^{-1}\bigl(\apC_+(\a,\tau)+z\bigr)\d_0,\d_0\Bigr\rangle.
\end{align*}
In fact, the map $(\a,\tau)\mapsto f^{\a,\tau}_+$ sets up a bijection between $\G^*\times\T$ and $\cS_+(\E)$. This also enables us to define the Abel map $\fA: \cD(\E)\to\G^*\times\T$ by
\begin{align}\label{eq:AbelMap}
\fA(D):=(\alpha,\tau),
\end{align}
where $(\alpha,\tau)$ is the character and phase of the function $f_+$ which corresponds to the divisor $D$.  Note that \eqref{eq:AbelMap} generalizes the definition of the Abel map for periodic CMV matrices; cf. \eqref{fig:1}.
%

The first lemma of this section demonstrates that the diagram \eqref{eq:commDiag} commutes if we replace $\A_\per(\vz)$ with the image $\mathbf{F}(\Gamma^* \times \T)$. 
Let $\tilde D$ be the divisor associated to the periodic operator $A(\a,\tau)\in\A_\per(\vz)$ by our construction in Section \ref{sec:DirectMCMV}, that is, $\tilde D=\mathbf G(\mathbf F(\a,\tau))$. 
\begin{lemma}\label{lem:DivisorCoinc}
Let $\tilde f_+$ and $\tilde F_+$ be the Schur and Caratheodory functions associated to the periodic operator $A(\a,\tau)$ as in Section \ref{sec:DST}, and let $f_+^{\a,\tau}$ and $F_+^{\a,\tau}$ be the functions associated to $(\a,\tau)$ by \eqref{eq:CaraReprod} above. Then we have
\begin{align}
\tilde f_+ = f_+^{\a,\tau}, \; \; \tilde F_+=F_+^{\a,\tau}, \; \mbox{ and } \; \mathfrak{A}(\tilde{D}) = (\alpha,\tau).
\end{align}
\end{lemma}
\begin{proof}
By \eqref{def:Basis} and \eqref{def:Basis2}, we have
\begin{align*}
f_+^{\a,\tau}\circ\fz=\frac{x_0^{\a,\tau}}{y_{0}^{\a,\tau}}.
\end{align*}
Therefore, the identities $\tilde f_+ = f_+^{\a,\tau}$ and $\tilde F_+=F_+^{\a,\tau}$ follow by comparing \eqref{eq:quadraticCara} with Theorem \ref{thm:monodromy}.

We've now seen that $(\tilde x_j,1)$ corresponds to poles of $\tilde F_+$.  To see $\mathfrak{A}(\tilde{D}) = (\alpha,\tau)$, we only need to show that $(\tilde x_j,-1)$ as defined in Section \ref{sec:DST} corresponds to poles of the function $\tilde F_-$. Due to \eqref{eq:caratheordory}, we have
\begin{align}\label{eq:Fplus}
\tilde F_+(z)=\frac{v(z)+\sqrt{\Delta_A(z)^2-4}}{u(z)},
\end{align}
where
\begin{align*}
v(e^{it})\in\R\quad\text{and}\quad u(e^{it})\in i\R.
\end{align*}
Consider now the function
\begin{align}\label{eq:Fminus}
\tilde{F}_-(z) := -\frac{v(z)-\sqrt{\Delta_A(z)^2-4}}{u(z)}.
\end{align}
Since $\sqrt{\Delta_A(e^{it})^2-4}\in i\R$ on $\E$, we obtain that $\overline{\tilde F_+(e^{it})}=\tilde F_-(e^{it})$ for all $e^{it} \in \E$.
Hence, if $\tilde x_j$ is a zero of $u$ and the numerator in \eqref{eq:Fplus} vanishes (i.e., $\e_j=-1$), then $\tilde F_-$ has a pole at $\tilde x_j$. This concludes the proof.
\end{proof}
\begin{remark}
Note that \eqref{eq:Fplus} and \eqref{eq:Fminus} show that the absolutely continuous parts of the corresponding measures agree and are given by \eqref{eq:acMeasure}.
\end{remark}

In order to show commutativity of \eqref{eq:commDiag}, we still have to show that for a given $\vz$, there is a functional model construction $\mathbf{F}$ which surjects onto those MCMV matrices in $\A_\per(\vz)$ having a fixed spectrum; that is, for an arbitrary $\vz$ and $A\in\A_\per(\vz)$ with spectrum $\sigma(A) = \E$, there exists a function $\mathfrak B=\prod_{j=0}^{n-1}\fb_{\z_j}$ whose character is a half-period such that $A=A(\a,\tau)$ corresponding to the functional model associated to $\mathfrak B$.
To show that the character of $\mathfrak B$ is a half-period, the representation of $\Delta_A$ from Lemma \ref{lem:repDeltaT} will be crucial.
\begin{proposition}\label{prop:DS1}
Let $A\in\A_{\per}(\vz)$ be a periodic MCMV matrix with associated discriminant $\Delta_A$, and set $\E=\Delta_A^{-1}([-2,2])$.  Then there exists a unique divisor $D\in\cD(\E)$ such that
\begin{align}\label{eq:periodicFuncModel}
A=A(\a,\tau),
\end{align}
where $\fA(D)=(\a,\tau)$ and $A(\a,\tau)$ is defined by \eqref{def:Aalpha} for the functional model associated to $\mathfrak B=\prod_{j=0}^{n-1}\fb_{\z_j}$. Moreover, the diagram \eqref{eq:commDiag} commutes.
\end{proposition}

\begin{proof}
Suppose $A$ is periodic and let $D$ be the associated divisor. Write $(\a,\tau)=\fA(D)$ for the corresponding character and phase. Due to Lemma \ref{lem:repDeltaT}, we have
\begin{align*}
\Delta_A\circ \fz=\Psi+\frac{1}{\Psi}, \quad \Psi=\mathfrak B\mathfrak B^*.
\end{align*}
Since the characters of $\mathfrak B$ and $\mathfrak B^*$ coincide and $\Psi$ is single-valued, the character of $\mathfrak B$ must be a half-period.

Suppose now that $A(\a,\tau)$ is the matrix representing multiplication by $\fz$ in the basis $\{\chi_k^{\a,\tau}\}$ for the functional model associated to $\mathfrak B$. Then, by Lemma \ref{lem:DivisorCoinc}, the Caratheodory functions of $A$ and $A(\alpha,\tau)$ coincide and hence we obtain \eqref{eq:periodicFuncModel}. Since it was already shown in \cite{YudPeher06} that the Abel map is a bijection, this also proves the commutativity of \eqref{eq:commDiag}.
\end{proof}

We are now in position to complete the proofs of our main results stated in the introduction.  As already mentioned in the first remark of Section \ref{sec:modifiedBasis}, the existence of the Ahlfors function for an arbitrary finite-gap set $\E$ ensures that there always is a function $\fB$ with the property that its character is a half-period.  To be specific, the function $\fw_\infty = \fz (w_\i\circ \fz)$ is an explicit choice of such a function.  Our first main result, Theorem \ref{t:periodiccoords}, now simply follows as the special case of Proposition \ref{prop:DS1} with $\vz = \vz_\E$ (and $\mathfrak B=\fw_\infty$).


\begin{proof}[Proof of Theorem \ref{t:periodiccoords}]
For the special choice of $\vz=\vz_\E$, we define the map $\mathbf F:\Gamma^*\times \T\to \cT_\MV(\E)$ in the same way as was done previously. That is, $$\mathbf F(\a,\tau):=A(\a,\tau), $$
where $A(\a,\tau)$ is the matrix representation of multiplication by $\fz$ in the basis $\{y_k^{\a,\tau}\}$ for the functional model associated to the function $\fB=\fw_\infty$.  Our considerations have shown this $\mathbf F$ is a bijection and that concludes the proof of \eqref{eq:bij}.

Tracing through our construction will also lead to the more explicit version of the one-to-one correspondence between an element $\apC\in\cT_\MV(\E)$ and its counterpart $C\in \cT_\mv(\E)$.
\end{proof}

Our analysis also leads to a quick proof of the Magic Formula, our second main result. 

\begin{proof}[Proof of Theorem \ref{t:magicformula}]
Let $A$ be an element of $\A(\vz_\E)$.  If $A\in \cT_{\MV}(\E)$, then we know from Proposition \ref{prop:DS1} that $A=A(\a,\tau)$ for some choice of character and phase.  Hence the Magic Formula
\[ 
 \Delta_\E(A)=S^{2(g+1)}+S^{-2(g+1)}
\]  
follows by Corollary \ref{cor:MagicFormulafunc} with $n=g+1$.  

Conversely, if $A$ satisfies the Magic Formula then $A$ must be periodic with period $2(g+1)$ due to Naiman's Lemma (see, e.g., \cite[Lemma 8.2.4]{SimonSzego}).  Moreover, as 
\[
 \Delta_\E(A)=S^{2(g+1)}+S^{-2(g+1)}=\Delta_A(A),
\] 
Proposition \ref{prop:DiscrCoincide} implies that $\Delta_A\equiv\Delta_\E$.  Since $\sigma(A)=\Delta_A^{-1}([-2,2])$, it therefore follows that $\sigma(A)=\E$ and thus $A\in\cT_{\MV}(\E)$.
The second part of the theorem follows along the same lines as in the proof of Theorem \ref{t:periodiccoords}.
\end{proof}

Finally, Theorem \ref{t:nevPick} is immediate:

\begin{proof}[Proof of Theorem \ref{t:nevPick}]
This is a straightforward consequence of Theorems \ref{t:periodiccoords}, \ref{t:magicformula}, and \ref{thm:monodromy}.
\end{proof}

%% file: appendix.tex
\section{The Ahlfors functions of finitely connected Denjoy domains}\label{sec:AppAhlfors}
We begin with an existence theorem:
\begin{theorem}
Let $\Omega$ be a region in $\overline{\C}$ and fix a point $z_0\in\Omega$. Suppose there exist nonconstant bounded analytic functions defined on $\Omega$. Then there is a unique analytic function $w_{z_0}:\Omega\to\D$ which solves the Ahlfors problem, that is,
\begin{align}
w'_{z_0}(z_0)=\sup\{|g'(z_0)|:\, g:\Omega\to\D \text{ analytic}, \, g(z_0)=0\}.
\end{align}
\end{theorem}
The function $w_{z_0}$ is called the \emph{Ahlfors function} of $\Omega$ (and $z_0$) and we always have that $w_{z_0}(z_0)=0$.  These functions were first studied for finitely connected domains by Ahlfors \cite{Ah47}, hence the name. Existence and uniqueness for arbitrary domains was later established by Fisher \cite{Fish69}, see also \cite[Section 8]{Sim15}.

In \cite{EYud17}, an explicit expression for the Ahlfors function of finitely connected domains with certain symmetries was presented. Let $0<a_j<b_j$ for $1\leq j\leq g$ and define
\begin{align}\label{eq:Ereal}
	\E_\R:=\R_+\setminus\bigcup_{j=1}^g(a_j,b_j), \quad
	\Omega_\R:=\C\setminus \E_\R.
\end{align}
Since the Ahlfors problem is conformally invariant, providing a solution for $\Omega_\R$ also leads to a solution for all conformally equivalent domains. In particular, for any finite systems of arcs $\E_\T$ on the unit circle, we can map $\D$ conformally onto $\C_+$ such that $\E_\T$ corresponds to some set $\E_\R$ of the form \eqref{eq:Ereal}. For notational simplicity, we merely present the results for sets of the form \eqref{eq:Ereal}. To any such set, we associate the function
\begin{align}
	H(z)=\frac{1}{\sqrt{-z}}\prod_{j=1}^{g}\sqrt{\frac{z-a_j}{z-b_j}},
\end{align}
where the square root is chosen in such a way that $H(z)$ becomes a Nevanlinna function (i.e., maps $\C_+$ analytically into $\C_+$). The following theorem generalizes a result of Pommerenke \cite{Pom1} (who dealt with the case of $z_0\in\R\setminus E_\R$).
\begin{theorem}{\cite[Theorem 2.3]{EYud17}}
The Ahlfors function of $\Omega_\R$ and $z_0$ is given explicitly by
\begin{align}\label{eq:propertyDiscrApp}
w_{z_0}(z)=\frac{z-z_0}{z-\overline{z_0}}\frac{H(z)-H(\overline{z_0})}{H(z)+H(z_0)}.
\end{align}
If $\Im z_0>0$, then $w_{z_0}$ has precisely $g$ zeros, say $\overline{z_1},\dots,\overline{z_g}$, in the lower half-plane $\C_-$ and together with $z_0$, these points account for all the zeros of $w_{z_0}$. Moreover,
\begin{align}
	\log|w_{z_0}(z)|=-G_{\Omega_\R}(z,z_0)-\sum_{j=1}^{g}G_{\Omega_\R}(z,\overline{z_j}).
\end{align}
\end{theorem}

Now, fix $z_0\in\C_+$ and consider the discriminant $\Delta_{\E_\R}$ defined by
\begin{align}\label{def:DeltaReal}
\Delta_{\E_\R}(z):=w_{z_0}(z)w_{\overline{z_0}}(z)+\frac{1}{w_{z_0}(z)w_{\overline{z_0}}(z)}.
\end{align}
We collect the properties of $\Delta_{\E_\R}$ in the following theorem and point out that the conclusion for the critical points appears to be new.
\begin{theorem}
$\Delta_{\E_\R}$ is a real rational function, that is, $\Delta_{\E_\R}(z)=\overline{\Delta_{\E_\R}(\overline{z})}$. Its poles are given by $z_0,\overline{z_0},\dots,z_g,\overline{z_g}$, where $\overline{z_1},\dots,\overline{z_g}$ are the zeros of $w_{z_0}$ in $\C_-$. Moreover,
\begin{align} \label{eq:preDelta}
\E_\R=\Delta_{\E_\R}^{-1}\bigl([-2,2]\bigr)
\end{align}
and $\Delta_{\E_\R}$ has exactly one critical point, say $c_{2j}$, in the $j$-th gap $(a_j,b_j)$ with $\Delta_{\E_\R}(c_{2j})>2$ and exactly one critical point, say $c_{2j+1}$, in the $j$-th band $(b_j,a_{j+1})$ with $\Delta_{\E_\R}(c_{2j+1})=-2$.
\end{theorem}
\begin{proof}
Since $\overline{H(\overline{z})}=H(z)$, we see that $\overline{w_{z_0}(z)}=w_{\overline{z_0}}(\overline{z})$. Also, for $z\in\E_\R$, one has
\begin{align*}
\lim\limits_{\e\to 0}\Delta_{\E_\R}(z+i\e)=\lim\limits_{\e\to 0}\Delta_{\E_\R}(z-i\e).
\end{align*}
Hence $\Delta_{\E_\R}$ is a real rational function on $\overline{\C}$. Note that in \eqref{def:DeltaReal} we have written $\Delta_{\E_\R}$ as the composition of $w_{z_0} w_{\overline{z_0}}$ and the Joukowsky map $u\mapsto u+{1}/{u}$. Since $|w_{z_0}(z)|=1$ if and only if $z\in\E_\R$, this proves \eqref{eq:preDelta}.  Moreover, since $H$ is a Nevanlinna function, it decreases monotonically from $i\i$ to $0$ as $z$ moves along a band $[b_j,a_{j+1}]$. Setting $H(z_0)=H_0$, it follows from the previous theorem that
\begin{align*}
w_{z_0}(z)w_{\overline{z_0}}(z)=\frac{H(z)-H_0}{H(z)+\overline{H_0}}\frac{H(z)-\overline{H_0}}{H(z)+H_0}.
\end{align*}
Recall now that
$
\frac{z-H_0}{z+\overline{H_0}}
$
is the Blaschke factor of the right half-plane. Therefore, as $H(z)$ decreases from $i\i$ to $0$, the values of $w_{z_0}w_{\overline{z_0}}$ runs through $\partial\D$ precisely once, starting and ending at $1$. Hence there is exactly one point $c_{2j+1}\in(b_j,a_{j+1})$ with $w_{z_0}(c_{2j+1})w_{\overline{z_0}}(c_{2j+1})=-1$, that is, $\Delta_{\E_\R}(c_{2j+1})=-2$.
To analyze the behavior in the gaps, we note that in each gap $H(z)$ increases monotonically from $0$ to $\infty$. Considering the function
\begin{align*}
\psi(x)=\frac{x-H_0}{x+\overline{H_0}}\frac{x-\overline{H_0}}{x+H_0}
\end{align*}
on $\R_+$ shows that there is exactly one critical point in each gap.
\end{proof}

\section{Operator M\"obius transforms}\label{sec:operatorMoebius}
Let $\cH$ be a Hilbert space and denote by $\mathcal{L}(\cH)$ the space of bounded linear operators from $\cH$ into itself equipped with the standard operator norm. By $\D(\cH)$ (resp. $\overline{\D(\cH)}$)  we denote the open (resp. closed) unit ball in $\mathcal{L}(\cH)$, that is, $\D(\cH)$ is the set of contractions on $\cH$. We will write operators $U\in \mathcal{L}(\cH\oplus\cH)$ in matrix form:
\begin{align}
	U=\begin{bmatrix}
	U_{11}& U_{12}\\
	U_{21}& U_{22}
	\end{bmatrix}, \quad U_{ij}\in \mathcal{L}(\cH).
\end{align}
To such matrices we can associate a linear fractional transformation
\begin{align}
\Phi_U(S):=(SU_{12}+U_{22})^{-1}(SU_{11}+U_{21}),
\end{align}
defined for those $S\in \mathcal{L}(\cH)$ for which $SU_{12}+U_{22}$ is boundedly invertible. It is straightforward to see that $\Phi_{U}(\Phi_{V}(S))=\Phi_{VU}(S)$ and that for any $\l\neq 0$, the operators $U$ and $\l U$ generate the same transform.

In \cite{Krein67}, a complete characterization of the class of operators $U$ such that $\Phi_{U}$ is a bijection from $\D(\cH)$ onto $\D(\cH)$ was given. In this case, $\Phi_U$ is called an \textit{operator M\"obius transform}. The characterization involves the special operator
\begin{align*}
	\mjH=\begin{bmatrix}
	1&0\\
	0& -1
	\end{bmatrix},
\end{align*}
where as usual $1$ denotes the identity operator on $\cH$. An operator $U\in \mathcal{L}(\cH\oplus\cH)$ is called \textit{$\mjH$-unitary} if $U^*\mjH U=\mjH$.
\begin{theorem}{\cite{Krein67}} 
	 $\Phi_A$ is an operator M\"obius transform if and only if $A$ is colinear with a $\mjH$-unitary operator.
\end{theorem}
A complete description of all $\mjH$-unitary operators and hence all operator M\"obius transforms is known:
\begin{theorem}{\cite{Krein67}} 
	The general form of a $\mjH$-unitary operator is
	\begin{align}
		U=\begin{bmatrix}
		1& A^*\\
		A& 1
		\end{bmatrix}
		\begin{bmatrix}
		\eta_{A^*}^{-1}& 0\\
		0&\eta_{A}
		\end{bmatrix}
		\begin{bmatrix}
		V_1&0\\
		0& V_2
		\end{bmatrix},
	\end{align}
	where $A\in \D(\cH)$, $\eta_A=\sqrt{1-AA^*}$, and $V_1,V_2$ are unitary operators on $\cH$.
\end{theorem}